\documentclass{mcom-l}

\usepackage{amssymb}

\usepackage{graphicx}

\usepackage{tikz}
\usepackage{pgfplots}
\pgfplotsset{compat=1.3}
\usetikzlibrary{calc}

\usepackage{subcaption} 
\usepackage{mathabx}

\newtheorem{theorem}{Theorem}[section]
\newtheorem{lemma}[theorem]{Lemma}

\newtheorem{proposition}[theorem]{Proposition}

\theoremstyle{definition}
\newtheorem{definition}[theorem]{Definition}
\newtheorem{example}[theorem]{Example}

\theoremstyle{remark}
\newtheorem{remark}[theorem]{Remark}

\numberwithin{equation}{section}

\def\d{\partial}
\newcommand{\X}[1]{X^{({#1})}}
\newcommand{\Xh}[1]{X_h^{({#1})}}
\newcommand{\Xhl}[1]{X_{h, [\ell]}^{({#1})}}
\newcommand{\Cstab}{C_{\mathrm{sta}}}
\newcommand{\Csta}{\Cstab}
\newcommand{\kk}{{{[\ell]}}}
\newcommand{\kkk}{{{{[\ell+1]}}}}
 
\newcommand{\Rt}{\tilde{R}}
\newcommand{\Et}{\tilde{E}}
\newcommand{\At}{\tilde{A}}

\newcommand{\TG}{{\mathtt{T2G}}}
\newcommand{\Rsem}{R^{\mathrm{sem}}}
\newcommand{\Rexp}{R^{\mathrm{exp}}}
\newcommand{\Rsat}{R^{\mathrm{sat}}}

\newcommand{\Rimp}{R^{\mathrm{imp}}}
\newcommand{\Rhimp}{R_{h, \mathrm{imp}}}

\newcommand{\RRhimp}{\RR_h^{\mathrm{imp}}}
\newcommand{\RRexp}{\RR^{\mathrm{exp}}}
\newcommand{\RRsat}{\RR^{\mathrm{sat}}}
\newcommand{\Ehimp}{E_{h, \mathrm{imp}}}
\newcommand{\usat}{u^{\mathrm{sat}}}
\newcommand{\uexp}{u^{\mathrm{exp}}}
\newcommand{\uhatexp}{\hat{u}^{\mathrm{exp}}}
\newcommand{\uhimp}{u_{h1}^{\mathrm{imp}}}
\newcommand{\uhatimp}{\hat{u}^{\mathrm{imp}}}
\newcommand{\Lv}{L^{\vtx}}
\newcommand{\Hv}{H^{\vtx}}
\newcommand{\Hhv}{H_h^{\vtx}}
\newcommand{\Hvh}{\Hhv}
\newcommand{\diam}{\mathop{\text{diam}}}
\newcommand{\cmax}{\hat{c}}
\newcommand{\hv}{h_{\vtx}}
\newcommand{\vtx}{{\mathtt{v}}}
\newcommand{\vPhi}{\varPhi}
\newcommand{\RR}{\mathcal{R}}
\newcommand{\RRR}{\mathbb{R}}
 
\newcommand{\ip}[1]{\langle {#1} \rangle}
\newcommand{\Lc}[1]{\mathcal{L}^{({{#1}})}}
\newcommand{\Sc}{S} 
\newcommand{\Gc}{\mathcal{G}}
\newcommand{\Dc}{\mathcal{D}}
\newcommand{\Dcxn}{\mathcal{D}(x)^{(n)}}
\newcommand{\Dcn}{\mathcal{D}^{(n)}}
\newcommand{\Dcnu}{\mathcal{D}^{(\nu)}}
\newcommand{\DcnF}{\mathcal{D}^{(n_F)}}
\newcommand{\Bc}{\mathcal{B}}
\newcommand{\Nc}{\mathcal{N}}
\newcommand{\Fc}{\mathcal{F}}

\newcommand{\Mco}{{M_0}} 
\newcommand{\Mcs}{M_1} 

\newcommand{\divx}{\mathop{\mathrm{div}_x}}
\newcommand{\gradx}{\mathop{\mathrm{grad}_x}}
\newcommand{\gradxt}{\mathop{\mathrm{grad}_{x\hat{t}}}}
\newcommand{\dt}{\mathop{\partial_t}}
\newcommand{\dth}{\mathop{\partial_{\hat{t}}}}
\newcommand{\uh}{{\hat{u}}}
\newcommand{\veps}{{\varepsilon}}
\newcommand{\om}{{\varOmega}}
\newcommand{\ov}{{\varOmega^{\vtx}}}
\newcommand{\ovh}{{\varOmega_h^{\vtx}}}
\newcommand{\oh}{{\varOmega_h}}
\newcommand{\din}{{\partial_{\text{in}}}}
\newcommand{\Fh}{\hat{F}^n}
\newcommand{\curl}{\mathop{\text{curl}}}
\newcommand{\vphi}{\varphi}
\newcommand{\tb}{\mathchar'26\mkern-7.5mu  b}
\newcommand{\dtb}{\partial_{\tb}}
\newcommand{\dbot}{\partial_{\mathrm{bot}}}
\newcommand{\dtop}{\partial_{\mathrm{top}}}
\newcommand{\dbdr}{\partial_{\mathrm{bdr}}}
\newcommand{\vptb}{\varphi_{\tb}}
\newcommand{\vpbot}{\varphi_{\mathrm{bot}}}
\newcommand{\vptop}{\varphi_{\mathrm{top}}}
\newcommand{\Vhv}{V_h^\vtx}
\newcommand{\Tv}{{T^\vtx}}
\newcommand{\Tvk}{{T_\kk^\vtx}}
\newcommand{\Tvh}{{\hat{T}^\vtx}}
\newcommand{\Fcvb}{{\Fc^\vtx_b}}
\newcommand{\Fcvi}{{\Fc^\vtx_i}}
\newcommand{\jmp}[1]{\ldbrack{{#1}}\rdbrack}
\allowdisplaybreaks   

\begin{document}

\title[Convergence analysis on tents]{Convergence analysis of some tent-based schemes for linear hyperbolic systems}

\author[D.~Drake]{Dow Drake}
\address{Portland State University, PO Box 751, Portland OR 97207,USA }
\email{ddrake@pdx.edu}

\author[J.~Gopalakrishnan]{Jay Gopalakrishnan}
\address{Portland State University, PO Box 751, Portland OR 97207,USA }
\email{gjay@pdx.edu}

\author[J.~Sch\"oberl]{Joachim Sch{\"oberl}}
\address{Technische Universit{\"a}t Wien, Wiedner Hauptstra\ss e 8-10, 1040 Wien, Austria}
\email{joachim.schoeberl@tuwien.ac.at}

\author[C.~Wintersteiger]{Christoph Wintersteiger}
\address{Technische Universit{\"a}t Wien, Wiedner Hauptstra\ss e 8-10, 1040 Wien, Austria}
\email{christoph.wintersteiger@tuwien.ac.at}

\thanks{This work was supported in part by NSF grant DMS-1912779.}


\date{}


\begin{abstract}
  Finite element methods
  for symmetric linear hyperbolic systems
  using unstructured advancing fronts
  (satisfying a causality condition) are considered in this work.
  Convergence results and error bounds are obtained for mapped tent
  pitching schemes made with standard discontinuous Galerkin
  discretizations for spatial approximation on mapped tents. 
  Techniques to study semidiscretization on mapped tents,
  design fully discrete schemes,
  prove local error bounds, 
  prove stability on spacetime fronts,
  and bound error propagated
  through unstructured layers are developed.
\end{abstract}

  
\maketitle

\section{Introduction}

Tent-based numerical methods for hyperbolic equations stratify a
spacetime simulation region in an unstructured manner, by tent-shaped
subregions, and advance solutions across them
progressively in time.
The partitioning into tents provides a rational design
for local time stepping, maintaining high order accuracy in both space
and time, and without any ad hoc projection or extrapolation steps, an
advantage that has been and continues to be effectively leveraged by
many researchers~\cite{AbediHaber18, AbediPetraHaber06, FalkRicht99,
  MonkRicht05, PerugSchobStock20, Richt94, YinAcharSobh00}.
Nevertheless, a drawback of tents is that they are not
tensor products of a spatial domain with a
time interval, necessitating development of new tent (spacetime)
discretizations coupling too many
spatiotemporal unknowns.  In~\cite{GopalSchobWinte17}, we overcame
this drawback by using a mapping technique that transforms spacetime
tents to tensor product spacetime cylinders. This opened avenues to
use standard techniques like the discontinuous Galerkin (DG)
discretizations for spatial discretization, together with efficient
matrix-free 
time stepping schemes, on the mapped tents.  Such methods, referred to
as {\em Mapped Tent Pitching} (MTP) schemes, have been applied to solve
a variety of linear and nonlinear hyperbolic
systems~\cite{GopalHochsSchob20, GopalSchobWinte17,
  GopalSchobWinte20}.

This is the first paper to provide convergence theorems for MTP
schemes. Although the scope of the analysis here is limited to linear
hyperbolic systems, we identify what we consider to be the basic
ingredients for error analysis of MTP schemes, such as a norm in which
stability on spacetime fronts can be obtained. We use techniques to
bound the propagation of error through layers of tents similar to
those in~\cite{FalkRicht99, MonkRicht05}. However, a number of new
tools are needed to overcome difficulties arising from a
time-dependent mass matrix generated due to the mapping. As we already
noted in \cite{GopalHochsSchob20}, the use of classical explicit
Runge-Kutta time stepping on these mapped systems
leads to loss of higher orders of convergence due
to the complications created by the map. We outlined an algorithmic
solution in~\cite{GopalHochsSchob20}, namely the  {\em Structure-Aware
  Taylor} (SAT) time stepping scheme, that accounts for the
specific structure of the time-dependent mass matrix. Here we shall
provide {\it a priori} error bounds for these as well as a few other
schemes.

Our analysis is divided into the next three sections. First, we borrow
a spatial DG discretization framework from~\cite{BurmaErnFerna10,
  ErnGuerm06} which permits the treatment of many important examples
of symmetric hyperbolic systems and many boundary condition choices,
all at once. The application of this framework to the mapped equation
(the pull back of the hyperbolic system from the physical tent to the
spacetime cylinder), is detailed in Section~\ref{sec:model-probl}.  We
then combine it with a semidiscrete analysis in
Section~\ref{sec:semidisc}. While the analysis in that section
ignores errors due to
time discretization that are undoubtedly present in practice, it
immediately clarifies in what norms one may expect stability on
spacetime advancing fronts, and what error bounds might be provable
after time discretization. The main result of this section is that
under the conditions spelled out later, we may expect the error in the
numerical solution at the final time to be $O(h^{p+1/2})$ where $h$
represents a spatial mesh size parameter and $p$ denotes the
spatial polynomial degree (used in the spatial DG discretization).
In this form,
the result is comparable to \cite[Theorem~5.1]{MonkRicht05} that
provides the same rate for their spacetime DG method using
spacetime polynomials of degree $p$ on tents.

In Section~\ref{sec:fully-discr}, we discuss several fully discrete
schemes that combine the spatial DG discretization on a mapped tent
with SAT or other time stepping.  We find that proving stability of
the fully discrete schemes requires some trickery. Ever since
the classical work of~\cite{LevyTadmo98}, we know that stability
regions and  ``naive
spectral stability analysis based on scalar eigenvalues arguments may
be misleading.''  Many researchers have since pursued energy-type arguments
to prove stability of time stepping schemes with spatial DG
discretizations \cite{BurmaErnFerna10, CockbShu98, SunShu17,
  ZhangShu04}. However, we are not able to directly apply existing 
techniques due to the nonstandard nature of the system we obtain after
mapping the hyperbolic equation.  Therefore, we start afresh,
beginning with the most basic scheme and proceeding to  more
complicated cases. Namely, in \S\ref{ssec:lowest-order-tent-implicit}
we prove unconditional strong stability for a lowest order
tent-implicit scheme. Then, we proceed to analyze a lowest order
``iterated'' explicit scheme in \S\ref{ssec:lowest-order-expl},
constructed as an iterative solver for the implicit scheme, for which
we prove a nonstandard conditional stability.  Then, in
\S\ref{ssec:arbitrary-order-sat}, we proceed to an $s$-stage SAT
scheme and show that its local error in a tent is $O(h^s)$, which
is comparable to the $s$th power of the time step since the amount of
local time advance in tents 
is tied to its spatial mesh size. We are
able to prove, in one case, stability under a traditional
Courant-Friedrichs-Levy (CFL) condition, by which we mean that the
amount of {\em local} time advance within a tent is limited by a
constant multiple of the {\em local} spatial mesh size. In another
case, we prove stability under a ``3/2 CFL'' condition.
In some non-tent-based
DG methods, others have encountered
a similar (4/3 CFL) limitation in stability
analyses~\cite{BurmaErnFerna10}. We offer the above-mentioned cases
not as the last word on stability, but rather to spur further research
into this interesting topic.

In the remainder of this section, we establish notation and the lingua
franca of tents that we use throughout.
Consider a
cylindrical domain $\om \times (0,T)$
in the physical spacetime, where the spatial
domain $\om$ is an open bounded subset of $\RRR^N$.
We assume that $\om$ is subdivided by a
simplicial mesh $\oh$.  The subscript $h$ denotes the maximal element
diameter of the spatial mesh $\oh$.  Spacetime tents are built atop
this spatial mesh, using the algorithms in~\cite{GopalSchobWinte17}
or~\cite{ErickGuoySulli05}.  We start by viewing the spatial mesh
$\oh$ at time $t=0$ as the initial advancing front. When one mesh
vertex $\vtx$ is moved forward in time, while keeping all other
vertices fixed, the advancing front is updated to the piecewise planar
surface formed by connecting the raised $\vtx$ to its neighboring
vertices. The new front differs from the old by a tent-shaped region,
which we denote by $\Tv$. Its projection onto $\om$ gives the vertex
patch $\ov$ of all spatial simplices connected to the vertex
$\vtx$. We shall refer to this process as {\em pitching} the tent
$\Tv$.  For concurrency, one pitches multiple tents simultaneously at
vertices whose vertex patches do not have a mesh element in their
pairwise intersections, as in Figure~\ref{fig:tents:L1}.
The canopies of these spacetime tents can be
represented as the graph of $\varphi_1(x)$, a continuous function that
is piecewise linear with respect to the mesh $\oh$ (whose value is
zero in locations where tents are not yet erected). These canopies
together form the next advancing front:
$C_1 = \{ (x, \vphi_1(x)): x \in \om\}$. Note that the time coordinate of
a point in $C_1$ is never less than that of
the corresponding point in the first front
$C_0 = \om \times \{0\}$.  This process is repeated by pitching tents
atop $C_1$, and later atop the subsequent advancing fronts that result
from each step (as illustrated 
in Figures~\ref{fig:tents:L2}--\ref{fig:tents:Lall}).

Reiterating,  the {\em advancing front} at step $i$ is
the graph of a lowest-order Lagrange finite element function
$\vphi_i(x)$:
\begin{equation}
  \label{eq:Ci}
  C_i = \{ (x, \vphi_i(x)): x \in \om\}.  
\end{equation}
We shall refer to the region between two successive advancing fronts
as a {\em layer}, i.e., 
\begin{equation}
  \label{eq:Li}
  L_i = \{ (x, t) \in \om \times (0, T): \; \varphi_{i-1}(x) \le t \le
  \varphi_i(x) \}
\end{equation}
denotes the $i$th layer for $i=1, 2, \ldots, m.$ The layer $L_i$
(see Figure~\ref{fig:tents}) is
made above $C_{i-1}$ by pitching tents atop vertex patches associated
with a subset of mesh vertices of $\oh$, which form the {\em pitch
locations} at that stage. Let $V_{i}$ denote the collection of such
vertices identifying the pitch locations on $C_{i-1}$. Then
$
  L_{i} = \mathop{\bigcup}_{\vtx \in V_i} \Tv.
$
The spacetime will generally contain multiple tents pitched at the
same vertex $\vtx$ at different time coordinates. Although referring
to a tent by its spatial pitch location alone, as in $\Tv$ above, is
generally ambiguous, it will not confuse us since we will usually be
occupied with analyzing one tent at a time.

A tent can be expressed as
\begin{equation}
  \label{eq:Tv}
  \Tv = \{ (x, t): x \in \ov, \; \vpbot^\vtx(x) \le t \le \vptop^\vtx(x)\}
\end{equation}
where
$\vpbot^\vtx$ and $\vptop^\vtx$ are continuous functions on $\ov$ that are
piecewise linear with respect to the mesh elements forming the vertex
patch $\ov.$ The function $\delta^\vtx(x) = \vptop^\vtx(x) - \vpbot^\vtx(x)$ on $\ov$
will feature often in the sequel. It arises as a weight in transformed
integrals and is degenerate at the points where tent top meets tent
bottom.

\begin{figure}
  \centering
  \subcaptionbox{Spatial mesh $\om_h$. \label{fig:tents:mesh}}{
      \includegraphics[trim=60   340   470   150, clip,  %
      width=0.3\linewidth]{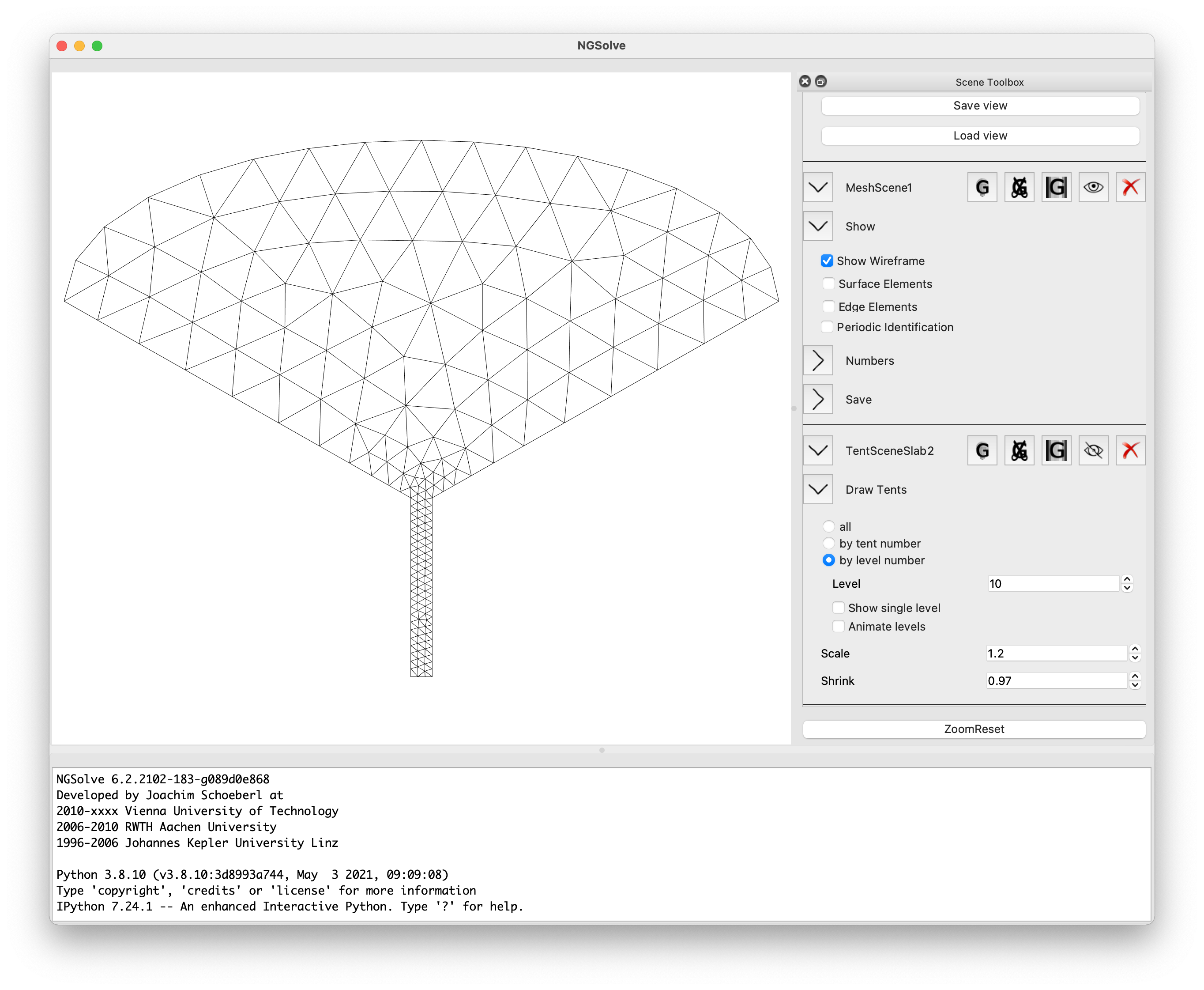}%
    }
    \subcaptionbox{
      The region in blue is the
      layer $L_1$. Tent canopies form part of the advancing
      front~$C_1$.
      \label{fig:tents:L1}
    }{
    \frame{
      \includegraphics[trim=60   400  500   150, clip,  %
      width=0.3\linewidth]{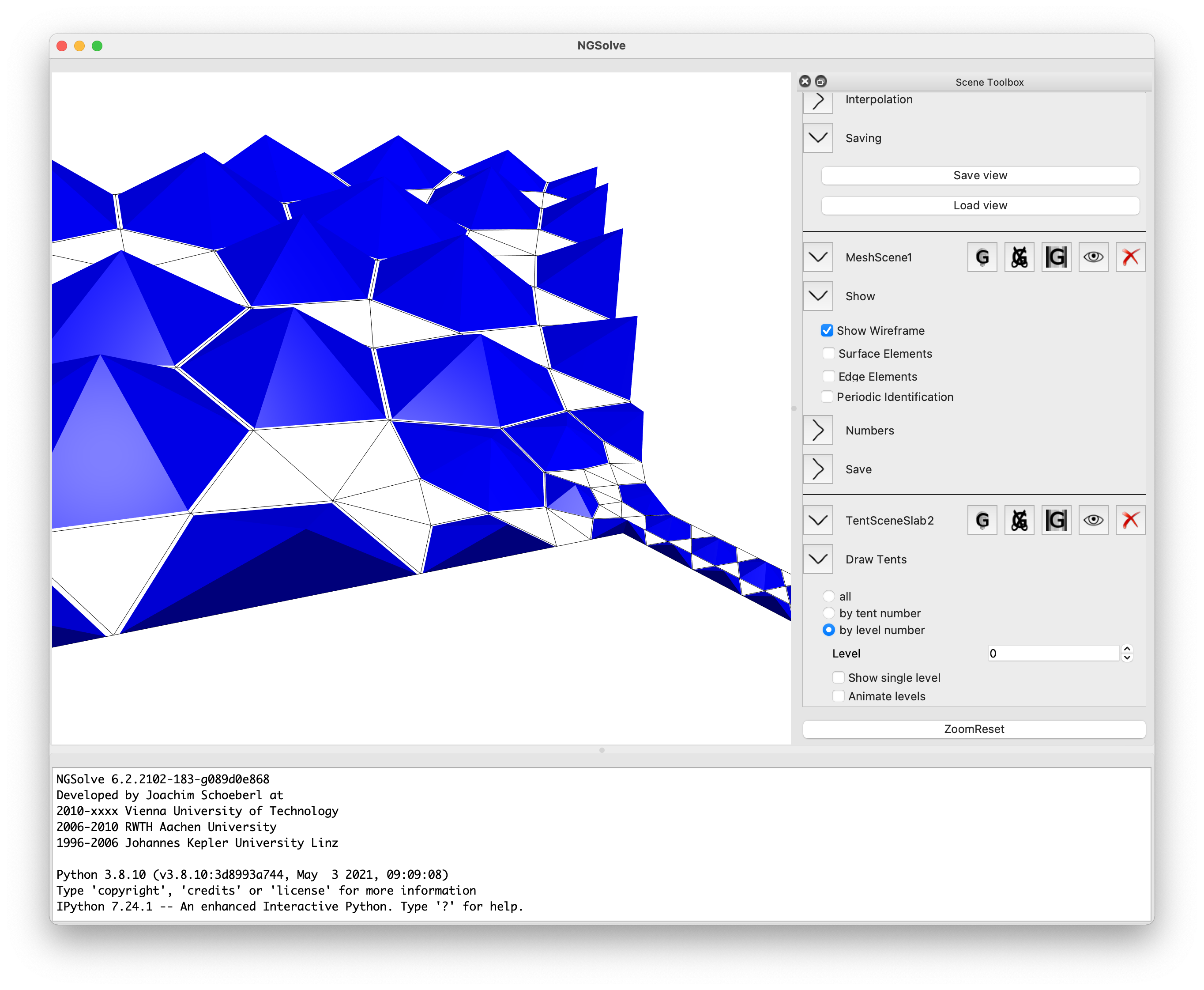}%
    }}
  \;
  \subcaptionbox{
    Green tents form
    layer $L_2$. Top  canopies (blue and green) are part of the 
    front~$C_2$.    \label{fig:tents:L2}
  }{
    \frame{
      \includegraphics[trim=60   400  500   150, clip,  %
      width=0.3\linewidth]{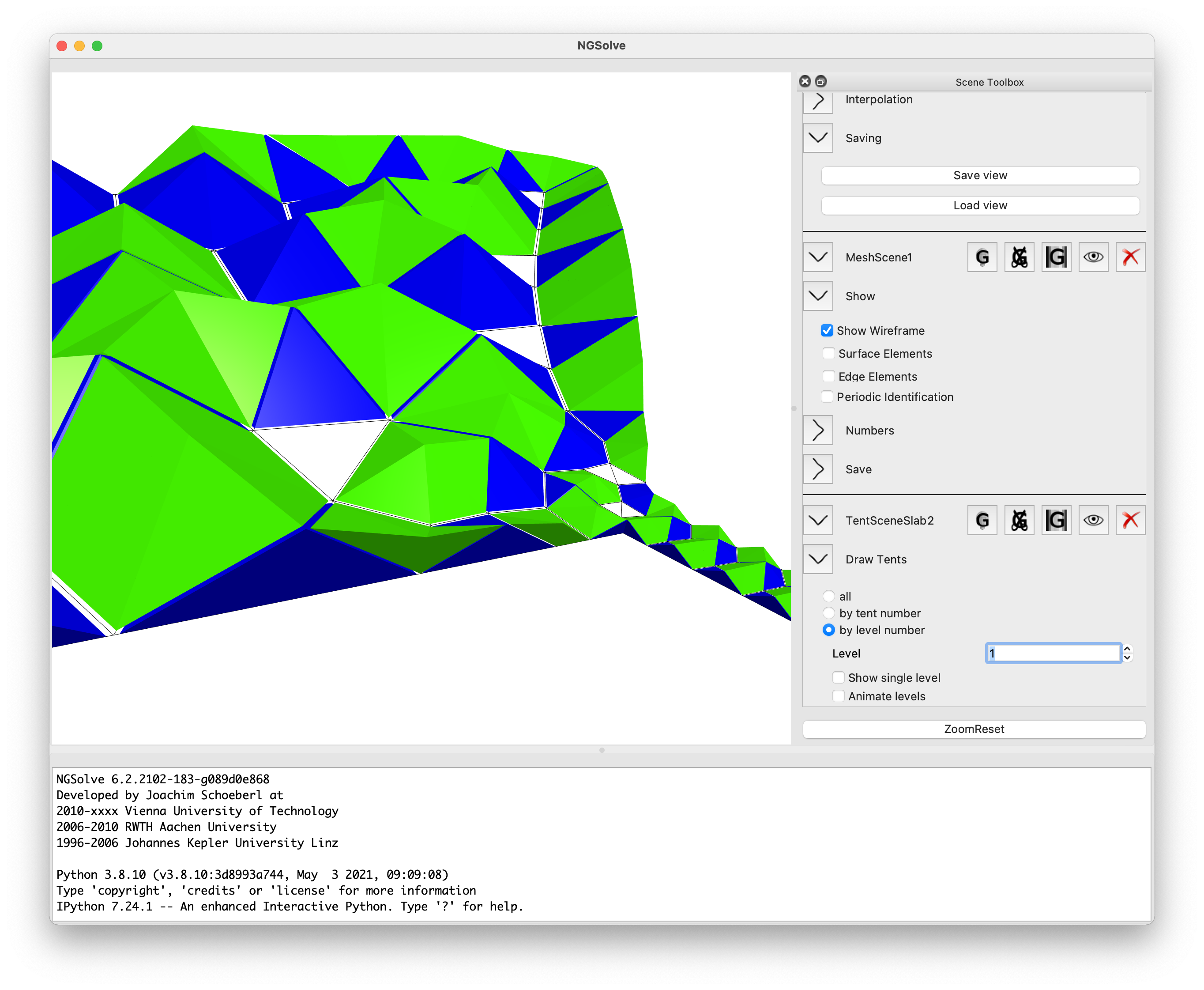}%
    }}
  \subcaptionbox{$L_3$ and $C_3$.}{
    \frame{
      \includegraphics[trim=60   400  500   150, clip,  %
      width=0.3\linewidth]{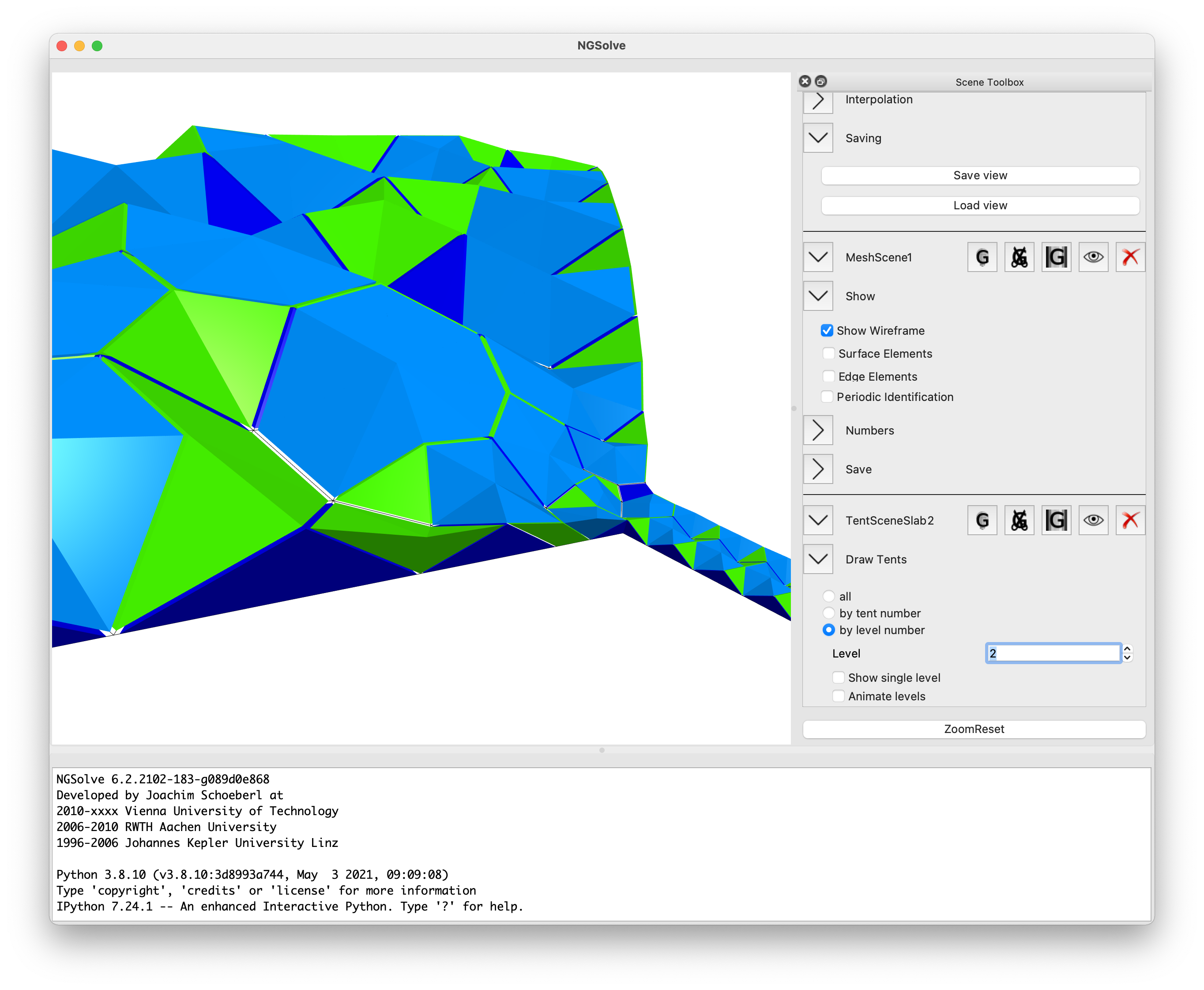}%
    }}\;
  \subcaptionbox{$L_4$ and $C_4$. \label{fig:tents:L4}}{
    \frame{
      \includegraphics[trim=60   400  500   150, clip,  %
      width=0.3\linewidth]{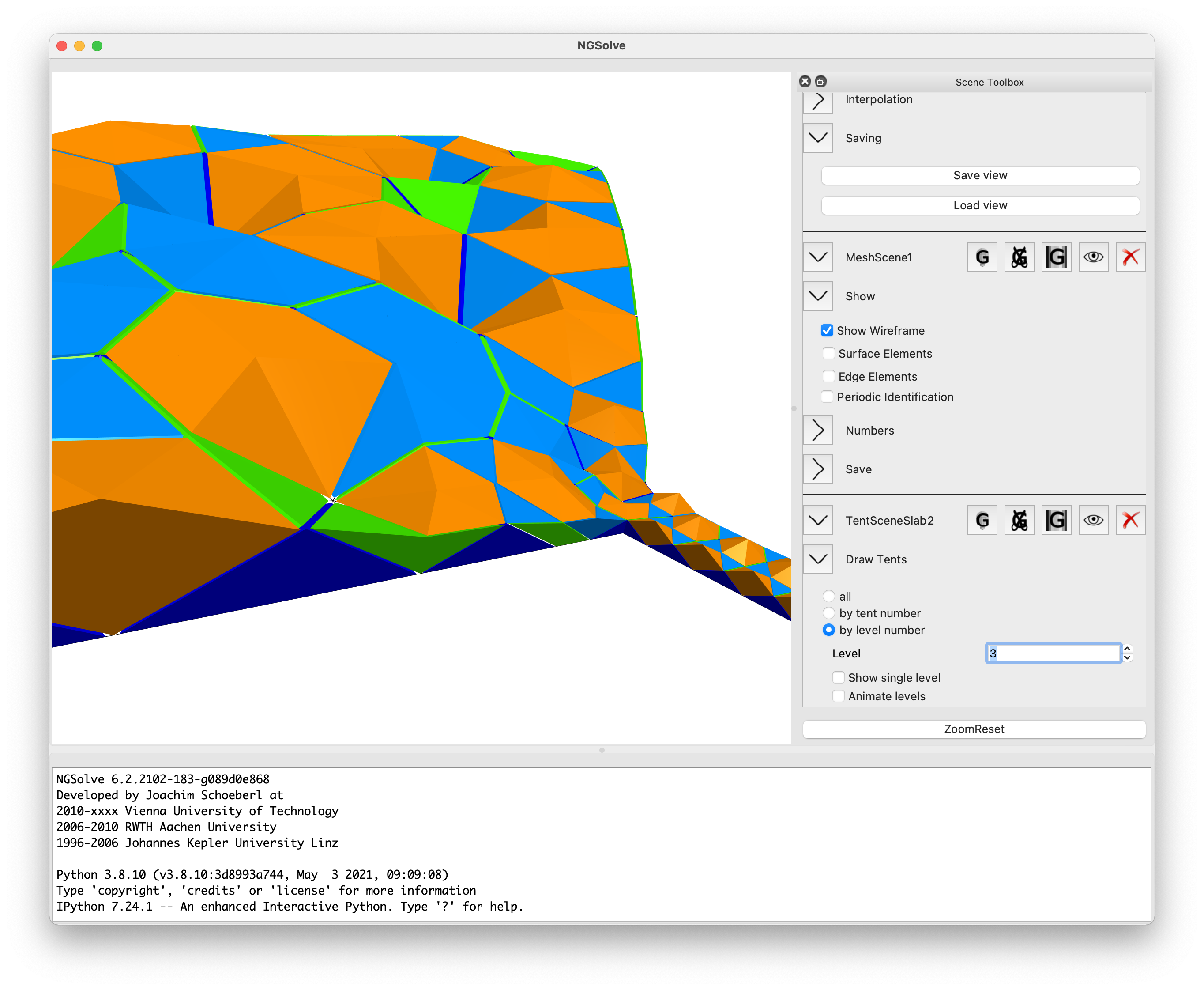}%
    }}\;
  \subcaptionbox{Layers filling $\om \times (0, t_{\mathrm{slab}})$.
    \label{fig:tents:Lall}}{
    \frame{
      \includegraphics[trim=60   400  500   150, clip,  %
      width=0.3\linewidth]{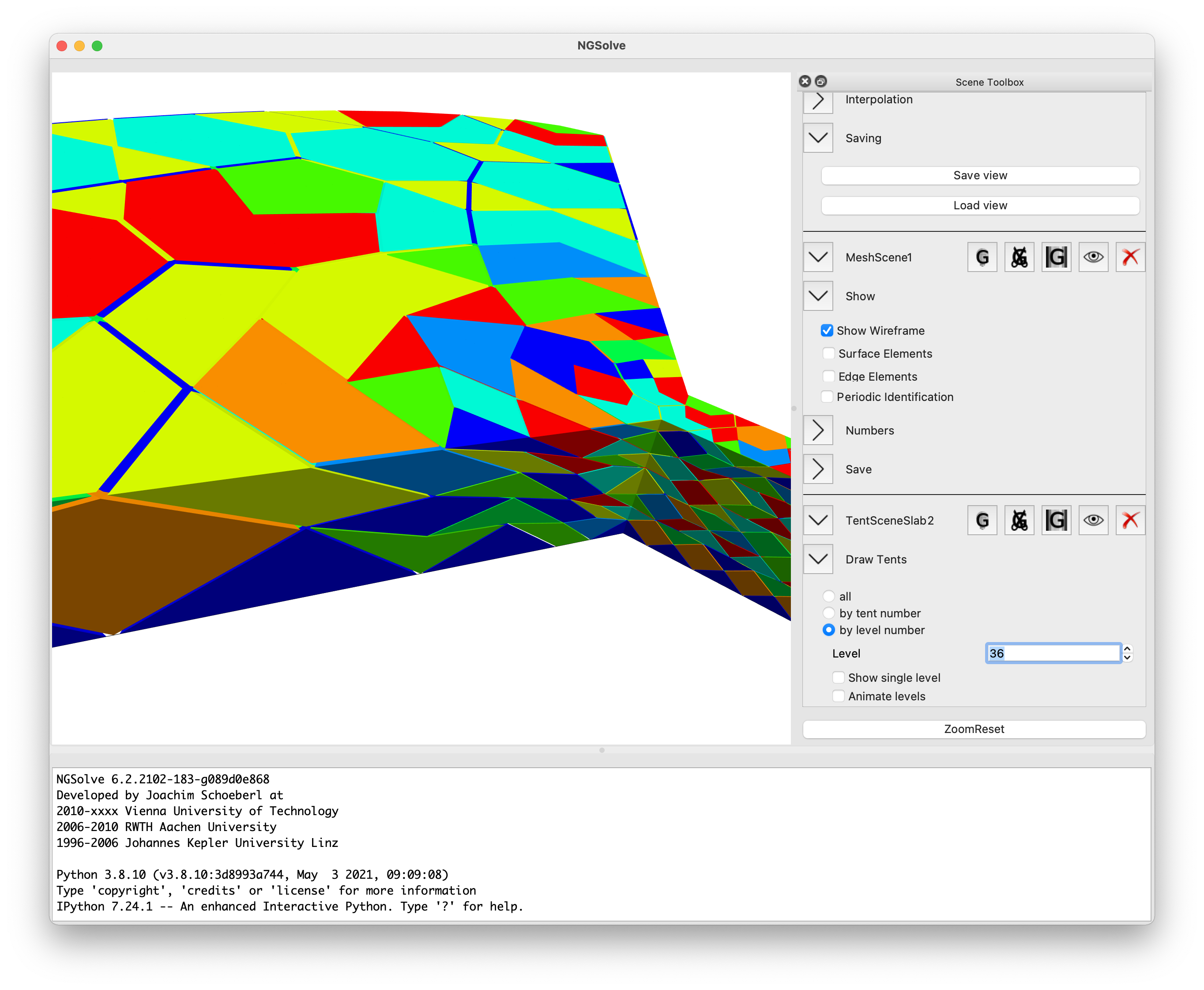}%
    }}
  \subcaptionbox{An initial pulse. \label{fig:tents:sol0}}{
    \frame{
      \includegraphics[trim=270  200   270   300, clip,  %
      width=0.3\linewidth]{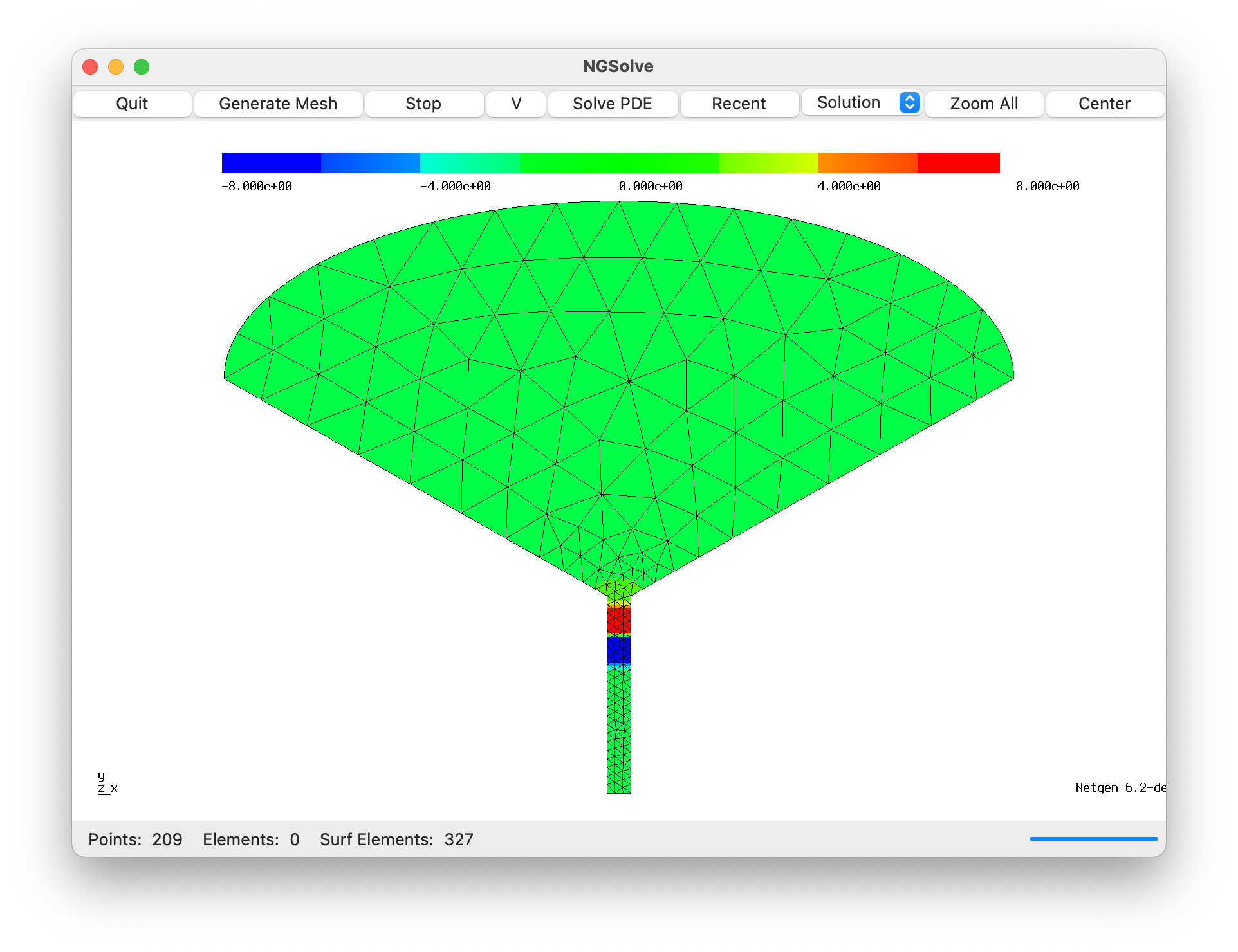}}
  }
  \subcaptionbox{Solution at $t =t_{\mathrm{slab}}$. \label{fig:tents:sol1}}{
    \frame{
      \includegraphics[trim=270  200   270   300, clip,  %
      width=0.3\linewidth]{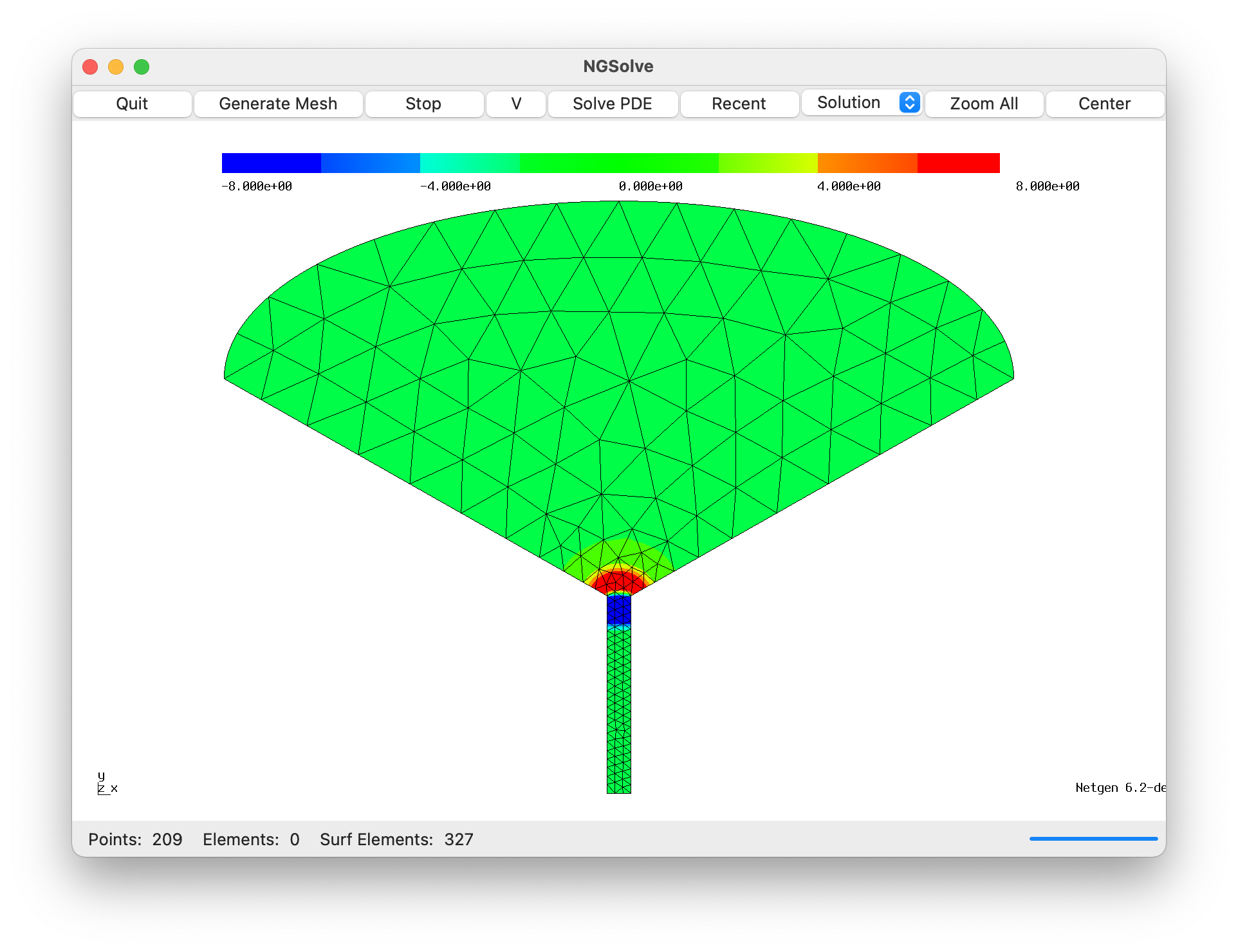}}
  }
  \subcaptionbox{Solution at $t =3t_{\mathrm{slab}}$. \label{fig:tents:sol3}}{
    \frame{
      \includegraphics[trim=270  200   270   300, clip,  %
      width=0.3\linewidth]{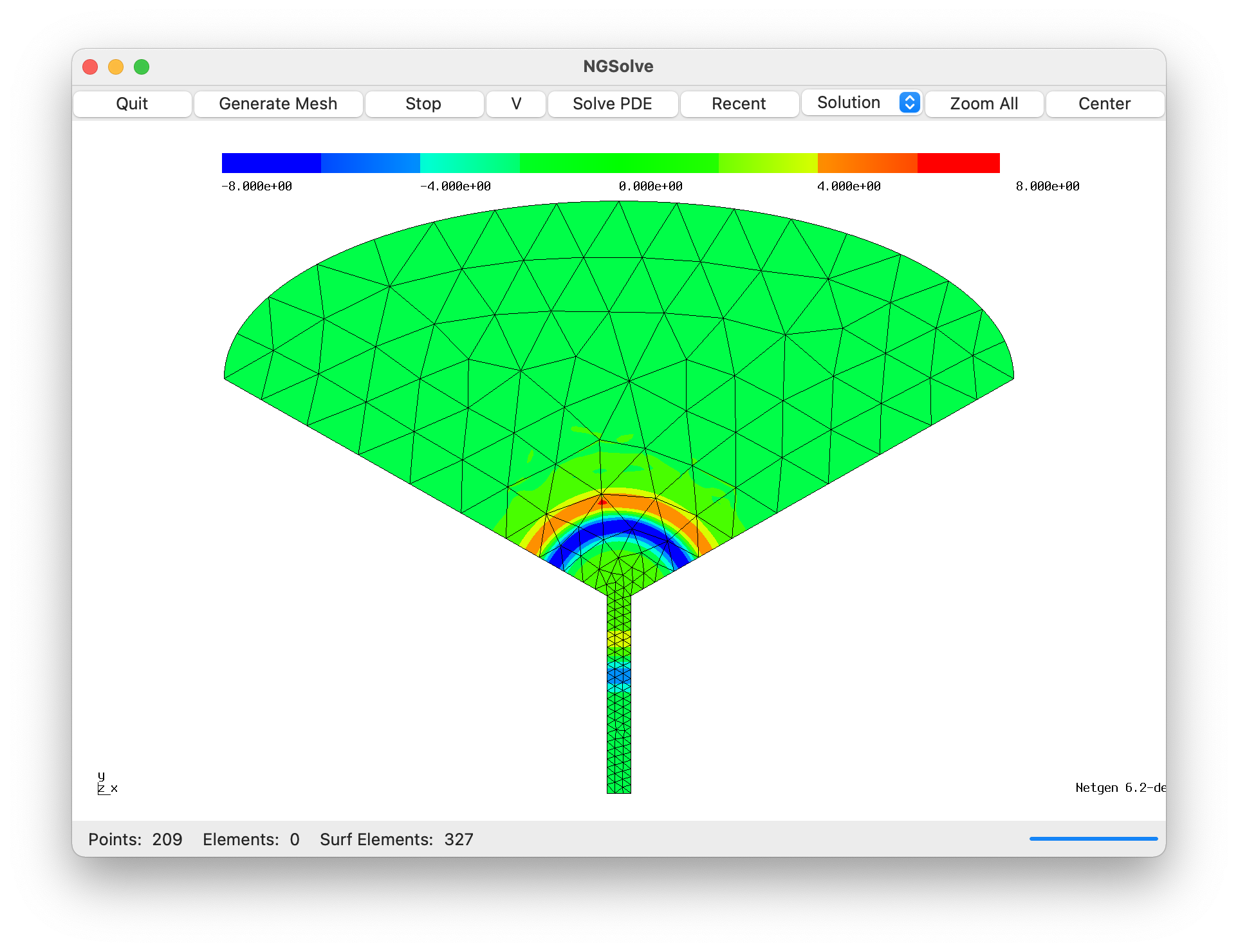}}
  }
  \caption{Tents, layers, advancing fronts, and solution snapshots
    from an acoustic wave simulation using the MTP scheme of
    \S\ref{ssec:arbitrary-order-sat}:
    Figures~\ref{fig:tents:L1}--\ref{fig:tents:L4} show successive
    layers of tents, viewed at an angle to show where small and large
    features meet (and time $t$ is in the vertical direction).
    Figure~\ref{fig:tents:Lall} shows how tents asynchronously
    enable both large and
    small time advances within a spacetime slab
    $\om \times (0, t_{\mathrm{slab}})$.  Plots of a time evolving
    wave solution at $t = 0, t_{\mathrm{slab}},$ and
    $3t_{\mathrm{slab}}$, computed using the tents
    in Figure~\ref{fig:tents:Lall},  are shown in
    Figures~\ref{fig:tents:sol0},~\ref{fig:tents:sol1},
    and~\ref{fig:tents:sol3}, respectively.}
  \label{fig:tents}
\end{figure}

One reason for working with tents is the ease by which {\em causality} can
be imposed, simply by adjusting the height of the {\em tent pole}, the
line segment connecting $(\vtx, \vpbot^\vtx(\vtx))$ to
$(\vtx, \vptop^\vtx(\vtx))$.  By the definition of hyperbolicity, the
maximal wave speed $c$ is finite.  When each spacetime tent encloses
the domain of dependence of all its points, causality holds. In other
words, if 
\[
  \| (\gradx \varphi_i) (x)\|_2 < \frac{1}{c}, \qquad x \in \om, 
\]
on all advancing fronts, then causality holds. In this paper, in place
of the strict inequality, we assume we are given a strict upper bound
$\cmax $ for the maximal wave speed $c$ and that our mesh of spacetime
tents is constructed so that
\begin{equation}
  \label{eq:causalitycondition}
  \| (\gradx \varphi_i) (x)\|_2 \le \frac{1}{\cmax}
  <  \frac{1}{c} \qquad x \in \om. 
\end{equation}
We will refer to~\eqref{eq:causalitycondition} as the {\em causality
  condition}. Algorithms for constructing tent meshes
satisfying~\eqref{eq:causalitycondition} can be
found in \cite{ErickGuoySulli05, GopalSchobWinte17}. They terminate
filling the spacetime $\om \times (0, T)$ with tents so that
the first and
the last advancing fronts, $C_0$ and $C_m$, are flat, i.e.,
$\vphi_0(x) \equiv 0$ and $\vphi_m(x) \equiv T$.
In practice, we often select (as in Figure~\ref{fig:tents})
a $t_{\mathrm{slab}}<T$ such that
$T = n_{\mathrm{slabs}} t_{\mathrm{slab}}$,  
run the tent meshing algorithm to
fill the subregion
$\om \times (0, t_{\mathrm{slab}})$, and then translate the
same mesh to  reuse it $n_{\mathrm{slabs}}-1$ times to 
cover the full spacetime domain $\om \times (0, T)$ without further
meshing overhead.
However, for the ensuing analysis, we ignore this extra subdivision
into smaller spacetime slabs (so $t_{\mathrm{slab}}=T$ henceforth).
With  this background
in mind, we proceed to  show 
how to map tents and construct an
MTP scheme after fixing a model problem.

\section{A model problem for analysis}  \label{sec:model-probl}

In this section, we describe a model symmetric linear hyperbolic
system and its MTP discretization that we shall be occupied
with. Although MTP schemes can be applied much more generally (as
shown in~\cite{GopalSchobWinte17, GopalSchobWinte20}), we restrict to
this model for transparently presenting the essential new ideas
needed for a convergence analysis.

\subsection{A symmetric linear hyperbolic system}

Let $L$ be a positive integer (and let $N$, as before,
denote the spatial dimension). Suppose that
$\Lc j: \om \to \RRR^{L \times L}$, for $j=1,\ldots, N$ and
$\Gc: \om \to \RRR^{L \times L}$ are symmetric bounded matrix-valued
functions and suppose $\Gc$ is uniformly positive definite in
$\bar\om$.  Our model problem is the following linear hyperbolic
system of $L$ equations, in $L$ unknowns, denoted by $u(x, t)$, or in
terms of  scalar 
components,  by $u_k(x, t)$:
\begin{equation}
  \label{eq:34}
  \dt g(u) + \divx f(u) = 0,  
\end{equation}
with
\[
  [g(u)]_l = \sum_{k=1}^L \Gc_{lk} u_k, \qquad [f(u)]_{lj} =
  \sum_{k=1}^L \Lc j_{lk} u_k.
\]
We have restricted ourselves to time-independent coefficients, and
shall do so also for boundary conditions, which are expressed using a
matrix field $\Bc: \partial \om \to \RRR^{L \times L}$.  Boundary
conditions are considered in the form studied by
Friedrichs~\cite{Fried58}, namely, 
\begin{equation}
  \label{eq:FreidrichsBC}
    (\Dc - \Bc) u = 0 \qquad \text { on } \partial \om
\end{equation}
where 
\begin{equation}
  \label{eq:D}
  \Dc = \sum_{j=1}^N n_j \Lc j.
\end{equation}
Note that $\Dc$, in general, depends on the point $x \in \d\om$ as
well as on $n(x)$, the spatial unit outward normal at $x$, and when
we wish to emphasize these dependencies, we shall denote it by $\Dcn$,
$\Dc(x)$, 
or $\Dcxn$.  Note that later in the sequel $n$ (or $n(x)$)
will also be used to
generically denote the outward unit normal on boundaries of other
domains (such as mesh elements).  Of course, the hyperbolic
system~\eqref{eq:34} must also be supplemented with an initial
condition,
\begin{equation}
  \label{eq:ic}
  u(x, 0) = u^0(x) \qquad  x \in \om
\end{equation}
at time $t=0$ for some given initial data $u^0$.

Friedrichs~\cite{Fried58}
identified conditions on the operator $\Bc$ for obtaining
well-posed boundary value problems. We shall borrow the same
conditions and impose them pointwise on $\d \om \times (0, T)$.  In
particular, at each $x \in \d\om$ and $t \in (0, T)$, we assume
$\ker( \Dc - \Bc ) + \ker (\Dc + \Bc) = \RRR^L$ and
$ \Bc + \Bc^t \ge 0.$ (The latter
is to be interpreted as
$(\Bc(x) + \Bc(x)^t) y \cdot y = 2\Bc(x) y \cdot y \ge 0$
for all vectors $y \in \RRR^L$ at any $x \in \d\om$.)  To simplify the analysis, we shall
also assume that 
\begin{equation}
  \label{eq:divL0}
  \sum_{j=1}^N \partial_j \Lc j = 0
\end{equation}
in the sense of distributions (so jumps in $\Lc j (x)$ are allowed so
long as~\eqref{eq:divL0} holds) and 
\begin{equation}
  \label{eq:G_L_const}
  \Gc \text{ and } \Lc j \text{ are constant on each mesh element }
  K \in \oh.
\end{equation}
Across a mesh facet $F$ of normal $n_F$, it is easy to see
that~\eqref{eq:divL0} implies the continuity of
$\DcnF.$ As long as we obtain this normal continuity (which is
needed in our analysis), 
assumption~\eqref{eq:divL0} can be relaxed to other forms (such
as what is suggested in~\cite[equation (A.4)]{ErnGuerm06}),
at the expense of a few
additional technicalities.
Assumption~\eqref{eq:G_L_const} allows us to zero out some projection
error 
terms instead of tracking such small terms in error estimates.

\subsection{Mapping  tents}

Consider a tent expressed as in~\eqref{eq:Tv}.  We map to a tent $\Tv$
from a cylindrical tensor product domain $\Tvh = \ov \times (0, 1)$
using the map $\vPhi^\vtx(x, \hat t) = (x, \vphi^\vtx(x, \hat t))$
where
\[
  \vphi^\vtx(x, \hat t) = (1- \hat t) \vpbot^\vtx(x) + \hat
  t\vptop^\vtx(x)
  = \vpbot^\vtx(x) + \hat t\, \delta^\vtx(x)
\]
(see \cite[Fig.~2]{GopalSchobWinte17} for an illustration of this map).
We will drop the superscript $\vtx$ when it is obvious from context.
Clearly, $\vPhi (\Tvh) = \Tv$ and the interior of $\Tvh$ is mapped
one-to-one onto the interior of $\Tv$ (but the map is not one-to-one
from the boundary of $\Tvh$ to the boundary of $\Tv$).  The coordinate
$\hat t$ in $\Tvh$ will be referred to as the {\em pseudotime}
coordinate. Note that the time coordinate $t$ in the physical
spacetime twists space and pseudotime together since it is given by
$t = \vphi(x, \hat t)$.

The Jacobian matrix of the map $\vPhi$ is easily computed:
  \begin{equation}
  \label{eq:gradxtPhi}
  \gradxt \vPhi =
  \begin{bmatrix}
    I & 0 \\ (\gradx\vphi)^T &\delta
  \end{bmatrix}
\end{equation}
Using it in a Piola transformation, it can be shown 
\cite[Theorem~2]{GopalSchobWinte17} that
the mapped hyperbolic solution
$\hat u = u \circ \vPhi$ satisfies 
\begin{equation}
  \label{eq:1}
  \dth \left[
    g(\uh) - f(\uh) \gradx \varphi
    \right]+ \divx \big[ \delta f(\uh)\big] = 0,\qquad \text{ in } \Tvh,
\end{equation}
whenever $u$ solves~\eqref{eq:34}.  MTP schemes proceed by
solving~\eqref{eq:1} by various discretization strategies
(particularly those that leverage the tensor product nature of space
and pseudotime in $\Tvh$) and then pulling back the computed solution
to the physical spacetime. We now proceed to discuss a discretization
strategy that uses a DG spatial discretization on $\ov$. We will
combine it with pseudotime discretizations later (in Section~\ref{sec:fully-discr}).

\subsection{Spatial discretization on mapped tents}

The mapped equation~\eqref{eq:1} can be approximated by any standard
scheme that allows for a time dependent mass matrix.  Our focus is on
discontinuous Galerkin (DG) discretizations. Since there are numerous
flavors of DG schemes and numerical fluxes, for efficiently covering
various choices, we use the framework of~\cite{ErnGuerm06} (see
also~\cite{BurmaErnFerna10}), a simplified version of which, adapted
to our purposes, is described next. Their framework is motivated by
the previously mentioned early work of Friedrichs~\cite{Fried58}, and
other previous authors have also been similarly motivated while
considering boundary conditions, notably~\cite{FalkRicht99}
and~\cite{MonkRicht05} in the context of tents and spacetime methods.

We assume that $\oh$ is a shape regular conforming simplicial mesh of
domain $\om$.  We use $a \lesssim b$ to indicate that there is a
constant $C>0$ such that $a \le C b $ and that the value of $C$ may be
chosen independently of any spatial mesh chosen from the shape regular
family.  The value of the generic constant $C$ in ``$\lesssim$'' may
differ at different occurrences and is allowed to depend on the wave
speed, material coefficients, spatial polynomial degree, etc.  Let
$P_p(K)$ denote the space of polynomials of degree at most $p$
restricted to the domain $K$ and let $V_h = \{ v : v|_K \in P_p(K)^L$
for any mesh element $K \in \oh\}$.  Let $\ovh$ denote the collection
of elements in the vertex patch $\ov$ of a mesh vertex~$v$.  Let
$\Vhv$ denote the restriction of the spatial DG space on $\ov$ and let
$\psi_j(x)$ denote a basis for $\Vhv$. The semidiscrete approximation
of $\uh$ is of the form
\[
  \uh_h(x, {\hat{t}}) = \sum_j U_j({\hat{t}}) \psi_j(x).
\]

We consider a DG semidiscretization of~\eqref{eq:1} of the form
displayed next in~\eqref{eq:DG}. In the spatial integrals there and
throughout, we do not explicitly indicate the measure (volume or
boundary measure) as it will be understood from context.  For each
fixed $0 < \hat t\le 1$, the function $\uh_h(\cdot, \hat t)$ satisfies
\begin{equation}
  \label{eq:DG}
  \int_\ov \dth \big[ g(\uh_h) - f(\uh_h) \gradx \varphi\big] \cdot  v
  = \sum_{K \in \ovh} \bigg[
  \int_K \delta f(\uh_h) : \gradx v -
  \int_{\partial K} \delta \Fh_{\uh_h} \cdot v \bigg],
\end{equation}
for all $ v \in \Vhv,$ where the numerical flux $\Fh_{\uh_h}$ on an
element boundary $\d K$ is defined using the values of $\uh_h$ from
the element $K$ as well as from the neighboring element $K_o$, as
follows. For any $w \in \Vhv$, at a point
$x \in \partial K \cap \partial K_o$, letting $w_o = w|_{K_o}$, define
\[
  \{ w \}(x) = \frac 1 2 ( w + w_o),
  \qquad
  \jmp{w}(x) = w - w_o.
\]
Then, $\Fh_w$  is assumed to take the form
\begin{equation}
  \label{eq:NumFlux}
  \Fh_w = 
    \begin{cases}
      \Dc \{ w \} + \Sc \jmp{w}
      & \text{ on } \partial K \setminus  \partial \om,
      \\
      \frac 1 2 (\Dc + B) w
      & \text{ on } \partial K \cap  \partial \om.
    \end{cases}
\end{equation}
Here, $\Dc$ is defined by~\eqref{eq:D}, with the vector
$n$ now denoting the
outward unit normal on~$\partial K$, 
$
  \Sc: \Fc_i \to \RRR^{L \times L}
$
denotes a 
stabilization matrix on interior facets of 
$\Fc_i = \bigcup\{ \d K \setminus \d\om: \; K \in \ovh \}$,
and
$B: \d \om \to  \RRR^{L \times L}$,
is used to model the exact boundary condition $\Bc$ with any needed extra
stabilization on boundary facets.
Note that $S$ is
single-valued on $\Fc_i$ (while $\Dc$ is multivalued and
depends on the sign
of the normal on an element boundary). 
Let $\| \cdot \|_2$ denote the Euclidean norm (of a vector, or the
induced norm for a matrix) and let
$|y|_S = (S y \cdot y)^{1/2}$ and $| y |_B = (B y \cdot y)^{1/2}$.
We assume that
\begin{subequations}
  \label{eq:DGdesign}
  \begin{align}
    \label{eq:kernelB}
    \ker\left(\Dc(x) - \Bc(x)\right)\,
    & \subseteq\, \ker\left(\Dc(x) - B(x)\right)     &&   x \in \d\om, 
    \\
    \label{eq:Bpd}
    B(x) + B(x)^t & \ge 0,  \qquad \| B(x) \|_2 \lesssim 1,
    &&   x \in \d\om, 
    \\ \label{eq:12}
    \left(\Dc(x) + B(x)\right) y \;\cdot z \;
    & \lesssim\; \| y \|_2\; |z|_B,    
                  &&  x \in \d\om,
                     \; y, z \in \RRR^L,
    \\
    \label{eq:Spd}
    S(x) + S(x)^t & \ge 0,  \qquad \| S(x) \|_2 \lesssim 1,
                                                     && x \in \Fc_i,
    \\
    \label{eq:Sbound}
    S(x) y \cdot z
    \;  & \lesssim\; | y|_S \;|z |_S, 
                  &&  x \in \Fc_i,
                     \; y, z \in \RRR^L.
    \\
    \label{eq:11}
    \Dc(x) y \cdot z\;  & \lesssim\; \| y\|_2 \;|z |_S, 
                  &&  x \in \Fc_i,
                     \; y, z \in \RRR^L.
  \end{align}
\end{subequations}
These form a subset of the  ``design conditions for DG methods''
in~\cite{ErnGuerm06} that we shall need for our analysis in
the next section.

\subsection{Examples}
\label{ssec:examples}

The following examples show a variety of equations, boundary
conditions, and their
well-known discretizations that conform to the  framework
introduced above. We will work out the first example in some detail
and describe the rest telegraphically since similar examples can be
found in the literature~\cite{BurmaErnFerna10, ErnGuerm06,
  FalkRicht99, MonkRicht05}.

\begin{example}[Maxwell equations with impedance boundary conditions]
  \label{eg:MaxwellImpedance}
  Suppose we are given
  electric permittivity $\veps(x)$ and magnetic
  permeability $\mu(x)$
  as positive functions on $\om$ and let $Z>0$.
  The Maxwell system for
  the electric field $E(x, t)$ and magnetic field $H(x, t)$, with
  impedance boundary conditions, consists of 
  \begin{subequations}
    \begin{align}
      \label{eq:36}
      \veps \dt E - \curl H   = 0,
                               \quad\quad
             \mu \dt H + \curl E & =0,
      && \text{ in } \om
         \times (0, T),
      \\
      \label{eq:35}
      n \times E - Z n \times (H \times n)& = 0,
      && \text{ on } \d\om
      \times (0, T).      
    \end{align}
  \end{subequations}
  To fit this into the prior setting, we put
  $  N=3,$ $ L=6,$ 
  and 
  \begin{align*}
    u =
    \begin{bmatrix}
      E \\ H 
    \end{bmatrix}, 
    \quad 
    \Lc j =
    \begin{bmatrix}
      0 & [\epsilon^j] \\
      [\epsilon^j]^t & 0 
    \end{bmatrix},
                       \quad
                       \Gc =
                       \begin{bmatrix}
                         \veps I &  0 \\
                         0       & \mu I 
                       \end{bmatrix},
                                   \quad
   \Bc =
                                   \begin{bmatrix}
                                     0 & \Nc \\
                                     \Nc & -2Z \Nc^t \Nc
                                   \end{bmatrix},
  \end{align*}
  where $\epsilon^j$ is the $3 \times 3$ matrix whose $(l,m)$th entry
  equals the value of the Levi-Civita alternator $\epsilon_{jlm}$ and
  $\Nc = \sum_{j=1}^3 n_j \epsilon^j \in \RRR^{3 \times 3}$.  Noting
  that $ \Nc^t = -\Nc,$ $\Nc E = E \times n,$ and
  $ \Nc^t \Nc H = n \times (H \times n),$
  it is easy to see that
  \begin{equation}
    \label{eq:37}
    \Dc =
    \begin{bmatrix}
      0 & \Nc \\
      \Nc^t & 0 
    \end{bmatrix},
    \quad (\Dc - \Bc)
    \begin{bmatrix}
      E \\ H 
    \end{bmatrix}
    =
    2
    \begin{bmatrix}
      0 \\  n \times E - Z n \times (H \times n)
    \end{bmatrix},
  \end{equation}
  so that~\eqref{eq:FreidrichsBC} indeed imposes the impedance
  boundary condition~\eqref{eq:35}.

  Next, for the DG discretization, set
  \[
    B =
    \frac{1}{1+Z} 
    \begin{bmatrix}
      2\Nc \Nc^t & (1-Z)\Nc \\
      (1-Z)\Nc & 2Z\Nc \Nc^t
    \end{bmatrix},
    \quad 
    S = 
    \begin{bmatrix}
      \Nc^t\Nc  & 0  \\
      0 & \Nc^t\Nc 
    \end{bmatrix}.
  \]
  Since $\| n \times (E \times n) \|_2^2 = \| \Nc E\|_2^2$, 
  setting $b =  n \times E - Z  n\times (H \times n)$,  it is easy to
  see that 
  \begin{equation}
    \label{eq:38}
    (\Dc - B)
    \begin{bmatrix}
      E \\ H 
    \end{bmatrix}
    = \frac{2}{1+Z}
    \begin{bmatrix}
      n \times b \\
      b 
    \end{bmatrix},
    \quad
    \left|
      \begin{bmatrix}
        E \\ H 
      \end{bmatrix}
    \right|_B^2 
    = \frac{2 }{1+Z}
    \left(
      \| n \times E \|_2^2 + Z \| n \times H\|_2^2
    \right),
  \end{equation}
  and the latter shows~\eqref{eq:Bpd} since $Z>0$.  The first identity
  of~\eqref{eq:38}, together with~\eqref{eq:37} implies
  \eqref{eq:kernelB}. Similar computations establish~\eqref{eq:12}.
  Finally, noting that $\left| [\begin{smallmatrix} E \\ H
  \end{smallmatrix}
  ]\right|_S^2 = \| n \times E \|_2^2 + \| n \times H \|_2^2,$ the
remaining properties~\eqref{eq:Spd}
and~\eqref{eq:11}, are easily established. Note that if the above
$S$ is scaled by 1/2, then the conditions of~\eqref{eq:DGdesign} continue
to hold and we get the ``classic upwind'' flux~\cite[p. 434]{HW08}
for Maxwell equations.
\end{example}

\begin{example}[Maxwell equations with perfect electric boundary
  conditions]
  \label{eg:MaxwellPEC}
  Reconsider Example~\ref{eg:MaxwellImpedance} with $Z=0$. This yields
  a Dirichlet boundary condition modeling the electrical isolation by
  a perfect electric conductor. Substituting $Z=0$ in previous
  choices of
  $\Bc, B, $ and $S$, one can show that all conditions
  of~\eqref{eq:DGdesign} continue to hold even though
  the $|\cdot |_B$ seminorm is now weaker: 
  $\left| [\begin{smallmatrix} E \\ H
  \end{smallmatrix}
  ]\right|_B^2 = 2\| n \times E \|_2^2$ from~\eqref{eq:38}.
\end{example}

\begin{example}[Advection]
  \label{eg:advection}
  The advection problem  with inflow boundary conditions, 
  \[
    \dt u + \divx ( b u)= 0\quad
    \text{ in }\om \times (0, T), \qquad
    u = 0 \quad \text{ on } \din \om \times (0, T),
  \]
  where 
  $b : \om \to \RRR^N,$
  is some given vector field, 
  fits the above setting with $L=1$ (keeping the
  spatial dimension $N$ arbitrary),
  $\Lc j = b_j \in \RRR^{1 \times 1},$ $\Gc =1$ and $\Dc = b \cdot n$.
  Examples  of $b$ that satisfy~\eqref{eq:divL0}
  and~\eqref{eq:G_L_const} are offered by divergence-free functions in the 
  lowest order Raviart-Thomas finite element space.
  The inflow boundary condition is recovered by setting
  $\Bc = |b\cdot n|$.
  The choices
  \[
     B = |b\cdot n|, \qquad \Sc = \frac 1 2 | b \cdot n|,
  \]
  are easily seen to verify~\eqref{eq:DGdesign} and yield the
  classical upwind DG discretization.
\end{example}

\begin{example}[Wave equation with Dirichlet boundary conditions]
  \label{eg:WaveDirichlet}
  Rewriting
  \begin{equation}
    \label{eq:39}
    \partial_{tt} \phi - \Delta \phi=0,
    \quad
    \text{ in }\om \times (0, T), \qquad
    \phi = 0 \quad \text{ on } \d\om \times (0, T),
  \end{equation}
  as a first order hyperbolic system for
  $L=N+1$ 
  variables using
  the flux $q = -\gradx \phi$
  and $\mu = \dt \phi$, we match the prior framework. Put 
  $u =
  [
  \begin{smallmatrix}
    q \\ \mu
  \end{smallmatrix}
  ]$, 
  $\Gc$ to identity,
  and $\Lc j = e_j e_{N+1}^t + e_{N+1} e_j^t$ (using the standard
  unit basis vectors $e_j$ of $\RRR^{N+1}$). The 
  Dirichlet boundary
  conditions on $\phi$ can be imposed by requiring that $\mu=0$ on
  $\d\om \times (0, T)$, which is what~\eqref{eq:FreidrichsBC} yields
  with 
  \[
    \Bc =
    \sum_{j=1}^N n_j (e_{N+1} e_j^t - e_j e_{N+1}^t)
    + 2 e_{N+1} e_{N+1}^t
    =
    \begin{bmatrix}
      0 & -n \\
      n^t & 2 
    \end{bmatrix}.
  \]
  All conditions of~\eqref{eq:DGdesign} are satisfied by setting
  \[
    B = \Bc, \qquad
    \Sc = 
    \begin{bmatrix}
      n n^t & 0 \\
      0 & 1
    \end{bmatrix}.
  \]
  These choices yield the DG discretization with upwind-like fluxes for the
  wave equation. 
\end{example}

\begin{example}[Wave equation with Robin boundary conditions]
  \label{eg:WaveRobin}
  We reconsider Example~\ref{eg:WaveDirichlet} after replacing the
  boundary condition in~\eqref{eq:39} by
  $\d \phi/ \d n + \rho \d_t \phi = 0$ for some $\rho > 0$, or
  equivalently, in terms of the variables $q, \mu$ introduced there,
  \begin{equation}
    \label{eq:40}
    n \cdot q - \rho \mu =0, \quad \text{ on } \d\om \times (0, T).    
  \end{equation}
  Keeping the same $S$ and changing 
  \[
    \Bc =
    \begin{bmatrix}
      0 & n \\
      -n^t & 2 \rho 
    \end{bmatrix},
    \quad
    B =
    \begin{bmatrix}
      \rho^{-1} n n^t & 0 \\
      0 &  \rho 
    \end{bmatrix},
  \]
  all conditions of~\eqref{eq:DGdesign} are satisfied.
\end{example}

\begin{example}[Wave equation with Neumann boundary conditions]
  \label{eg:WaveNeumann}
  This is the boundary condition obtained when  $\rho=0$
  in~\eqref{eq:40}. The boundary condition as well as the conditions
  of~\eqref{eq:DGdesign} are verified with
  \[
    \Bc =
    \begin{bmatrix}
      0 & n \\
      -n^t & 0
    \end{bmatrix},
    \quad
    B =
    \begin{bmatrix}
      n n^t & n  \\
      -n^t  &  0
    \end{bmatrix},
  \]
  keeping  $S$ unchanged.
\end{example}

\section{Analysis of semidiscretization} \label{sec:semidisc}

In this section, we prove stability of the
DG semidiscretization~\eqref{eq:DG}
on the
advancing fronts. When combined with standard finite element
approximation estimates, this leads to the main result of this section,
namely the error estimate of Theorem~\ref{thm:semidiscrete} below.

\subsection{Preparatory observations}

Let $H^1(\ovh)$ denote the broken Sobolev space isomorphic to
$\Pi_{K\in \ovh} H^1(K)$. Since our variables have $L$ unknown
components, we
will need $L$ copies of this space. To ease notation, we abbreviate
$\Hhv = H^1(\ovh)^L,$  $\Hv = H^1(\ov)^L$,
and $\Lv = L^2(\ov)^L$.  Since the traces of a
$w \in \Hhv$ on element boundaries are square integrable, the
following definition of
the bilinear form
$a: \Hhv\times \Hhv \to \RRR$,
with the numerical fluxes $\Fh_w$ from~\eqref{eq:NumFlux}, makes sense:
\begin{align}
  \nonumber
  a(w, v)
  & = \sum_{K \in \ovh} \bigg[
    \int_K \delta f(w) : \gradx v -
    \int_{\partial K} \delta \Fh_w \cdot v\bigg].
\end{align}
Let $(w, v)_D$ denote the inner product in $L^2(D)$, or its Cartesian
products, for any domain $D$, and let $\| w\|_D = (w, w)_D^{1/2}$.
Using this notation, we may alternately
write $a(\cdot, \cdot)$ as
\begin{equation}
  \label{eq:6}
  a(w, v)  = \sum_{K \in \ovh}
  \left[
     \sum_{j=1}^N (\delta \Lc j w, \partial_j v)_K 
     -
     (\delta \Fh_w,  v)_{\d K}
     \right].
\end{equation}
When the domain is the often used vertex patch $\ov$, we use the
abbreviated notation
$ (w, v)_\vtx = (w, v)_\ov = \int_\ov w \cdot v.  $ Using it, we
define $\Mco: \Lv \to \Lv$, and
$\Mcs: \Lv \to \Lv$ by
\begin{align}
  \label{eq:M0}
  (\Mco w,   v)_\vtx
  & =
    (\Gc w, v)_\vtx  - 
    \sum_{j=1}^N ((\partial_j \vpbot)\, \Lc j w, v)_\vtx,
  \\ \label{eq:M1}
  (\Mcs w,  v)_\vtx
  & =
    \sum_{j=1}^N (
    (\partial_j \delta)\, \Lc j w, v)_\vtx
\end{align}
for all $w, v \in \Lv$. Let $M(\tau) = M_0 - \tau M_1$. (We will often abbreviate $M(\tau)$  to
simply $M$.)  Using these definitions,
we may now rewrite~\eqref{eq:DG}
succinctly as 
$(\d_{\hat t} (M(\hat t) \hat u_h), v)_\vtx = a( \hat u_h, v).$
Let $\| v \|_\vtx = (v, v)_\vtx^{1/2},$ and for any operator
$\mathcal{O}$ on $\Lv$, let 
\begin{equation}
  \label{eq:20}
\|  \mathcal{O}\|_\vtx = 
\sup_{v, w \in \Lv} 
\frac{(\mathcal{O} v, w)_\vtx}{ \| v \|_\vtx \| w \|_{\vtx}}.  
\end{equation}

\begin{lemma}
  \label{lem:Mspd}
  The causality condition implies that $M(\tau)$ is a selfadjoint
  positive definite operator on $\Lv$ for any $0 \le \tau \le 1$ and
  that there is a mesh-independent constant $C_{\mathcal{L}, c}$
  (depending on $\Lc j$ and $c$) such that
  \begin{equation}
    \label{eq:23}
    \left(1 - \frac{c}{\cmax}
    \right) (\Gc w, w)_\vtx
    \le (M w, w)_\vtx
    \le C_{\mathcal{L}, c} (\Gc w, w)_\vtx
  \end{equation}
  holds for all $w \in \Lv$. Moreover, 
  \[
    \max\big(\| M_0\|_\vtx, \| M_0^{-1}\|_\vtx,
    \| M_1\|_\vtx, \|M\|_\vtx,
    \|M^{-1}\|_\vtx \big)
    \lesssim 1.
  \]
\end{lemma}
\begin{proof}
  Let $\vphi = \tau \vptop + (1- \tau) \vpbot$. 
  Since 
  \begin{equation}
    \label{eq:3}
    (M(\tau) v, w)_\vtx
    = 
    (\Gc v,  w)_\vtx - \sum_{j=1}^N
    ((\partial_j \vphi)\, \Lc j v, w)_\vtx
  \end{equation}
  and $\d_j \varphi$ is constant on each element, 
  the selfadjointness is immediate from the symmetry of $\Lc j$ and
  $\Gc$.
  It remains to prove the stated positive definiteness.
  In accordance with~\eqref{eq:D}, let 
  $\Dcnu = \sum_{j=1}^N \nu_j \Lc j.$
  Recall~\cite[p.~53]{Dafer10} that hyperbolicity implies the existence of
  real eigenvalues $\lambda_i$ and accompanying eigenvectors $e_i$
  (forming a complete set) satisfying
  $
    \Dcnu e_i = \lambda_i \Gc e_i
  $
  for any  unit vector $\nu \in \RRR^N$.  Since
  $\Dcnu $ and $\Gc$ are symmetric, it is
  easy to see that the eigenvectors $e_i$ must be orthogonal in the
  $\ip{x, y}_\Gc = \Gc x \cdot y$ inner product. Expanding any vector $v
  \in \RRR^L$ in the eigenbasis $e_i$ as follows,
  \[
    v =
    \sum_{j=1}^L v_i e_i \quad \text{ with } v_i = \ip{v, e_i}_\Gc,
  \]
  and recalling  that the maximal wave speed $c$ is the maximum of
  all such $|\lambda_i|$, 
  \begin{align*}
    \Dcnu v \cdot v 
    & = \sum_{i=1}^L v_i \lambda_i \Gc e_i \cdot v
      = \sum_{i=1}^L  \lambda_i \, |v_i|^2
    \le c \sum_{i=1}^L |v_i|^2 = c \,\ip{v, v}_\Gc.
  \end{align*}
  Using this inequality with $\nu = (\gradx \vphi) / \| \gradx
  \vphi \|_2$, we have 
  \begin{align}
    \nonumber
    \Gc v \cdot  v -  \sum_{j=1}^N (\partial_j \vphi)   \Lc j v
    \cdot v
    & =
      \Gc v \cdot  v -  
      \|\gradx  \vphi\|_2 \;
      \Dcnu v \cdot  v
    \\ \label{eq:4}
    & \ge \left(
      1 -  \|\gradx  \vphi\|_2  \,c \right)
      \; \Gc v \cdot  v. 
  \end{align}
  Since $\varphi$ is a convex combination of $\vpbot$ and $\vptop$,
  both of which satisfy the causality
  condition~\eqref{eq:causalitycondition}, we have
  $\| \gradx \vphi \|_2 \le 1/\cmax$.  Applying this, after
  using~\eqref{eq:4} in~\eqref{eq:3}, the proof of the lower bound
  of~\eqref{eq:23} is finished. The upper bound is a consequence of
  the boundedness of the $\Lc j$ and $\Gc$. Finally, the stated
  operator norm bounds on $M(\tau),$ $M_0 = M(0)$, and their inverses
  follow immediately from~\eqref{eq:23}.   The estimate for the
  operator norm of $M_1$ also
  follows easily since $|\d_j \delta| \lesssim 1$.
\end{proof}

Let $\Fc^\vtx$ denote the set of facets (i.e., $(N-1)$-subsimplices)
of the simplicial mesh $\ovh$ of the vertex patch $\ov.$ This set is
partitioned into the collection of facets on the boundary
$\partial \ov$ of the vertex patch, denoted by $\Fcvb$, and the
remainder, denoted by $\Fcvi$, the set of interior facets of
$\ovh$. We assume that each facet $F$ of the entire spatial mesh $\oh$
is endowed with a unit normal $n_F$ whose orientation is arbitrarily
fixed, unless if $F$ is contained in the global boundary
$\partial\om$, in which case it points outward. Then, for any
$x \in F$, set
$\jmp{u}_F(x) = \lim_{\veps \to 0} u( x + \veps n_F) - u(x - \veps
n_F)$. Note that $\jmp{u}_F$ agrees with the previously defined jump
$\jmp{u}$ on element boundaries, except possibly for a sign. Let
\[
  d(w, v) = -\left[
    a(w, v) + a(v, w) + (M_1 w, v)_\vtx \right]
\]
for $w, v \in {\Hhv}$. The first identity of the
next lemma shows that $d(w, w) \ge 0$
due to~\eqref{eq:Bpd} and \eqref{eq:Spd}.

\begin{lemma}
  \label{lem:intgparts}
  For all $v, w \in {\Hhv}$, 
  \begin{align}
    \label{eq:8}
    d(w, w) 
    & =
    \sum_{F \in \Fcvi} 2 \big(\delta\, \Sc \jmp{w}_F,  \jmp{w}_F\big)_F
    +
      \sum_{F \in \Fcvb}  (\delta \,B w, w)_F,
    \\
    \label{eq:2}
    -a(v, w)
    & = \sum_{K \in \ovh} 
      ( \divx( \delta f(v)), w)_K
    \\\nonumber 
    & + \sum_{F \in \Fcvi}
      \left[
      (\delta \DcnF \jmp{v}_F, \{ w\})_F +
      (\delta S \jmp{v}_F, \jmp{w}_F)_F\right]
      \\ \nonumber 
    &
      - 
      \sum_{F \in \Fcvb}
      \frac 1 2 (\delta (\Dc - B) v, w)_F.
  \end{align}
\end{lemma}
\begin{proof}
  Integrating by parts on an element $K \in \ovh$,
  \[
     \sum_{j=1}^N (\delta\, \Lc j w,  \d_j w)_K =
     (\delta \Dc w,   w)_{\d K}
    - \sum_{j=1}^N (\d_j (\delta \Lc j w),  w)_K.
  \]
  Applying the product rule to expand the derivative in the last
  term, using \eqref{eq:divL0}, and the symmetry of $\Lc j$, we obtain 
  \begin{equation}
    \label{eq:5}
    2 \sum_{j=1}^N (\delta \Lc j w, \d_j w)_K
    =
    -\sum_{j=1}^N  ( w, (\d_j \delta) \Lc j w)_K
    +
    (\delta w,  \Dc w)_{\d K}.
  \end{equation}
  Using this in the first term of the definition of
  $a(w,w)$, we have
  \begin{align*}
    d(w, w)
    & = -2 \,a(w, w)
      -
      \sum_{j=1}^N (\partial_j \delta\; \Lc j w  , w)_\vtx
      && \hspace{-4em}\text{by~\eqref{eq:M1}},
    \\
    & =
      \sum_{K \in \ovh}
      \big[ -( \delta \, w,  \Dc w)_{\d K}
      + 2 (\delta \,w,  \Fh_w)_{\d K}\big]
    && \hspace{-4em}\text{by~\eqref{eq:6} and~\eqref{eq:5},}
    \\
    & = \sum_{K \in \ovh}
      \big[-(\delta \, w,  \Dc w) _{\d K \cap \d\om}
      + (\delta \,w,  (\Dc + B) w)_{\d K \cap \d\om} \big]
    \\
    & + \sum_{K \in \ovh}
      \big[
      -(\delta \, w, 
         \Dc w)_{\d K \setminus \d\om}
      +( \delta \, w,
      2\Dc \{ w \} + 2\Sc \jmp{w} )_{\d K \setminus \d\om}\big],
  \end{align*}
  where we used the definition of the numerical flux in
  \eqref{eq:NumFlux}, splitting the right hand side sum into two to 
  accomodate  the two cases in \eqref{eq:NumFlux}.  
  The first sum,
  when rewritten using boundary facets, 
  immediately yields the last term of~\eqref{eq:8}
  since 
  $\delta=0$ on $\d\ov \setminus \d\om$.
  The second sum can
  be 
  rearranged to a sum over interior facets
  $F \in \Fcvi$, where $\jmp{\DcnF}_F=0$ due
  to~\eqref{eq:divL0}, which allows for further
  simplifications, eventually yielding the
  other term on the right hand side of~\eqref{eq:8}.

  The proof of~\eqref{eq:2} involves a  similar integration by parts
  starting from~\eqref{eq:6} and a similar
  rearrangement of sums over element boundaries
  to sums over facets.
\end{proof}

We  use  $C^s(0, \tau, X)$, for some Banach space $X$,
to denote the $X$-valued
function $w: [0, \tau] \to X$
that is $s$ times continuously differentiable. For any
$v, w \in C^1(0, 1, {\Hhv})$, define
\[
  b_\tau(v, w) =
  \int_0^\tau
  \big(\dth [ M(\hat t) v(\hat t)], w({\hat t})\big)_\vtx \; d\hat t
  \;-
  \int_0^\tau a(v({\hat{t}}), w({\hat{t}}))
  \; d\hat t.
\]
Note that the temporal snapshots $w(\hat t)$ and $v(\hat t)$ used
above, being  in $\Hhv$,  are admissible as arguments of the form $a$.
For any $z \in \Lv$, define
$
  \| z \|_{M(\tau)} = \left(M(\tau) z, z \right)_\vtx^{1/2}.
$
This is a norm due to  Lemma~\ref{lem:Mspd}.

\begin{lemma}
  \label{lem:bdiag}
  For all $w \in C^1(0, 1, {\Hhv})$,
  \[
    2\, b_\tau(w, w)
    = \| w (\tau)\|_{M(\tau)}^2 - \| w (0) \|_{M(0)}^2
    + \int_0^\tau d(w({\hat{t}}), w({\hat{t}}))\; d\hat t.
  \]
\end{lemma}
\begin{proof}
  The proof relies on a simple but key identity, which is best
  expressed writing $M$ for $M(\hat t) = M_0 - \hat t M_1$, as
  follows:
  \begin{align*}
    \frac{d }{d \hat t}
    \int_{\ov} M w \cdot w
    =
    \int_\ov 2
    \frac{\d }{\d \hat t}(M w) \cdot w -
    \int_\ov 
    \frac{d M}{d\hat t} w \cdot w.
  \end{align*}
  It implies
  $ 2\left( \dth [M(\hat t) w], w\right)_\vtx = \dth ( M(\hat t) w,
  w)_\vtx - (M_1 w, w)_\vtx.$ Using this in the definition of
  $b_\tau$, we obtain 
  \begin{align*}
   2 b_\tau(w, w)
    & =
      \int_0^\tau
      \left[
      \frac{d }{d \hat t}\,
      \big( M(\hat t) w({\hat{t}}), w({\hat{t}})\big)_\vtx
      - \big(M_1 w({\hat{t}}), w({\hat{t}})\big)_\vtx
      - 2\,a(w({\hat{t}}), w({\hat{t}}))\right]
       d\hat t
    \\
    & = (M(\tau) w, w)_\vtx - (M(0) w, w)_\vtx
      - \int_0^\tau \left[
      2\,a(w({\hat{t}}), w({\hat{t}}))
      + (M_1 w({\hat{t}}), w({\hat{t}})) \right]\, d\hat t
  \end{align*}
  so the result follows from the definition of $d(\cdot, \cdot)$.
\end{proof}

\subsection{Stability on spacetime surfaces}

We will first establish a bound on the exact solution on spacetime
tents (Proposition~\ref{prop:exact_stability}), which will then serve
as motivation for our approach to proving stability
(Lemma~\ref{lem:stability_semidiscrete}). Let $\dtop \Tv, \dbot \Tv$
and $\dbdr \Tv$ denote the top, bottom, and boundary parts,
respectively, of the boundary of a tent $\Tv$, i.e.,
\begin{gather*}
  \dtop \Tv  = \{ (x, t) \in \d\Tv: \; t = \vptop(x) \}, 
  \qquad 
  \dbot \Tv  = \{ (x, t) \in \d \Tv:  \; t = \vpbot(x) \}
  \\
  \dbdr \Tv  =
  \{ (x, t) \in \d \Tv:  \; (x, t) \text{ is neither in }
  \dtop\Tv \text{ nor in } \dbot\Tv \}.
\end{gather*}
Note that $\dbdr \Tv $ is empty whenever $\d\ov$ does not
intersect $\d\om$.  The next result shows that the
solution on $\dtop \Tv$ can be bounded, in a
tent-specific norm, by that on $\dbot\Tv$.
Specifically, defining
\begin{equation}
  \label{eq:41}
  \| w \|_{\dtb \Tv}^2
  = \int_\ov
  \big[
  g( w(x, \vptb(x))) - f(w(x, \vptb(x)))\, \gradx\vptb
  \big] \cdot w(x, \vptb(x)),  
\end{equation}
for 
$\tb \in \{\mathrm{top}, \mathrm{bot}\}$, 
it follows from the next result that
$\| u \|_{\dtop \Tv} \le \| u \|_{\dbot \Tv}$ (because 
 $\| u \|_{\dtop \Tv} $ and $\| u \|_{\dbot \Tv}$ coincide
with $\| \uh \|_{M(1)}$ and $\| \uh \|_{M(0)},$ respectively).

\begin{proposition} \label{prop:exact_stability}
  On a spacetime tent $\Tv$ satisfying causality,
  suppose a solution $u$ of
  \begin{subequations}
    \label{eq:tenteqforu}
    \begin{align}
      \dt g(u) + \divx f(u) & = 0 && \text{ in } \Tv,
      \\
      \label{eq:7}
      (\Dc - \Bc) u & = 0 &&\text{ on } \dbdr \Tv,
    \end{align}
  \end{subequations}
  is smooth enough for $\hat u = u \circ \vPhi$ to be in
  $C^1(0, 1, {\Hv})$. Then $\hat u = u \circ \vPhi$
  at pseudotime~$\tau,$ for any $0 \le \tau \le 1,$ satisfies
  \[
    \| \hat u (\tau)\|_{M(\tau)} \le \| \hat u (0) \|_{M(0)}.
  \]
\end{proposition}
\begin{proof}
  Since $\hat u$ satisfies the mapped equation
  $\dth (g(\hat u) - f(\hat u) \gradx \varphi)
  + \divx\big[\delta f(\hat
  u)\big]=0$, we have
  $ (\dth [M(\hat t) \hat u], v)_\vtx + (\divx(\delta f(\hat u)),
  v)_\vtx = 0 $ for all $v \in {\Hv}$. Now, observe that
  \[
    -a(\hat u({\hat{t}}), v)    =
    \big(\divx\big[\delta f(\hat u({\hat{t}}))\big], v\big)_\vtx 
  \]
  by Lemma~\ref{lem:intgparts}: indeed, the jumps in \eqref{eq:2} of
  the lemma vanish when applied to $\hat u({\hat{t}})$ since
  $\hat u({\hat{t}}) \in {\Hv}$, and moreover, the last term of
  \eqref{eq:2} also vanishes due to \eqref{eq:7} and \eqref{eq:kernelB}.
  Thus,
  \[
  (\dth [M(\hat t) \hat u], v)_\vtx
    - a(\hat u({\hat{t}}), v) = 0
  \]
  for all $v \in {\Hv}$ and each $0 \le \hat t\le 1$.  Integrating
  over $\hat t$ from 0 to $\tau$, we obtain
  \begin{equation}
    \label{eq:exactu-bform}
    b_\tau(\hat u, w) = 0, 
  \end{equation}
  for all $w \in C^1(0, 1, {\Hv}).$ Choosing $w = \hat u$ and
  applying Lemma~\ref{lem:bdiag}, we have
  \[
    \| \hat u (\tau)\|_{M(\tau)}^2 - \| \hat u (0) \|_{M(0)}^2
    + \int_0^\tau d(\hat u({\hat{t}}), \hat u({\hat{t}}))\; d\hat t = 0.
  \]
  Finally, we apply~\eqref{eq:8} of Lemma~\ref{lem:intgparts}. Noting
  that $\jmp{\hat u({\hat{t}})}_F = 0$ on all interior facets $F$ and
  recalling the positivity assumption~\eqref{eq:Bpd} on $B$, we
  complete the proof.
\end{proof}

\begin{definition}[Semidiscrete flow: $\Rsem_h(\tau)$]
  \label{def:semidiscrete}
  For any $0\le \tau \le 1$, define $\Rsem_h(\tau): \Vhv \to \Vhv$
  as follows. Given a $v_h^0 \in \Vhv$, let $v_h \in C^1(0, 1, \Vhv)$
  solve
  \begin{align}
    \label{eq:9v}
    \left(
    \dth\big[g(v_h) - f(v_h)\gradx
    \varphi\big],      w\right)_\vtx
    & = a(v_h(\hat t), w),  && 0 < \hat t \le 1,
    \\ \nonumber 
    v_h( 0) & = v_h^0,         && \hat t = 0,  
  \end{align}
  for all $w \in \Vhv$.  Set $\Rsem_h(\tau) v_h^0$ to $v_h(\tau)$. (In
  particular $\Rsem_h(0) v_h^0 = v_h(0) = v_h^0$.)
\end{definition}

\begin{lemma}[Stability of semidiscretization]
  \label{lem:stability_semidiscrete}
  For any $0 \le \tau \le 1$, and any $v \in \Vhv$, 
  \[
    \| \Rsem_h  (\tau) v\|_{M(\tau)} \le \| v \|_{M(0)}.
  \]
\end{lemma}
\begin{proof}
  The argument is similar to the proof of
  Proposition~\ref{prop:exact_stability}.  Let
  $v_h(\tau) = \Rsem_h(\tau) v$.  We need to bound $v_h(\tau)$ by
  $v_h(0) = v$.
  Replacing $w$
  in~\eqref{eq:9v} by a time-dependent test function
  $\tilde v \in C^1(0, 1, \Vhv)$ and integrating over $\hat t$ from $0$ to
  $\tau$, we have
  \[
    \int_0^\tau \bigg[ (\dth [M(\hat t) v_h], \tilde v)_\vtx
    - a(v_h({\hat{t}}), \tilde v(\hat t)) \bigg] \, d\hat t= 0,
  \]
  or equivalently,
  $
    b_\tau(v_h, \tilde v) = 0, 
    $ for all $ \tilde v \in C^1(0, 1, \Vhv).
  $
  Now, choosing $\tilde v = v_h$ and applying Lemma~\ref{lem:bdiag},
  we find that
  \[
    \| v_h (\tau) \|_{M(\tau)}^2 =
    \| v_h (0) \|_{M(0)}^2  - \int_0^\tau
    d(v_h({\hat{t}}), v_h({\hat{t}})) \; d\hat t.
  \]
  Since $d(v_h ({\hat{t}}), v_h({\hat{t}}))  \ge 0$ by
  \eqref{eq:8} of Lemma~\ref{lem:intgparts}, the proof is complete.
\end{proof}

\subsection{Local error in a  tent}

To estimate the error in the semidiscrete solution, we use, like
previous authors~\cite{CockbShu98}, the spatial $L^2$ projection into
the DG space $\Vhv$.  Let $P_h : \Lv \to \Vhv$ be defined by
$(P_h v, w)_\vtx = (v, w)_\vtx$ for all $v \in \Lv $ and $w \in \Vhv$.
Define
\[
  |v |_d = d(v, v)^{1/2}, \qquad v \in {\Hhv}.
\]
This is a seminorm by \eqref{eq:8} of Lemma~\ref{lem:intgparts} and
our assumptions~\eqref{eq:Bpd} and~\eqref{eq:Spd}.  Let
$h_K = \diam K$ for any spatial element $K$. The next lemma also uses
the broken Sobolev space $H^{s}(\ovh) = \Pi_{K \in \ovh} H^s(K)$, and
\[
  \hv = \max_{K \in \ovh} h_K, 
  \qquad
  |w|_{H^s(\ovh)^L}^2 = \sum_{K\in \ovh} |w|_{H^s(K)^L}^2.
\]

\begin{lemma}
  \label{lem:a_proj_bd}
  If $w \in H^{l}(\ovh)^L$ for some $1 \le l \le p+1$, then for any $v_h
  \in \Vhv$, 
  \[
    a(w - P_h w, v_h)\, \lesssim\, \hv^{l}\, |w|_{H^{l}(\ovh)^L } |v_h|_d.
  \]
\end{lemma}
\begin{proof}
  Let $e = w - P_hw$. Then the first term on the right hand side of 
  \begin{align*}
    a(e, v_h)
    & = \sum_{K \in \ovh}
      \sum_{j=1}^N (\delta \Lc j e, \partial_j v_h)_K 
     -
      \sum_{K \in \ovh}     (\delta \Fh_e,  v_h)_{\d K}
  \end{align*}
  must vanish, because $\delta \d_j v_h|_K$ is a polynomial of degree
  at most $p$ and $\Lc j$ is constant due to
  assumption~\eqref{eq:G_L_const}. Hence
  \begin{align*}
    a(e, v_h)
    & =
      -\sum_{K \in \ovh}
      (\delta (\Dc \{e\} + S\jmp{e}), v_h)_{\d K \setminus \d\om}
      +
      \frac 1 2 (\delta  ( \Dc + B ) e, v_h)_{\d K \cap \d \om}
    \\
    & =
      \sum_{F \in \Fcvi}
      (\delta \DcnF \{e\}, \jmp{v_h}_F)_F
      - (\delta S\jmp{e}_F, \jmp{v_h}_F)_F
      - \sum_{F \in \Fcvb}
      \frac 1 2 (\delta  ( \Dc + B ) e, v_h)_F
    \\
    & \lesssim
      \sum_{F \in \Fcvi} \int_F \delta \|e \|_2 \,\big|\jmp{v_h}\big|_S
      +
      \sum_{F \in \Fcvb} \int_F \delta \|e \|_2 \,|v_h|_B
  \end{align*}
  due to assumptions~\eqref{eq:Sbound}, \eqref{eq:11},  and~\eqref{eq:12}.
  On any facet $F$ adjacent
  to an element $K$,  by
  shape regularity and the
  well-known properties of $L^2$ projectors, 
  $
    \hv^{1/2} \| e \|_F \lesssim \hv^{l} | w |_{H^{l}(K)^L}.
  $
  Since
  $\delta \lesssim \hv$, the result now follows after applying
  Cauchy-Schwarz inequality and~\eqref{eq:8} of
  Lemma~\ref{lem:intgparts}.
\end{proof}

The next lemma provides control of the local error at any pseudotime
$\tau$ in terms of the initial error. To measure the regularity of
functions $w$ on a tent
$\Tv$, we find it convenient to use (semi)norms computed using the pull
back $w \circ \vPhi$ on $\Tvh$, defined by
\begin{equation}
  \label{eq:spatialseminorm}
  |w |_{\vtx, l} =
  \sup_{0 \le \tau \le 1} | (w \circ \vPhi) (\tau)|_{H^{l}(\ovh)^L}, 
  \qquad
  \|w \|_{\vtx, l} =
  \sup_{0 \le \tau \le 1} \| (w \circ \vPhi) (\tau)\|_{H^{l}(\ovh)^L}.
\end{equation}
Clearly, these are  bounded when $\hat{w} = w \circ \vPhi$ is in
$C^0(0, 1, H^{l}(\ovh)^L)$.

\begin{lemma}[Local error bound]
  \label{lem:local_semidiscrete_error}
  Let $u$ be the exact solution of~\eqref{eq:tenteqforu} on a
  causal tent $\Tv$,
  $\uh = u \circ \vPhi \in C^1(0, 1, \Hv \cap
  H^{p+1}(\ovh)^L)$, and let
  $\uh_h(\tau) = \Rsem_h(\tau) \uh_h^0$ for any $\uh_h^0 \in \Vhv$. Then 
  \[
    \| \uh (\tau) - \uh_h (\tau) \|_{M(\tau)}
    \lesssim
    \| \uh (0) - \uh_h^0 \|_{M(0)}
    + \hv^{p+1}| u |_{\vtx, p+1}.
  \]
\end{lemma}
\begin{proof}
  Integrating~\eqref{eq:9v} of
  Definition~\ref{def:semidiscrete}, we see that the semidiscrete
  solution $\uh_h$ satisfies 
  $b_\tau(\uh_h,  w_h) = 0$ for all $w_h \in C^1(0, 1, \Vhv).$
  We have also shown that the exact solution $\uh$ satisfies a similar
  identity, namely~\eqref{eq:exactu-bform}. Subtracting these
  identities, we have
  \begin{equation}
    \label{eq:13}
    b_\tau(\uh - \uh_h, w_h)=0 \qquad
    \text{for all }w_h \in C^1(0, 1, \Vhv).    
  \end{equation}
  Let $e_h(x, t)$ in $C^1(0, 1, \Vhv)$ denote the function
  whose time slices are defined by
  $e_h (\tau) = \uh_h (\tau) - P_h \uh (\tau)$ for each $0 \le \tau \le 1.$
  Equation~\eqref{eq:13} implies that 
  $b_\tau(e_h, e_h) = b_\tau( \uh - P_h \uh, e_h)= b_\tau( e, e_h), $
  where we have set $e = \uh - P_h \uh$.  Therefore, together with
  Lemma~\ref{lem:bdiag}, we obtain 
  \begin{align*}
    \nonumber
    \frac 1 2
    &
      \left(
    \| e_h (\tau)\|_{M(\tau)}^2 - \| e_h (0) \|_{M(0)}^2
    + \int_0^\tau \big|e_h({\hat{t}})\big|^2_d\; d\hat t
      \right)
      =  b_\tau( e_h, e_h) = b_\tau( e, e_h)
    \\ \nonumber
    & = \int_0^\tau
      \big[ ( \dth[M(\hat t)e], e_h)_\vtx -
      a(e({\hat{t}}),
      e_h({\hat{t}}))\big]\, d\hat t
    \\  
    & =
      (M(\tau) e (\tau), e_h (\tau))_\vtx
      - (M_0 e (0),  e_h (0))_\vtx
      -\int_0^\tau \Big[
      (M(\hat t) e,   \dth e_h)_\vtx
      + a(e({\hat{t}}), e_h({\hat{t}}))
      \Big]\, d\hat t.
  \end{align*}
  Since $ \dth e_h$ is of degree at most
  $p$ on each element and 
  $
    ( M(\hat t) e,   \dth e_h)_\vtx
    = (\Gc e, \dth e_h)_\vtx - \sum_{j=1}^N((\d_j \vphi) \,\Lc j e, \dth
    e_h)_\vtx
    = 0
  $
  by~\eqref{eq:G_L_const} and the orthogonality property
  of the projection error.
  Applying Lemma~\ref{lem:a_proj_bd} to
  the last term,   
  \begin{align*}
    \| e_h &(\tau)\|_{M(\tau)}^2  - \| e_h (0) \|_{M(0)}^2
     + \int_0^\tau \big|e_h({\hat{t}})\big|^2_d\; d\hat t
      \;\lesssim \;
    \\
    & (M(\tau) e (\tau), e_h (\tau))_\vtx -
      (M(0) e (0), e_h (0))_\vtx
      + \hv^{p+1} \int_0^\tau |\uh({\hat{t}})|_{H^{p+1}(\ovh)^L} 
      |e_h (\hat t)|_d\;
    d\hat t.
  \end{align*}
  By Cauchy-Schwarz inequality in the inner product
  (see Lemma~\ref{lem:Mspd}) generated by $M(\tau)$, 
  \[
    (M(\tau) e (\tau), e_h (\tau))_\vtx
    \;\lesssim\; 
    \hv^{p+1} |\uh(\tau)|_{H^{p+1}(\ovh)^L}\| e_h (\tau) \|_{M(\tau)}, 
  \]
  which holds also when $\tau=0$.
  Further applications of Cauchy-Schwarz and Young's
  inequalities yield 
  \begin{align*}
    \| e_h &(\tau) \|_{M(\tau)}^2
     + \int_0^\tau \big|e_h({\hat{t}})\big|_d^2 \; d\hat t
    \\
    & \lesssim
    \| e_h (0) \|_{M(0)}^2 
    \\
    & +
    {\hv}^{2(p+1)}\!
      \left( |\uh(0)|_{H^{p+1}(\ovh)^L}^2
      + 
      |\uh(\tau)|_{H^{p+1}(\ovh)^L}^2
      + \int_0^\tau
      \big|\uh({\hat{t}})\big|_{H^{p+1}(\ovh)^L}^2\, d\hat t\right)
    \\
    &\lesssim
    \| e_h (0) \|_{M(0)}^2 + {\hv}^{2(p+1)}| u |_{\vtx, p+1}^2.
  \end{align*}
  Finally, using the well known error bounds for the $L^2$ projection
  and the triangle inequality, we obtain the result of the lemma.
\end{proof}

\subsection{Global error bound}

Recall the advancing front $C_i$ defined by~\eqref{eq:Ci}
and the layer $L_i$  defined by~\eqref{eq:Li}. We will use the
following ``$\TG$'' procedure several times in the sequel.

\begin{definition}[Tent propagators to global propagators: $\TG$]
  \label{def:T2G}
  Suppose we are given a collection of operators $\RR$, one for each
  tent. The element of $\RR$ corresponding to a tent $\Tv$ is an
  operator $R^{\Tv}: \Lv \to \Lv$, which we refer to as the given {\em
    tent propagator} on $\Tv$,
  or more precisely, on its preimage $\Tvh$. We think of $R^{\Tv}$
  as transforming functions given at the bottom of
  $\Tvh$ to functions at the  top of $\Tvh$ 
  by some specific discrete process or by the exact solution operator.
  To produce global propagation operators from the collection $\RR$,
  we start by mapping functions on $C_{i-1}$ to functions on $C_i$, or
  equivalently
  per the advancing front
  definition~\eqref{eq:Ci},
  by mapping functions of
  $(x, \vphi_{i-1}(x))$ to functions of $(x, \vphi_{i}(x))$.
  The {\em layer propagator} of the layer $L_i$ generated by $\RR$,
  denoted by $G^{i, i-1}: L^2(C_{i-1})^L \to L^2(C_i)^L$,
  is defined by first considering 
  points on the front $C_i$ which have not advanced in time, where 
  $G^{i, i-1} w$ simply coincides with $w$, and then considering
  the remaining points
  $(x, \vphi_i(x))$ on $C_i$ which
  are separated from $(x, \vphi_{i-1}(x))$
  on $C_{i-1}$ by a tent, say $\Tv$, where we 
  use the tent propagator of
  $\Tv$ (see Figure~\ref{fig:t2g}).
  The next formula states this precisely. 
  For any $w \in L^2(C_{i-1})^L$, 
  \[
    (G^{i, i-1} w)(x, \vphi_i(x))
    =
    \begin{cases}
      w(x, \vphi_{i-1}(x))
      &  \text{ at } x \in \om \text{ where } \vphi_i(x)  =
      \vphi_{i-1}(x),
      \\
      (R^{\Tv} \hat w_\vtx) (x)
      & \text{ if } x \in \ov \text{ for some } \vtx \in V_i,
    \end{cases}
  \]
  where $\hat w_\vtx (x) = w(x, \vphi_{i-1}(x))|_{\ov}$.
  Finally, for a pair $i, j$ with $i>j \ge 0$, the {\em global
      propagator} generated by $\RR$ is the operator
    $G^{i, j}: L^2(C_j)^L \to L^2(C_i)^L$, defined by
    \[
      G^{i, j} = G^{i, i-1} \circ G^{i-1, i-2} \circ \cdots\circ
      G^{j+1, j}.
    \]    
  Let $\TG$ denote this process of
  producing global propagators from a collection of
  tent propagators, i.e., we define 
  $\TG(i, j, \RR)$ to be the $G^{i, j}$ above.
\end{definition}

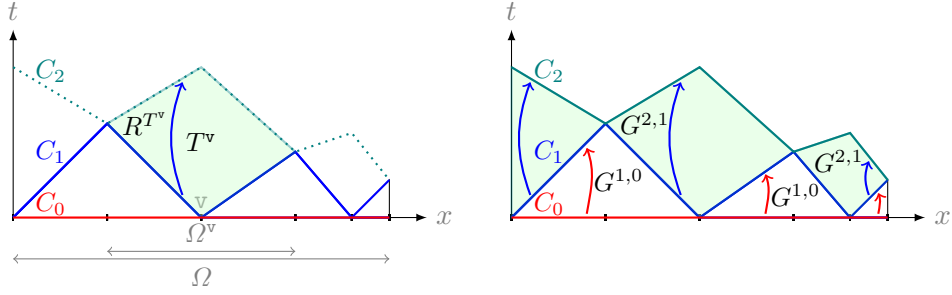
\begin{figure}
  \centering
  \begin{tikzpicture}[scale=2.5]
    \begin{scope}      
    \coordinate (a) at (0,0);
    \coordinate (b) at (0.5,0);
    \coordinate (c) at (1,0);
    \coordinate (d) at (1.5,0);
    \coordinate (e) at (1.8,0);
    \coordinate (f) at (2,0);
    \def\tao{0}
    \def\tbo{0}
    \def\tco{0}
    \def\tdo{0}
    \def\teo{0}
    \def\tfo{0}

    \draw[<->, gray] ($(a)-(0, 0.22)$)--($(f)-(0, 0.22)$)
    node[midway, below] {$\om$};

    \draw[<->, gray] ($(b)-(0, 0.18)$)--($(d)-(0, 0.18)$)
    node[midway, above] {$\ov$};

    \node[above, gray] at (c) {$\vtx$};
    
    \draw[mark=|,mark size=0.5pt,mark options={line width=.6pt},-latex]
    plot (a) -- plot (b) -- plot (c) -- plot (d) -- plot (e) -- 
    plot (f) -- (2.2,0.0) node[right, gray] {\large $x$};
    
    \draw[-latex] (a) -- (0.0,1.0) node[above, gray] {\large $t$};

    \def\tbo{0.5}
    \draw[blue, thick] let \p{a}=(a), \p{b}=(b), \p{c}=(c),
    \p{d}=(d), \p{e}=(e), \p{f}=(f) in
    plot coordinates {(\x{a},\tao)} -- plot coordinates {(\x{b},\tbo)} -- 
    plot coordinates {(\x{c},\tco)} -- plot coordinates {(\x{d},\tdo)} -- 
    plot coordinates {(\x{e},\teo)} -- plot coordinates {(\x{f},\tfo)};
    
    \def\tdn{0.35}
    \draw
    let \p{c}=(c), \p{d}=(d), \p{e}=(e) in
    plot (c) -- plot coordinates {(\x{d},\tdo)} -- plot (e) -- 
    plot coordinates {(\x{d},\tdn)} -- cycle;

    \def\tdo{0.35}
    \draw
    [blue, thick] let \p{a}=(a), \p{b}=(b), \p{c}=(c),
    \p{d}=(d), \p{e}=(e), \p{f}=(f) in
    plot coordinates {(\x{a},\tao)} -- plot coordinates {(\x{b},\tbo)} -- 
    plot coordinates {(\x{c},\tco)} -- plot coordinates {(\x{d},\tdo)} -- 
    plot coordinates {(\x{e},\teo)} -- plot coordinates {(\x{f},\tfo)};

    \def\tfn{0.2}
    \draw let \p{e}=(e), \p{f}=(f) in
    plot (e) -- plot coordinates {(\x{f},\tfn)} -- plot (f) -- cycle;

    \def\tfo{0.2}

    \draw
    [blue, thick] let \p{a}=(a), \p{b}=(b), \p{c}=(c),
    \p{d}=(d), \p{e}=(e), \p{f}=(f) in
    plot coordinates {(\x{a},\tao)} -- plot coordinates {(\x{b},\tbo)} -- 
    plot coordinates {(\x{c},\tco)} -- plot coordinates {(\x{d},\tdo)} -- 
    plot coordinates {(\x{e},\teo)} -- plot coordinates {(\x{f},\tfo)};
    
    \def\tcn{0.8}

    \draw
    [green!50!blue, thick, fill=green!20, opacity=0.4]
    let \p{b}=(b), \p{c}=(c), \p{d}=(d) in
    plot coordinates {(\x{b},\tbo)} -- plot coordinates {(\x{c},\tcn)} -- 
    plot coordinates {(\x{d},\tdo)} -- plot coordinates {(\x{c},\tco)}
    -- cycle;

    \node at ($(c)+(0,0.4)$)  {$T^\vtx$};

    \draw[blue, thick, ->] ($(c)+(-0.1,0.12)$)
    [out=110, in=-110] to  ($(c)+(-0.1, 0.72)$);
    
    \node at ($(c)+(-0.3,0.5)$)  {{$R^{T^{\vtx}}$}};

    \def\tao{0.8}
    \def\tbo{0.5}
    \def\tco{0.8}
    \def\ten{0.45}
    \def\teo{0.45}
    \draw
    [green!50!blue, dotted, thick] let \p{a}=(a), \p{b}=(b), \p{c}=(c),
    \p{d}=(d), \p{e}=(e), \p{f}=(f) in
    plot coordinates {(\x{a},\tao)} -- plot coordinates {(\x{b},\tbo)} -- 
    plot coordinates {(\x{c},\tco)} -- plot coordinates {(\x{d},\tdo)} -- 
    plot coordinates {(\x{e},\teo)} -- plot coordinates
    {(\x{f},\tfo)};

    \draw[red, thick] (a)--(f);

    \node[red] at ($(b)+(-0.3, 0.07)$)  {$C_0$};
    \node[blue] at ($(b)+(-0.3,0.35)$)  {$C_1$};
    \node[green!50!blue] at ($(b)+(-0.3,0.78)$)  {$C_2$};
  \end{scope}
  \begin{scope}[shift={(2.65,0)}]

    \coordinate (a) at (0,0);
    \coordinate (b) at (0.5,0);
    \coordinate (c) at (1,0);
    \coordinate (d) at (1.5,0);
    \coordinate (e) at (1.8,0);
    \coordinate (f) at (2,0);
    \def\tao{0}
    \def\tbo{0}
    \def\tco{0}
    \def\tdo{0}
    \def\teo{0}
    \def\tfo{0}

    \draw[mark=|,mark size=0.5pt,mark options={line width=.6pt},-latex]
    plot (a) -- plot (b) -- plot (c) -- plot (d) -- plot (e) -- 
    plot (f) -- (2.2,0.0) node[right, gray] {\large $x$};
    
    \draw[-latex] (a) -- (0.0,1.0) node[above, gray] {\large $t$};
    
    \def\tbo{0.5}
    \draw[blue, thick] let \p{a}=(a), \p{b}=(b), \p{c}=(c),
    \p{d}=(d), \p{e}=(e), \p{f}=(f) in
    plot coordinates {(\x{a},\tao)} -- plot coordinates {(\x{b},\tbo)} -- 
    plot coordinates {(\x{c},\tco)} -- plot coordinates {(\x{d},\tdo)} -- 
    plot coordinates {(\x{e},\teo)} -- plot coordinates {(\x{f},\tfo)};
    
    \def\tdn{0.35}
    \draw
    let \p{c}=(c), \p{d}=(d), \p{e}=(e) in
    plot (c) -- plot coordinates {(\x{d},\tdo)} -- plot (e) -- 
    plot coordinates {(\x{d},\tdn)} -- cycle;

    \def\tdo{0.35}
    \draw
    [blue, thick] let \p{a}=(a), \p{b}=(b), \p{c}=(c),
    \p{d}=(d), \p{e}=(e), \p{f}=(f) in
    plot coordinates {(\x{a},\tao)} -- plot coordinates {(\x{b},\tbo)} -- 
    plot coordinates {(\x{c},\tco)} -- plot coordinates {(\x{d},\tdo)} -- 
    plot coordinates {(\x{e},\teo)} -- plot coordinates {(\x{f},\tfo)};

    \def\tfn{0.2}
    \draw let \p{e}=(e), \p{f}=(f) in
    plot (e) -- plot coordinates {(\x{f},\tfn)} -- plot (f) -- cycle;

    \def\tfo{0.2}

    \draw
    [blue, thick] let \p{a}=(a), \p{b}=(b), \p{c}=(c),
    \p{d}=(d), \p{e}=(e), \p{f}=(f) in
    plot coordinates {(\x{a},\tao)} -- plot coordinates {(\x{b},\tbo)} -- 
    plot coordinates {(\x{c},\tco)} -- plot coordinates {(\x{d},\tdo)} -- 
    plot coordinates {(\x{e},\teo)} -- plot coordinates {(\x{f},\tfo)};

    \def\tcn{0.8}
    
    \draw
    [green!50!blue, thick, fill=green!20, opacity=0.4]
    let \p{b}=(b), \p{c}=(c), \p{d}=(d) in
    plot coordinates {(\x{b},\tbo)} -- plot coordinates {(\x{c},\tcn)} -- 
    plot coordinates {(\x{d},\tdo)} -- plot coordinates {(\x{c},\tco)}
    -- cycle;

    \def\ten{0.45}
    \draw
    [green!50!blue, thick, fill=green!20, opacity=0.4]
    let \p{d}=(d), \p{e}=(e), \p{f}=(f) in
    plot coordinates {(\x{d},\tdo)} -- plot coordinates {(\x{e},\ten)} -- 
    plot coordinates {(\x{f},\tfo)} -- plot coordinates {(\x{e},\teo)}
    -- cycle;

    \def\tao{0.8}
    \def\tbo{0.5}
    \def\tco{0.8}
    \def\ten{0.45}
    \def\teo{0.45}

    \def\tan{0.8}
    
    \draw [green!50!blue, thick, fill=green!20, opacity=0.4]
    let \p{a}=(a), \p{b}=(b) in
    plot (a) -- plot coordinates {(\x{a},\tan)} -- plot coordinates
    {(\x{b},\tbo)} -- cycle;

    \draw
    [green!50!blue,  thick] let \p{a}=(a), \p{b}=(b), \p{c}=(c),
    \p{d}=(d), \p{e}=(e), \p{f}=(f) in
    plot coordinates {(\x{a},\tao)} -- plot coordinates {(\x{b},\tbo)} -- 
    plot coordinates {(\x{c},\tco)} -- plot coordinates {(\x{d},\tdo)} -- 
    plot coordinates {(\x{e},\teo)} -- plot coordinates
    {(\x{f},\tfo)};

    \draw[red, thick] (a)--(f);
    \node[red] at ($(b)+(-0.3, 0.07)$)  {$C_0$};
    \node[blue] at ($(b)+(-0.3,0.35)$)  {$C_1$};
    \node[green!50!blue] at ($(b)+(-0.3,0.78)$)  {$C_2$};      
    \draw[blue, thick, ->] ($(c)+(-0.1,0.12)$)
    [out=110, in=-110] to  ($(c)+(-0.1, 0.72)$);    

    \draw[blue, thick, ->] ($(a)+(0.1,0.12)$)
    [out=110, in=-110] to  ($(a)+(0.1, 0.72)$);
    \draw[blue, thick, ->] ($(f)+(-0.1,0.12)$)
    [out=110, in=-110] to  ($(f)+(-0.1, 0.3)$);

    \node at ($(f)+(-0.26,0.3)$)  {{$G^{2, 1}$}};
    \node at ($(c)+(-0.29,0.5)$)  {{$G^{2, 1}$}};

    \draw[red, thick, ->] ($(b)+(-0.1, 0.02)$)
    [out=80, in=-70] to  ($(b)+(-0.1, 0.37)$);
    \draw[red, thick, ->] ($(d)+(-0.15, 0.02)$)
    [out=80, in=-70] to  ($(d)+(-0.15, 0.23)$);
    \draw[red, thick, ->] ($(f)+(-0.05, 0.02)$)
    [out=80, in=-70] to  ($(f)+(-0.05, 0.13)$);

    \node at ($(b)+(0.07,0.2)$)  {{$G^{1, 0}$}};
    \node at ($(d)+(0.01,0.12)$)  {{$G^{1, 0}$}};
    
  \end{scope}
  \end{tikzpicture}
  \caption{Schematic of a tent propagator (left) and
    two layer propagators (right).}
  \label{fig:t2g}
\end{figure}

For the semidiscretization, the tent propagator on $\Tv$ is the
operator 
$\Rsem_h(1) \circ P_h : \Lv \to \Vhv \subset \Lv$, set using the operator
$\Rsem_h(\tau)$ of Definition~\ref{def:semidiscrete}, 
evaluated at pseudotime $\tau=1$ (corresponding to the tent top).
Collecting these semidiscrete tent propagators into $\RR_h$ we use
Definition~\ref{def:T2G} to set the corresponding semidiscrete global
propagators $R_h^{i, j} = \TG(i, j, \RR_h).$ The exact propagator
$R^{i, j}$ is defined similarly, replacing $\Rsem_h$ by the exact
propagator of the hyperbolic system on tents (without projecting tent
bottom data), so that if $u(x,t)$ is the global exact solution of the
hyperbolic system on $\om \times [0, T]$, then
\begin{equation}
  \label{eq:16}
  R^{i, j}(u|_{C_j}) = u|_{C_i}.  
\end{equation}
The semidiscrete error propagation operators across layers can now be
defined by
\[
E_h^{i, j} = R^{i, j} - R_h^{i, j}.
\]
Letting $C_m$ denote the final front and $C_0$ the first, we are
interested in bounding the error at the final front, which is simply
$E_h^{m, 0} u^0$.  Setting $R^{0,0}$ and $R_h^{m, m}$ to the trivial
identity operators, we have the following lemma.

\begin{lemma}
  \label{lem:errorpropagation}
  \quad $
    \displaystyle{
    E_h^{m, 0} =
    \sum_{j=1}^m R_h^{m, j} E_h^{j, j-1} R^{j-1, 0}.
  }$
\end{lemma}
\begin{proof}
  Adding and subtracting $R_h^{m, m-1} \circ R^{m-1, 0}$, 
  \begin{align*}
    E_h^{m, 0}
    &  = R^{m, 0} - R_h^{m, 0}
      =
      R^{m, m-1} \circ R^{m-1, 0} 
      - R_h^{m, m-1} \circ R_h^{m-1, 0}
    \\
    & 
      = (R^{m, m-1} -  R_h^{m, m-1}) \circ R^{m-1, 0}
      + R_h^{m, m-1} \circ (R^{m-1, 0} -R_h^{m-1, 0}),
  \end{align*}
  i.e., 
  \[
    E_h^{m, 0} = E_h^{m, m-1} R^{m-1, 0} + R_h^{m, m-1} E_h^{m-1, 0}.
  \]
  The last term admits a recursive application of the same identity.
  Doing so $m-1$ times, the lemma is proved.
\end{proof}

Our global error analysis proceeds in a norm on advancing fronts
defined by
\[
  \| w \|_{C_i}^2
  = \int_\om
  \big[
  g( w(x, \vphi_i(x))) - f(w(x, \varphi_i(x)))\, (\gradx \vphi_i)(x)
  \big] \cdot w(x, \vphi_i(x)).
\]
Let $\| w \|_{C_i, \vtx}$ be defined by the same equality after
replacing the integral over $\om$ by integral over $\ov$.  Since the
first and last fronts, $C_0$ and $C_m$, respectively, are flat
\begin{equation}
  \label{eq:14}
  \| w \|_{C_0}^2 = (\Gc w(0), w(0))_\om
  \quad \text{ and }\quad 
  \| w \|_{C_m}^2 = (\Gc w(T), w(T))_\om,
\end{equation}
where, as before,  $T$ is the final time.

\begin{lemma}
  \label{lem:Rh_bound}
  For all $w \in L^2(C_j)^L$ and $i > j$, we have
  \[
    \| R_h^{i, j} w \|_{C_i} \le \| w \|_{C_j}.
  \]
\end{lemma}
\begin{proof}
  First consider the case $j = i-1$ and a $\vtx \in V_i$.
  Applying Lemma~\ref{lem:stability_semidiscrete} on 
  tent $\Tv$  in $L_i$,  we obtain that $\hat r_h = (R_h^{i, i-1} w)
  \circ \vPhi$ and $\hat w = w \circ \vPhi$ satisfies 
  \begin{equation}
    \label{eq:10}
    \| \hat r_h \|_{M(1)}^2 \le \| P_h \hat w\|_{M(0)}^2.
  \end{equation}
  By~\eqref{eq:G_L_const},
  \begin{align*}
    \| P_h \hat w \|_{M(0)}^2
    & =
      (\Gc P_h \hat w, P_h \hat w)_\vtx -
      \sum_{j=1}^N (\d_j \vpbot \Lc j P_h \hat w, P_h \hat w)_\vtx
    \\
    & = (M(0) P_h \hat w, \hat w)_\vtx
    \le \| P_h \hat w \|_{M(0)} \| \hat w \|_{M(0)},
  \end{align*}
  so~\eqref{eq:10} implies that $ \| \hat r_h \|_{M(1)}^2 \le
  \|\hat w\|_{M(0)}^2,$ which is   the same as
  $
    \| R_h^{i, i-1} w \|^2_{C_i, \vtx} \le \| w \|_{C_{i-1}, \vtx}^2.
  $
  Summing over $\vtx \in V_i$, we prove that
  \begin{equation}
    \label{eq:9}
    \| R_h^{i, i-1} w \|_{C_i} \le \| w \|_{C_{i-1}}.   
  \end{equation}
  Repeatedly
  applying this inequality on any further layers in between $i-1$ and $j$
  proves the lemma.
\end{proof}

In the subsequent statements of error estimates like in the next
theorem, we will tacitly assume that the exact solution
is smooth enough for
the seminorms on the right hand side to be finite.

\begin{theorem}[Error estimate for the semidiscretization]
  \label{thm:semidiscrete}
  Suppose $\om \times (0, T)$ is meshed by $m$ layers of tents
  satisfying the causality condition~\eqref{eq:causalitycondition}. 
  At the final time $T$,
  the difference between the exact solution $u(T)$ 
  and the semidiscrete MTP solution
  $u_h (T) \in V_h$ satisfies
  \[
    \| u (T)  - u_h (T)  \|_\om
    \;\lesssim\;
    \bigg(\sum_{j=1}^m h_j\bigg)^{1/2}
    \bigg(
    \sum_{j=1}^m \sum_{\vtx \in V_j} \hv^{2p+1} |u|_{\vtx, p+1}^2
    \bigg)^{1/2},
  \]
  where $h_j = \max_{\vtx \in V_j} \hv$.
\end{theorem}
\begin{proof}
  Let $u_j = u |_{C_j}$. Then, per~\eqref{eq:16},
  $R^{j-1, 0} u^0 = u_{j-1}$. Therefore, 
  \begin{align}
    \nonumber 
    \| u (T)  - u_h (T)  \|_\om
    & \lesssim \| u- u_h \|_{C_m} = \| E_h^{m, 0} u^0 \|_{C_m}
    && \text{by~\eqref{eq:14}}
    \\ \nonumber 
    & \le \sum_{j=1}^m \|  R_h^{m, j} E_h^{j, j-1} u_{j-1}
      \|_{C_m}
    && \text{by Lemma~\ref{lem:errorpropagation}}
    \\ \label{eq:15}
    & \le \sum_{j=1}^m \|  E_h^{j, j-1} u_{j-1}
      \|_{C_j}
    && \text{by Lemma~\ref{lem:Rh_bound}}.
  \end{align}
  Since the spatial projection of the support of $E_h^{j, j-1} u_j$
  can be subdivided into the union of non-overlapping
  vertex patches $\ov$ for all
  pitch vertices $\vtx \in V_j$,
  \[
    \|  E_h^{j, j-1} u_{j-1}\|_{C_j}^2
    = \sum_{\vtx \in V_j}  \|  E_h^{j, j-1} u_{j-1}\|_{C_j, \vtx}^2.
  \]
  On a tent $\Tv$ with $\vtx \in V_j$, note that
  $E_h^{j, j-1} u_{j-1}|_{\dtop\Tv} = (u_j - R_h^{j, j-1}
  u_{j-1})|_{\dtop\Tv} = u|_{\dtop\Tv} - R_h^{\Tv} (u|_{\dbot
    \Tv}\circ \Phi^{-1}).$ Putting $\uh = u|_\Tv \circ \vPhi$ and
  $\uh_h = R_h^{\Tv}\uh(0) = \Rsem_h(1) \circ P_h \uh(0)$,
  applying 
  Lemma~\ref{lem:local_semidiscrete_error} with
  $\tau=1$ and $\uh_h^0 = P_h \uh(0)$  yields
  \begin{align*}
    \|  E_h^{j, j-1} u_{j-1}\|_{C_j, \vtx}
    &=
      \| \uh(1) - \uh_h(1) \|_{M(1)}
      \\
    & \lesssim
      \| \uh(0) - P_h \uh(0)\|_{M(0)} + \hv^{p+1} | u|_{\vtx, p+1}
      \lesssim \hv^{p+1} | u|_{\vtx, p+1}.
  \end{align*}
  Using this in~\eqref{eq:15},
  \begin{align*}
    \| u (T)  - u_h (T)  \|_\om
    &
      \lesssim
      \sum_{j=1}^m
      \bigg(
      \sum_{\vtx \in V_j}
      \hv^{2p+2} | u|_{\vtx, p+1}^2
      \bigg)^{1/2}
      \lesssim
      \sum_{j=1}^m
      h_j^{1/2}
      \bigg(
      \sum_{\vtx \in V_j}
      \hv^{2p+1} | u|_{\vtx, p+1}^2
      \bigg)^{1/2},      
  \end{align*}
  so the proof is finished by applying the Cauchy-Schwarz inequality.
\end{proof}

\begin{remark}
  Note that $h_j$ may be interpreted as
  the ``layer height'' of $L_j$ due to the causality condition.
  Suppose 
  \begin{equation}
    \label{eq:32}
    \sum_{i=1}^m h_i \lesssim T.
  \end{equation}
  Theorem~\ref{thm:semidiscrete} then yields $O(h^{p+1/2})$-rate of
  convergence with $h = \max_\vtx \hv$.  Of course, \eqref{eq:32} can be
  violated by choosing very sparse layers (e.g., with one tent per
  layer), but this is not useful to get the best estimate from
  Theorem~\ref{thm:semidiscrete}, nor is it useful in practice:
  indeed, a large number of non-interacting
  tents (such as the tents of the same color in
  Figures~\ref{fig:tents:L1}--\ref{fig:tents:L4})  in each layer
  allows for
  better parallelism.  
\end{remark}

\begin{remark}
  \label{rem:weaker_stab}
  Suppose that instead of the operators $R_h^{i,i-1}$
  satisfying~\eqref{eq:9}, we are given operators
  $\Rt_h^{i, i-1}: L^2(C_{i-1})^L \to L^2(C_i)^L$ admitting the weaker
  stability bound
  \begin{equation}
    \label{eq:44}
    \| \Rt_h^{i, i-1} w\|_{C_i} \le (1 + \Cstab h_{i-1})\| w\|_{C_{i-1}}
  \end{equation}
  with some mesh and layer independent constant $\Cstab>0$ for all
  $w \in L^2(C_{i-1})^L$.  For any $i > j$, consider
  $ {\Rt_h}^{i, j} = {\Rt_h}^{i, i-1} \circ {\Rt_h}^{i-1, i-2} \circ
  \cdots\circ {\Rt_h}^{j+1, j}.$ Note that for any $i > j$, using the
  arithmetic-geometric mean inequality and the inequality
  $(1+ \alpha)^m \le e^{\alpha m}$,
  \begin{align*}
    (1 & +\Csta h_j)
    (1  +\Csta h_{j+1}) \cdots (1+ \Csta h_i)
    \le 
    \prod_{i=1}^m (1+\Csta h_i)
    \\
    & \le
    \left[\frac 1 m \sum_{i=1}^m (1+\Csta h_i)\right]^m
    \le
      \left[ 1 + \frac{\Csta}{m} \sum_{i=1}^m h_i\right]^m
      \le \exp\Big( {\Csta \sum_{i=1}^m h_i}\Big).
  \end{align*}
  Therefore, whenever~\eqref{eq:32} holds,  iterative
  application of~\eqref{eq:44}  gives the following layer-uniform
  bound for  any $i> j$:
  \[
    \| \Rt_h^{i, j} w\|_{C_i} \le \,e^{\Cstab T}\, \| w \|_{C_j}.
  \]
  Using this in place of Lemma~\ref{lem:Rh_bound}, the proof of
  Theorem~\ref{thm:semidiscrete} can be extended, replacing
  $R_h^{i, j}$ by $\Rt_h^{i, j}$, $E^{i, j}$ by
  $\Et^{i, j} = R^{i, j} - \Rt_h^{i, j}$, and ``$\le$''
  in~\eqref{eq:15} by ``$\lesssim$'' subsuming the $T$-dependent
  constant into the error estimates.
\end{remark}

\section{Analysis of fully discrete schemes} \label{sec:fully-discr}

In this section we use time stepping schemes to arrive at practical
fully discrete schemes from the semidiscretization studied in the
previous section.  Before studying these fully discrete schemes on a
mapped tent, it is useful to quickly make a few observations on the
time derivatives and  Taylor expansion of the exact solution.

\subsection{Preparatory observations}

The bilinear form $a(\cdot, \cdot)$ defines an operator $A$ from
$\Hvh$ to its dual space $(\Hvh)'$ in the usual way:
$(Aw)(v) = a(w, v)$ for $w, v \in \Hvh$.  Recall the previously
defined $L^2$ projector $P_h: \Lv \to \Vhv$. Since
$\Vhv \subset \Hvh$, the projector $P_h$ extends naturally from $\Lv$
to $(\Hvh)'$, so, e.g., $P_h A : \Hvh \to \Vhv$ satisfies
$(P_hA w, v_h)_\vtx = (A w)(v_h) = a(w, v_h)$ for all $w\in \Hvh$ and
$v_h \in \Vhv$.  While describing fully discrete schemes,
$A_h: \Vhv \to \Vhv$, defined by
$
  (A_h w, v)_\vtx  = a(w, v),$ for all $ w, v \in \Vhv$
will be useful. 
One may also consider $\At: {\Hv} \to \Lv$
defined by $ (\At w, v)_\vtx = (\divx \big[\delta f(w)\big], v)_\vtx,$
for all $ w \in {\Hv}$ and $v \in \Lv.$ It is easy to see
from~\eqref{eq:2} of Lemma~\ref{lem:intgparts} that $A$ coincides with
$\At$ on functions $w \in \Hv$ with $(\Dc - B )w=0$ on $\d\om$. In
particular, on such functions $w$, we may view $Aw $ as a function in
$\Lv$. The pull back $\uh$ of 
the exact hyperbolic solution $u$ from a tent $\Tv$ to
the cylinder $\Tvh$ is one such function. Therefore the following
equation holds in $\Lv$:
\begin{equation}
  \label{eq:21}
  \d_{\hat t} (M \hat u) = A \hat u, \qquad 0 \le \hat t\le 1.
\end{equation}

We will proceed  assuming that the exact solution
$\hat u$ is regular enough to admit
the Taylor expansion
\begin{equation}
  \label{eq:u-TaylorExp}
  \uh(\tau) = \sum_{k=0}^{s} \frac{\tau^k}{k!} \uh^{(k)}(0)
  + \rho_{s+1}(\tau),
\end{equation}
for some $s\ge 1$. Here, $\uh^{(k)}(\hat t)$ denotes the $k$th order
time derivative $ d^k \uh / d \hat t^k$ (which is a function in
${\Hv}$ when the solution is smooth---see
Lemma~\ref{lem:time-derivative-uhat} below), and the remainder term
$\rho_{s+1}(\tau)$ can be expressed as the $\Hv$-valued Riemann
integral
\begin{equation}
  \label{eq:33}
  \rho_{s+1}(\tau) = \frac{\tau^{s+1}}{s!} \int_0^1 (1-\hat t)^s
  \uh^{(s+1)}(\hat t \tau) \; d\hat t.  
\end{equation}
It is well known that the expansion~\eqref{eq:u-TaylorExp} holds for
$\tau$ in an interval containing $0$ whenever $\uh$ is $s+1$ times
continuously differentiable (as an ${\Hv}$-valued function) in that
interval. When applied to a spacetime hyperbolic solution $u$ in the
physical domain, the smallness of the higher order terms
in~\eqref{eq:u-TaylorExp} (written there in terms of the mapped
function $\hat u$), is evident from the following lemma, since
$\delta(x) \lesssim \hv$. 

\begin{lemma}
  \label{lem:time-derivative-uhat}
  The function $\hat u = u \circ \vPhi$ satisfies
  \[
    \uh^{(k)} = (\partial_t^k u \circ \vPhi)\, \delta^k.
  \]
  Consequently, at each pseudotime $\hat t$, within each spatial  element
  $K \in \ovh$, as a function of the spatial variable $x$,
  $\uh^{(k)}(\hat t)$ is as smooth as
  $(\partial_t^k u)(x, \vphi(x, \hat t))$.  Moreover,
  $\uh^{(k)}(\hat t)$ is in $\Hv$ if $\partial_t^k u$ is continuously
  differentiable in $\Tv$.
\end{lemma}
\begin{proof}
  Let $e$ denote the spacetime unit vector in the time direction i.e.,
  all its components are zero except for the last (time) component
  which is 1. Then, at some fixed spacetime point $\hat P$ in $\Tvh$,
  we may write
  $ \uh^{(k)}(\hat P) = D^k\uh(\hat P) (e, e, \ldots, e), $ where
  $D^k\uh$ is the multilinear form representing the $k$th order
  Fr\'echet derivative of $\uh$, and $e$ is repeated $k$ times in its
  argument list.  Then, letting $P = \vPhi(\hat P)$ denote the mapped
  point in $\Tv$, by standard arguments~\cite{Ciarl78} for affine
  maps,
  \begin{align*}
    \uh^{(k)}(\hat P)
    & = D^k ( u\circ \vPhi)(\hat P) (e, e, \ldots, e)
    \\
    & =    D^k u (P) ([\gradxt \vPhi]e, [\gradxt \vPhi]e, \ldots,
      [\gradxt \vPhi]e)
    \\
    & =    D^k u(P) (\delta e, \delta e, \ldots, \delta e),
  \end{align*}
  where we have used~\eqref{eq:gradxtPhi} in the last step. Since the last
  term above equals the product of $\delta^k$ and the derivative
  $\d^k u/ \d t^k$ at $P$, the proof is complete.
\end{proof}

In view of Lemma~\ref{lem:time-derivative-uhat}, when the exact
solution is smooth in the physical spacetime, we expect it to have the
following (semi)norms finite, in addition to the ones
in~\eqref{eq:spatialseminorm}:
\begin{equation}
  \label{eq:seminorms2}
  \begin{gathered}
  |w |_{\vtx, l, m} =
  \sup_{0 \le \tau \le 1}
  \sum_{k=0}^m 
  \left| \hat w^{(k)}(\tau) 
  \right|_{H^{l}(\ovh)^L}, 
  \\
  \|w \|_{\infty, \vtx}
  = \sup_{0 \le \tau \le 1}
  \| \hat w (\tau) \|_\vtx, \qquad 
  \|w \|_{s, \infty, \vtx}
  = \sum_{\ell=0}^s \| \d_t^\ell w\|_{\infty, \vtx}.
  \end{gathered}
\end{equation}
When $m=0$, the first seminorm coincides with the seminorm
in~\eqref{eq:spatialseminorm}.  The next result bounds the Taylor
remainder term in terms of the mapped time derivative
$\d_t^s u \circ \vPhi$.

\begin{lemma}
  \label{lem:rho-Taylor-remainder}
  The Taylor remainder term satisfies  $
  \| \rho_{s} \|_\vtx \lesssim \tau^s \hv^{s} \| \d_t^{s} u \|_{\infty, \vtx}.
  $
\end{lemma}
\begin{proof}
  Starting from~\eqref{eq:33}, by Fubini's theorem and
  Cauchy-Schwarz inequality,
  \begin{align*}
    \tau^{-2s} \| \rho_{s} \|_{\vtx}^2
    &\lesssim
      \int_0^1 (1 - \hat t)^{2s}
      \| 
      \uh^{(s)}( \hat t \tau) \|_\vtx^2\, 
      d\hat t
      \le \bigg(\sup_{0 \le \tau \le 1}
      \|      \uh^{(s)}( \tau) \|_\vtx^2 \bigg)
      \int_0^1 (1 - \hat t)^{2s}\, d \hat t
    \\
    & \lesssim \sup_{0 \le \tau \le 1}
      \| \delta^{s} (\d_t^{s} u \circ \vPhi)(\tau) \|_\vtx^2,
  \end{align*}
  due to Lemma~\ref{lem:time-derivative-uhat}. Since  $\delta \lesssim
  \hv,$ the result follows.
\end{proof}

\begin{lemma}
  \label{lem:exact-time-deriv}
  For any $k \ge 1,$ whenever the exact time derivative
  $\uh^{(k-1)}(0)$ exists in $\Hv$, we have
  $
    \uh^{(k)}(0) = M_0^{-1} (A + k M_1) \,\uh^{(k-1)}(0).
  $
\end{lemma}
\begin{proof}
  Differentiating both sides of~\eqref{eq:21} $k-1$ times,
  $(M \uh)^{(k)} (\hat t) = A \uh^{(k-1)}(\hat t)$. Simplifying the
  left hand side by Leibniz rule and the linearity of $M(\hat
  t)$, we have
  \[
    M(\hat t) \uh^{(k)}(\hat t) - k M_1 \uh^{(k-1)}(\hat t) =
    A \uh^{(k-1)}(\hat t).
  \]
  Evaluating  at $\hat t=0$ and rearranging, the proof is complete.
\end{proof}


Note that $\Vhv$ is an invariant subspace of the previously defined
operators $M_0$ and $M_1$, due to~\eqref{eq:G_L_const}. It will be
understood from context whether we consider $M_0, M_1$ as operators on
$L^\vtx$ or as operators on $\Vhv$.  For operators on $\Vhv$, we
define a discrete operator norm, analogous to \eqref{eq:20}, for
operators $\mathcal{O}_h$ on $\Vhv$, by
\[
  \|  \mathcal{O}_h\|_{\vtx, h} := 
    \sup_{v_h, w_h \in \Vhv} 
    \frac{( \mathcal{O}_hv_h, w_h)_\vtx}{ \| v_h \|_\vtx \| w_h \|_{\vtx}}
\]
for all $v_h, w_h \in \Vhv$.

\begin{lemma}
  \label{lem:AXRbounds}
  We have 
  $
    \|  A_h\|_{\vtx, h} \lesssim 1, \;
    \| M_1 \|_{\vtx, h}  \lesssim 1, \;
    \|M \|_{\vtx, h} \lesssim 1, \;
    \|M^{-1} \|_{\vtx, h} \lesssim 1.
  $
\end{lemma}
\begin{proof}
  To prove the first inequality,
  consider the terms that make up
  $a(v_h, w_h) = ( A_hv_h, w_h)_\vtx$
  for any $v_h, w_h \in \Vhv$. On  any $K \in \ovh$,
  since $\Lc j$ is uniformly bounded and $\delta \lesssim h_K$, 
  \[
    ( \delta \Lc j v_h,  \d_j w_h)_K
    \lesssim \| v_h \|_K h_K\| \d_j w_h \|_K
    \lesssim \| v_h \|_K \| w_h \|_K
  \]
  where we have applied an inverse inequality in the last step.
  Next, consider an element boundary term in $(A_hv_h, w_h)$, restricted
  to say a facet $F \subset \d K$, shared with 
  the boundary of another element $K_o$ in $\ovh$:
  \[
    (\delta\, \Dc \{v_h\}, w_h|_{\d K})_F
    \lesssim
    \left(
      h_{K} \| v_h \|_{\d K}^2 + 
      h_{K_o} \| v_h \|_{\d K_o}^2
    \right)^{1/2}
    \left( h_{K}^{1/2} \| w_h \|_{\d K}\right)
    \lesssim
    \| v_h \|_{\ov} \| w_h \|_K,
  \]
  where we have again used $\delta \lesssim h_K$ and local scaling
  arguments.  Continuing to use similar arguments on all the remaining
  terms that make up $( A_hv_h, w_h)_\vtx$, we obtain
  $\| A_h \|_{\vtx, h} \lesssim 1$.  Finally, Lemma~\ref{lem:Mspd} shows
  that $\| M_1 \|_{\vtx, h}$, $\|M(\tau)\|_{\vtx, h}$, and
  $\|M(\tau)^{-1}\|_{\vtx, h}$ also admit mesh-independent bounds.
\end{proof}

The projector $P_h$ enjoys the commutativity properties
\begin{equation}
  \label{eq:MPcommute}
  M_1 P_h = P_h M_1,
  \quad
  M_0 P_h = P_h M_0, 
\end{equation}
because of~\eqref{eq:G_L_const}. Although a similar commutativity
identity cannot be expected of $A_h$, we have the following lemma.

\begin{lemma}
  \label{lem:APcomm}
  For any $w \in H^{l}(\ovh)^L$, $1\le l \le p+1$, the function $\eta_h
  = (A_hP_h - P_h A) w$ satisfies
  $
    \| \eta_h \|_\vtx \,\lesssim\, \hv^l \, |w|_{H^l(\ovh)^L}.
  $
\end{lemma}
\begin{proof}
  Since
  $\| \eta_h \|_\vtx^2
  = (A_h P_h w, \eta_h)_\vtx - (P_h A w, \eta_h)_\vtx
  = a(P_h w, \eta_h)_\vtx - a( w, \eta_h)_\vtx$, 
  \begin{align*}
    \| \eta_h \|_\vtx^2
    & = a(P_h w - w, \eta_h)
      \,\lesssim\, \hv^l |w|_{H^l(\ovh)^L} |\eta_h|_d
  \end{align*}
  by Lemma~\ref{lem:a_proj_bd}.  By Lemma~\ref{lem:AXRbounds},
  $|\eta_h|_d^2 =  -((2 A_h+M_1) \eta_h, \eta_h)_\vtx
  \lesssim \| \eta_h \|_\vtx^2$, so the inequality of
  the lemma follows.
\end{proof}

\subsection{Lowest order tent-implicit scheme}
\label{ssec:lowest-order-tent-implicit}

While the overall MTP strategy is a tent-by-tent time-marching
strategy akin to explicit methods, within a mapped tent, one may
choose between explicit or implicit schemes.  By a ``tent-implicit''
scheme, we mean a method that solves the
semidiscretization~\eqref{eq:DG} on a mapped tent using implicit time
stepping.  Although this requires matrix inversion, the size of the
matrix is only as large as the number of spatial degrees of freedom in
one tent (much smaller than the size of the global matrix that needs
to be inverted in standard implicit schemes for method of lines
discretizations).  Numerical results using tent-implicit schemes  of
various orders were reported first in~\cite[\S5.4]{GopalSchobWinte17}. In
this subsection, we provide a convergence analysis of the lowest-order case.

To  derive the lowest order tent-implicit method, we begin by 
rewriting \eqref{eq:DG} in a form analogous to \eqref{eq:21}, i.e.,
\[
  \d_{\hat t} (M \uh_h ) = A_h \uh_h, \quad 0 \le \hat t \le 1.
\]
Then, putting $y_h = M \uh_h$, we have
$\partial_{\hat t}\,y_h = A_h M^{-1} y_h$. The implicit Euler
method applied to this defines
an approximation $y_{h1}(\tau)$ to $y_h(\tau)$
given by $ y_{h1}(\tau) - y_h(0) = \tau A_hM^{-1} y_{h1}(\tau).  $ Since
$\uh_h = M^{-1} y_h$, an approximation to $\uh_h(\tau)$ is furnished
by $M^{-1} y_{h1}(\tau)$, which after simplification becomes
$M^{-1} (I - \tau A_h M^{-1})^{-1} M_0 \uh(0).  $ This motivates the
following definition of the discrete propagator.

\begin{definition}[Lowest order tent-implicit flow: $\Rimp_{h1}(\tau)$]
  Define   $\Rimp_{h1}(\tau) : \Vhv \to \Vhv$ by 
\[
  \Rimp_{h1}(\tau) = 
  M(\tau)^{-1} (I - \tau A_h M(\tau)^{-1})^{-1}
  M_0.
\]
The two inverses required for this definition are both well defined:
first, 
$M$ is invertible by Lemma~\ref{lem:Mspd}; second, 
$I - \tau A_h M^{-1}$ is invertible because
\begin{align*}
  (I - \tau A_h M^{-1}) M
  & = M - \tau A_h
    = (M_0 -\frac {\tau}{2} M_1)  - \frac{\tau}{2}    (2A_h + M_1),
\end{align*}
together with Lemmas~\ref{lem:intgparts} and~\ref{lem:Mspd}
imply that for any $0 \ne v\in \Vhv$,
\[
  (  (I - \tau A_h M^{-1}) M v, v)_\vtx
  \ge \| v\|_{M(\tau/2)}^2  + | v|_d^2 > 0.
\]
\end{definition}

We proceed to prove convergence of the scheme, beginning with the next
stability result that closely resembles the inequality
of Lemma~\ref{lem:stability_semidiscrete}.

\begin{lemma}[Unconditional strong stability]
  \label{lem:stability-imp}
  For any $v \in \Vhv$ and any $0 \le \tau \le 1$, 
  \[
    \| \Rimp_{h1}(\tau) v \|_{M(\tau)}^2 
    \le \| v \|_{M_0}^2.
  \]
\end{lemma}
\begin{proof}
  Let  $v_\tau = \Rimp_{h1}(\tau) v$. Then $(M - \tau A_h) v_\tau =
  M_0 v.$ Taking the inner product with $v_\tau$ on both sides,
  \begin{align*}
    \| v_\tau \|_M^2
    & = (M_0v, v_\tau)_\vtx + \tau (A_h v_\tau, v_\tau)_\vtx
    \\
    & \le \frac 1 2 \| v \|_{M_0}^2 + \frac 1 2 \| v_\tau \|_{M_0}^2
      + \tau (A_h v_\tau, v_\tau)_\vtx
    \\
    & = \frac 1 2 \| v \|_{M_0}^2 + \frac 1 2 \| v_\tau \|_{M}^2
      + \frac{\tau}{2} ((2A_h + M_1) v_\tau, v_\tau)_\vtx.
  \end{align*}
  Now, since $((2A_h + M_1) v_\tau, v_\tau)_\vtx = - | v_\tau|_d^2$
  (see Lemma~\ref{lem:intgparts}), the proof is complete.
\end{proof}

When using any (spatial) polynomial degree $p \ge 0$, we obtain the
following bound for the lowest order method (showing that the rate is
limited by the time discretization error),  which uses the
(semi)norms defined in \eqref{eq:spatialseminorm}
and~\eqref{eq:seminorms2}.

\begin{lemma}[Local error bound]
  \label{lem:local-error-implicit}
  Let $\uh$ denote the exact solution on $\Tvh$ and let
  $\uhatimp_{h1}(\tau) = \Rimp_{h1}(\tau) \uh_h^0$ for some $\uh_h^0 \in \Vhv$.
  Then, 
  \begin{align*}
    \| \uh (\tau) - \uhatimp_{h1} (\tau) \|_{M(\tau)}
    & \;\lesssim\;
      \| \uh(0) - \uh_h^0 
      \|_{M(0)}
      +
      \hv
      \big(
      \|  u\|_{2, \infty, \vtx}
      + | u |_{\vtx, 1}
      \big).
  \end{align*}
\end{lemma}
\begin{proof}
  Let
  $X_h = M_0^{-1}(A_h + M_1)$ and $X = M_0^{-1} (A + M_1)$.
  By~\eqref{eq:MPcommute},  
  \begin{equation}
    \label{eq:18}
    P_h X - X_h P_h = M_0^{-1}(P_h A - A_h P_h).
  \end{equation}
  An alternate expression for the discrete propagator will also be
  useful: $\Rimp_{h1} = M^{-1}(I - \tau A_h M^{-1})^{-1} M_0 = (M-
  \tau A_h)^{-1} M_0 = (M_0 - \tau(A_h + M_1))^{-1} M_0,$ i.e.,
  \begin{equation}
    \label{eq:30}
    \Rimp_{h1} (\tau) = (I - \tau X_h)^{-1}.
  \end{equation}

  With these preparations, we derive an ``error equation'' for
  $\veps_h = \uhatimp_{h1}(\tau) - P_h \uh(\tau)$. 
  Note  that 
  $\veps_h$ is a function in $\Vhv$ for each $\tau$. Writing 
  \[
    \veps_h =  \Rimp_{h1} \big[ \uh_h^0 - P_h \uh(0)\big] + \phi_h,
  \]
  with $ \phi_h = \Rimp_{h1} P_h \uh(0) - P_h \uh(\tau),$ we analyze
  $\phi_h$ further as follows.
  \begin{align*}
    \phi_h
    & = (I - \tau X_h)^{-1} P_h \uh(0) - P_h \uh(\tau)
    && \text{by}~\eqref{eq:30},
    \\
    & = \big[(I - \tau X_h)^{-1} -I \big]P_h \uh(0)
      - \tau P_h \uh^{(1)}(0) - P_h \rho_2
      &&\text{by~\eqref{eq:u-TaylorExp},}
    \\
    & =
      \tau (I-\tau X_h)^{-1} X_h P_h \uh(0)
      - \tau P_h \uh^{(1)}(0) - P_h \rho_2
    \\
    & = \tau (I-\tau X_h)^{-1}
      \big[P_h X - M_0^{-1}(P_hA - A_h P_h)\big] \uh(0)
    \\
    &\quad       - \tau P_h \uh^{(1)}(0) - P_h \rho_2
    &&\text{by~\eqref{eq:18},}
    \\
    & = \tau^2
      X_h (I-\tau X_h)^{-1} P_h \uh^{(1)}(0)
      - \tau (I-\tau X_h)^{-1} M_0^{-1} \eta_h - P_h \rho_2
  \end{align*}
  with $\eta_h = (P_hA - A_h P_h)
  \uh(0)$. We have used Lemma~\ref{lem:exact-time-deriv} in the last
  step.

  To bound $\phi_h$, first note that by
  Lemma~\ref{lem:AXRbounds}, $\| X_h\|_{\vtx, h} \lesssim 1.$  Also,
  by \eqref{eq:30} and Lemma~\ref{lem:stability-imp},
  $\|\Rimp_{h1}(\tau)\|_{\vtx, h} = \| (I-\tau X_h)^{-1}\|_{\vtx, h}
  \lesssim 1$, so
  $\| \phi_h \|_\vtx
     \lesssim \tau^2 \| \uh^{(1)}(0) \|_\vtx 
      + \tau \|\eta_h\|_{\vtx} + \| \rho_2\|_\vtx.$
  Now, applying Lemmas~\ref{lem:time-derivative-uhat},
  \ref{lem:APcomm} and~\ref{lem:rho-Taylor-remainder}, 
  \begin{align*}
    \| \phi_h \|_\vtx
    & \lesssim \tau^2 \hv \|\d_t u \|_{\infty, \vtx}
      + \tau \hv | u |_{\vtx, 1}
      + \tau^2\hv^2 \| \d_t^2 u \|_{\infty, \vtx}.
  \end{align*}
  Together with the stability result of Lemma~\ref{lem:stability-imp},
  this proves
  \[
    \| \veps_h \|_M \lesssim \| \uh_h^0 - P_h \uh(0) \|_{M_0} +
    \hv
    \big(
    \| \d_t u\|_{\infty, \vtx}
    + \| \d_t^2 u\|_{\infty, \vtx}
    + | u |_{\vtx, 1}
    \big).
  \]
  Using the triangle inequality,
  $\| \uh(\tau) - \uhatimp_{h1}(\tau) \|_M \le \| \uh(\tau) - P_h \uh(\tau)
  \|_M + \| \veps_h \|_M,$ and the standard estimate for $L^2$
  projection,
  $\| \uh(\tau) - P_h \uh(\tau) \|_\vtx \lesssim \hv |u|_{\vtx, 1}$,
  the proof can now be completed.
\end{proof}

The previous two lemmas lead to a global convergence theorem, as we
shall now see.
The implicit scheme's tent propagator on $\Tv$ is the operator
$ \Rimp_{h1}(1) \circ P_h : \Lv \to \Vhv,$ set using $\Rimp_{h1}(\tau)$
evaluated at pseudotime $\tau=1$ corresponding to the tent top.
Letting $\RRhimp$ denote the collection of such tent propagators over
all tents, we use Definition~\ref{def:T2G} to set the global
propagator $\Rhimp^{i, j} = \TG(i, j, \RRhimp)$, and consider the
discrete solution $\uhimp = \Rhimp^{m, 0} u^0$ at the final time $T$.

\begin{theorem}[Error estimate for the lowest order tent-implicit
  scheme]
  \label{thm:lowestimplicit}
  Under the same conditions as Theorem~\ref{thm:semidiscrete}, for any
  spatial degree $p \ge 0$, the
  fully discrete solution $\uhimp$ satisfies
  \[
    \| u (T)  - \uhimp  \|_\om
    \;\lesssim\; 
    \bigg(\sum_{j=1}^m h_j\bigg)^{1/2}
    \bigg[
    \sum_{j=1}^m \sum_{\vtx \in V_j} \hv
          \big(
      \| u\|_{2, \infty, \vtx}
      + | u |_{\vtx, 1}
      \big)^2
    \bigg]^{1/2}.
  \]
\end{theorem}
\begin{proof}
  First, due to Lemma~\ref{lem:stability-imp}, we observe that in
  complete analogy with 
  \\
  Lemma~\ref{lem:Rh_bound}, one
  can prove that for $i> j$,
  \[
    \| \Rhimp^{i, j} w\|_{C_i} \le \| w \|_{C_j}.
  \]
  Defining
  $
    \Ehimp^{i, j} = R^{i, j} - \Rhimp^{i, j},
  $
  in analogy with Lemma~\ref{lem:errorpropagation}, we can show that
  \[
    \Ehimp^{m, 0} =
        \sum_{j=1}^m \Rhimp^{m, j} \Ehimp^{j, j-1} R^{j-1, 0}.
  \]
  Hence the theorem can be proved along the same lines as the proof of
  Theorem~\ref{thm:semidiscrete}, using
  Lemma~\ref{lem:local-error-implicit} in place of
  Lemma~\ref{lem:local_semidiscrete_error}.
\end{proof}

As before, under the further assumption that~\eqref{eq:32} holds,
Theorem~\ref{thm:lowestimplicit} gives an $O(h^{1/2})$ rate of
convergence for the solution at the final time.

\subsection{Lowest order explicit scheme}
\label{ssec:lowest-order-expl}

A perhaps nonstandard route to derive an explicit scheme is to view it
as an iterative method for solving the equations of an implicit
scheme. Pursuing this approach using the tent-implicit scheme of
\S\ref{ssec:lowest-order-tent-implicit}, we write 
$v_\infty = \Rimp_{h1} (\tau) v_0$, or equivalently,
using
the operator $X_h = M_0^{-1}(A_h + M_1)$ in~\eqref{eq:30},
\[
  (I - \tau X_h) v_\infty = v_0.
\]
Hence the Richardson iteration for solving this linear system for
$v_\infty$ takes the form
\begin{equation}
  \label{eq:42}
  v_{\ell +1} = v_\ell + \left(v_0 - (I - \tau X_h) v_\ell\right),
  \qquad \ell =0, 1, \ldots.   
\end{equation}

\begin{definition}[Lowest order explicit discrete flows:
  $\Rexp_{h1}(\tau)$ and $\Rexp_{h1q}(\tau)$]
  \label{def:lowest-order-explicit}
  Let $v_0 \in \Vhv$.  The result $v_1$ after one iteration
  of~\eqref{eq:42} defines the operator
  $\Rexp_{h1} (\tau): \Vhv \to \Vhv$:
  \[
    \Rexp_{h1} (\tau) v_0 = v_1 =  (I+ \tau X_h) v_0.
  \]
  The result $v_q$ obtained after
  performing   $q\ge 1$ iterations  defines
  $\Rexp_{h1q} (\tau): \Vhv \to \Vhv$ by 
  \[
    \Rexp_{h1q} (\tau)
    v_0 = v_q =  v_0 + \tau X_h v_{q-1}.
  \]
  Note that no matrix inversions are required for conducting these $q$
  iterations, except for one local mass matrix inversion ($M_0^{-1}$)
  per tent.
\end{definition}

Unlike the tent-implicit scheme, we are now able to obtain stability
for the explicit scheme 
only under further conditions. From Lemma~\ref{lem:AXRbounds}, we know
that $\| X_h \|_{M_0} \lesssim 1$. Hence the condition~\eqref{eq:Xhcond} in
the next result can be met by performing sufficiently many iterations.

\begin{lemma}[Conditional stability]
  \label{lem:lowest-order-explicit-q-stability}
  If $q$ is large enough to admit
  \begin{equation}
    \label{eq:Xhcond}
    \| X_h \|_{M_0} \lesssim
    \hv^{1/(q+1)},
  \end{equation}
  then there is a $c_q>0$ independent of $\hv$ such
  that for all $v_0 \in \Vhv$, 
  \[
    \|  \Rexp_{h1q} (\tau) v_0 \|_{M(\tau)} \le (1 + c_q \hv) \| v_0 \|_{M_0}.
  \]
\end{lemma}
\begin{proof}
  Recursively expanding $v_q = v_0 + \tau X_h v_{q-1}$, we obtain
  \[
    v_q = \sum_{j=0}^q (\tau X_h)^j v_0.
  \]
  Rewriting this, using (\ref{eq:30}), as
  \begin{equation}
    \label{eq:43}
    v_q = (I - \tau X_h)^{-1} \big[I - (\tau X_h)^{q+1}\big] v_0 =
    \Rimp_{h1} \big[1 - (\tau X_h)^{q+1}\big] v_0,     
  \end{equation}
  we apply 
  Lemma~\ref{lem:stability-imp}. Hence
  \begin{align*}
    \| v_q \|_M
    & \le
      \left\| v_0 - (\tau X_h)^{q+1} v_0 \right\|_{M_0}
      \le \| v_0 \|_{M_0} + \| X_h \|_{M_0}^{q+1} \| v_0\|_{M_0}
  \end{align*}
  and the result follows using~\eqref{eq:Xhcond}.
\end{proof}

\begin{lemma}[Local error bound]
  \label{lem:local-error-explicit}
  Let $\uh$ denote the exact solution on $\Tvh$, let
  $\uhatexp_{h1q}(\tau) = \Rexp_{h1q}(\tau) \uh_h^0$ for some
  $\uh_h^0 \in \Vhv$, and suppose~\eqref{eq:Xhcond} holds.
  Then, 
  \begin{align*}
    \| \uh (\tau) - \uhatexp_{h1q} (\tau) \|_{M(\tau)}
    & \;\lesssim\;
      \| \uh(0) - \uh_h^0 
      \|_{M(0)}
      +
      \hv
      \big(
      \|  u\|_{2, \infty, \vtx}
      + \| u \|_{\vtx, 1}
      \big).
  \end{align*}
\end{lemma}
\begin{proof}
  Let $e_h = \uhatimp_{h1} (\tau) - \uhatexp_{h1q}(\tau)$. Since
  $\uh (\tau) - \uhatexp_{h1q} (\tau) -e_h = \uh (\tau) - \uhatimp_{h1}
  (\tau)$ can be bounded by Lemma~\ref{lem:local-error-implicit}, it
  suffices to bound $e_h$. By~\eqref{eq:43},
  \begin{align*}
    e_h
    & = (I - \tau X_h)^{-1} \uh_h^0 - 
      (I - \tau X_h)^{-1}\big[ I - (\tau X_h)^{q+1}] \uh_h^0
    \\
    & = (I - \tau X_h)^{-1} (\tau X_h)^{q+1} \uh_h^0.
  \end{align*}
  Thus, by Lemma~\ref{lem:stability-imp} and~\eqref{eq:Xhcond},
  $\| e_h \|_M \lesssim \hv \| \uh_h^0\|_{M_0}$. We may further write
  $ \uh_h^0 $ as the sum of $\uh_h^0 - \uh(0)$ and
  $ \uh(0)$ and apply triangle inequality to obtain the right
  hand side of the stated bound.
\end{proof}

Letting $\RRexp_{h1q}$ denote the collection of explicit tent
propagator operators $\Rexp_{h1q}(1) \circ P_h : \Lv \to \Vhv$ on all
tents, we use Definition~\ref{def:T2G} to set the global propagators
$\TG(i, j, \RRexp_{h1q})$, and consider the discrete solution
$\uexp_{h1q} = \TG(m, 0, \RRexp_{h1q}) u^0$ at the final time $T$.

\begin{theorem}[Error estimate for iterated lowest order
  explicit scheme]
  \label{thm:lowest-explicit-q}  
  Suppose~\eqref{eq:Xhcond}, \eqref{eq:32}, and the conditions of
  Theorem~\ref{thm:semidiscrete} hold. Then for any spatial degree
  $p \ge 0$, the fully discrete explicit solution $\uexp_{h1q}$
  satisfies
  \[
    \| u (T)  - \uexp_{h1q} \|_\om^2
    \;\lesssim\;
    \sum_{j=1}^m \sum_{\vtx \in V_j} \hv
          \big(
      \|  u\|_{2, \infty, \vtx}
      + \| u \|_{\vtx, 1}
      \big)^2.
  \]
\end{theorem}
\begin{proof}
  The proof proceeds along the lines of the proof of
  Theorem~\ref{thm:semidiscrete}, replacing the applications of
  Lemmas~\ref{lem:stability_semidiscrete}
  and~\ref{lem:local_semidiscrete_error}, respectively, by those of
  Lemmas~\ref{lem:lowest-order-explicit-q-stability}
  and~\ref{lem:local-error-explicit} instead. The main difference is that
  we must now invoke the argument of Remark~\ref{rem:weaker_stab} due
  to the weaker stability estimate of
  Lemma~\ref{lem:lowest-order-explicit-q-stability}.
\end{proof}

The $O(h^{1/2})$ rate of convergence given by
Theorem~\ref{thm:lowest-explicit-q} is the same as the rate given by
Theorem~\ref{thm:lowestimplicit} for the lowest order tent-implicit
scheme. Increasing the iteration number~$q$ can improve stability but
does not generally improve the order of convergence.

\subsection{Arbitrary order  SAT schemes}
\label{ssec:arbitrary-order-sat}

Letting $\Xh 0$ denote the identity operator on $\Vhv$, recursively
define further operators on $\Vhv$ by 
\begin{equation}
  \label{eq:Xhrecurse}
  \Xh k
  = M_0^{-1} (A_h + k M_1) \Xh {k-1}, \quad k \ge 1.    
\end{equation}
Similarly, let $\X 0 = 1$ and $\X k = M_0^{-1} (A + k M_1) \X {k-1}$
for $ k \ge 1.$ By
Lemma~\ref{lem:exact-time-deriv}, the time derivative of the exact
solution satisfies $\uh^{(k)}(0) = \X k \uh(0)$. Hence the
expansion~\eqref{eq:u-TaylorExp} may be written as
\begin{equation}
  \label{eq:17}
  \uh(\tau) = \sum_{k=0}^{s} \frac{\tau^k}{k!} \X k \uh(0)
  + \rho_{s+1}(\tau),  
\end{equation}
This motivates us to define the SAT flow by replacing $\X k$ with
the discrete operator $\Xh k$
as follows. (A gentler derivation can be found
in~\cite{GopalHochsSchob20} and it can be seen easily that the
discrete flow defined there coincides with the one in the next
definition.)

\begin{definition}[Discrete $s$-stage SAT flow: $\Rsat_{hs}(\tau)$ for
  $s\ge 1$]
  \label{def:sSATflow-Xh}
  Define $\Rsat_{hs}(\tau) : \Vhv \to  \Vhv$ by
  \[
    \Rsat_{hs}(\tau) v
    =
    \sum_{k=0}^{s-1} \frac{\tau^k}{k!} \Xh k   v
    + \frac{\tau^s}{s!} M(\tau)^{-1} M_0 \Xh s v, \qquad v \in \Vhv.
  \]
\end{definition}

\begin{lemma}
  \label{lem:Rsat-ee-0}
  Let $u$ be the exact solution of~\eqref{eq:tenteqforu} on a causal
  tent $\Tv$ and $\uh = u \circ \vPhi \in C^{s+1}(0, 1, \Hv \cap
  H^{p+1}(\ovh)^L)$.
  Then for any $s\ge 1$, 
  \begin{align*}
    \| \Rsat_{hs}(\tau) \uh_h^0 - P_h \uh(\tau)\|_{M(\tau)}
    & \lesssim
      \| \uh_h^0 - P_h\uh (0)   \|_{M(0)}
      +
      \tau^s \hv^s \| u\|_{s+1, \infty, \vtx}
      \\
     &+
      \tau \hv^{p+1}| u |_{\vtx, p+1, s-1}.
  \end{align*}
\end{lemma}
\begin{proof}
  Let $\uh_{hs}(\tau) = \Rsat_{hs}(\tau)\uh_h^0 $ and let
  $\veps_k = (\Xh k P_h - P_h \X k) \uh(0)$. Then projecting and
  subtracting~\eqref{eq:17} from the expansion defining
  $\Rsat_{hs}(\tau)\uh_h^0$, we obtain
  \begin{align*}
    \uh_{hs}(\tau) - P_h \uh(\tau)
    & = \sum_{k=0}^{s-1} \frac{\tau^k}{k!} \big[
      \Xh k (\uh_h^0 - P_h \uh(0) )
      + \veps_k 
      \big]
      \\
     & +
      \frac{\tau^s}{s!}  M^{-1}M_0 \big[\Xh s(\uh_h^0 - P_h \uh(0))\big]
    \\
    & + \frac{\tau^s}{s!}
      \big[ (M^{-1}M_0 -I) P_h \uh^{(s)}(0)
      + M^{-1}M_0 \veps_s
      \big]
      - P_h \rho_{s+1}.
  \end{align*}
  Letting $\mu_s = (M^{-1}M_0 -I) P_h \uh^{(s)}(0)$, 
  and noting that $\veps_0 = 0$,
  \begin{equation}
    \label{eq:19}
    \uh_{hs}(\tau) - P_h \uh (\tau)=
    \Rsat_{hs}(\tau)\big[\uh_h^0 - P_h \uh(0) \big]
    + \sum_{k=1}^{s} \frac{\tau^k}{k!} \veps_k
    +
    \frac{\tau^s}{s!}  \big[
    \mu_s
    + M^{-1}M_0 \veps_s
    \big]
    - P_h \rho_{s+1}.    
  \end{equation}
  To estimate the terms on the right hand side, we first use 
  Lemma~\ref{lem:AXRbounds} to conclude that
  $
  \|\Rsat_{hs}(\tau)[\uh_h^0 - P_h \uh(0) ]\|_{M} \lesssim
  \| \uh_h^0 - P_h \uh(0) \|_{M_0}.  
  $
  By Lemma~\ref{lem:time-derivative-uhat},
  $
  \| \mu_s\|_\vtx \lesssim
  \hv^s \| \d_t^s u\|_{\infty, \vtx}.     
  $
  To bound $\veps_k$, note that
  \begin{equation}
    \label{eq:48}
    \veps_k = M_0^{-1} (A_h + k M_1) \veps_{k-1} + \eta_{k-1}
  \end{equation}
  where $\eta_j = M_0^{-1}(A_h P_h - P_h A) \uh^{(j)}(0)$.  By
  Lemma~\ref{lem:APcomm},
  $\| \eta_j\|_\vtx \lesssim \hv^{p+1} | u|_{\vtx, p+1, j}.$ Hence
  recursively bounding $ \| \veps_k\|_\vtx$ by
  $ \| \veps_{k-1}\|_\vtx$ using~\eqref{eq:48}, and noting that
  $\veps_0=0$, we have
  $ \| \veps_k \|_{M} \lesssim \hv^{p+1} | u|_{\vtx, p+1, k-1}.  $ The
  final term in~\eqref{eq:19} can be treated using
  Lemma~\ref{lem:rho-Taylor-remainder}, which yields
  $\| P_h\rho_{s+1}(\tau) \|_\vtx \lesssim \tau^{s+1}\hv^{s+1} \|
  \d_t^{s+1} u \|_{\infty, \vtx}.$ When these estimates are used to
  bound the terms in the right hand side of~\eqref{eq:19} (and noting
  that $\tau$ is a common factor in all terms except the first), we
  obtain the stated inequality.
\end{proof}

In order to improve the stability of these explicit SAT schemes, we
shall now divide each tent into $r$ subtents and apply the SAT scheme
in each subtent.

\begin{definition}[Subtents]
  \label{def:subtents}
  Subdivide a tent $\Tv$ into $r$ subtents as follows. For
  $\ell=1, \ldots, r$, 
  define the $\ell$th subtent by
  \[
    \Tvk = \{(x, t) : x \in \ov, \; \vphi^\kk(x) \le t \le
    \vphi^\kkk(x)\},
  \]
  where   $\hat t^\kk = (\ell-1)/r$,
  $\vphi^\kk = \vphi(x, \hat t^\kk).$ 
  Let $\delta^\kk = \vphi^\kkk - \vphi^\kk.$ Using
  $\delta^\kk$ in place of $\delta$ in \eqref{eq:6} and \eqref{eq:M1},
  we define $a^\kk(w, v)$ and $M_1^\kk$, respectively, and let
  $A_h^\kk: \Vhv \to \Vhv$ be defined by
  $(A_h^\kk w, v)_\vtx = a^\kk(w, v)$ for $w, v \in \Vhv$. Finally let
  $M_0^\kk$ be defined by \eqref{eq:M0} after replacing $\vpbot$ there
  by $\vphi^\kk$ and let  $M^\kk(\tau) = M_0^\kk - \tau M_1^\kk$.
  It is easy to see that
  $\delta^\kk = \delta /r$ and 
  \begin{equation}
    \label{eq:subtent-ids}
    A_h^\kk = \frac 1 r A_h,     \quad
    M_1^\kk = \frac 1 r M_1, \quad
    M^\kk(0) = M(\hat t^\kk), \quad 
    M^\kk(1) = M(\hat t^\kkk). 
  \end{equation}
\end{definition}

\begin{definition}[Discrete $s$-stage SAT propagator using $r$ subtents: $\Rsat_{rhs}$]
  \label{def:s-SAT}
  Define $\Xhl k$ by replacing $A_h, M_1,$ and $M_0,$ by
  \smash{$A_h^\kk, M_1^\kk,$ and $M_0^\kk,$} respectively,
  in~\eqref{eq:Xhrecurse}.  Define $\Rsat_{\kk, hs}(\tau)$ on a
  subtent $\Tvk$ by replacing $M_0, \Xh k$ and $M$ by
  $M_0^\kk, \Xhl k$ and $M^\kk$, respectively, in
  Definition~\ref{def:sSATflow-Xh}.  Applying on the $r$ subtents
  successively, we define
  $ \Rsat_{rhs} = \Rsat_{[r], hs}(1) \circ \Rsat_{[r-1],
    hs}(1) \cdots\circ \Rsat_{[1], hs}(1).  $
\end{definition}

Note that the constant in ``$\lesssim$'' will not be allowed to depend
on $r$ (so that we may  admit
examples with $\hv$-dependent  $r$), as
emphasized in the next lemma.

\begin{lemma}[Local error in a tent]
  \label{lem:local-err-sat}
  Let $u$ be the exact solution of~\eqref{eq:tenteqforu} on a causal
  tent $\Tv$, $\uh = u \circ \vPhi \in C^{s+1}(0, 1, \Hv \cap
  H^{p+1}(\ovh)^L)$,
  and let $\uh_{rhs} = \Rsat_{rhs} \uh_h^0.$
  Then there is a mesh-independent constant $c_{s, p}$ that is also 
  independent of $r$ such that 
  \begin{align*}
    c_{s, p}\| \uh_{rhs} - \uh (1)  \|_{M(1)}
    & \le 
      \| \uh_h^0 - P_h\uh (0)   \|_{M(0)}
      +     \hv^s \| u\|_{s+1, \infty, \vtx}
      + \hv^{p+1}| u |_{\vtx, p+1, s-1}.
  \end{align*}
\end{lemma}
\begin{proof}
  Denoting the discrete solutions by
  $\uh_{1, hs} = \Rsat_{[1], hs} \uh_h^0$ and
  $\uh_{\ell, hs} = \Rsat_{[\ell], hs} \uh_{\ell-1, hs}$ for $1\le
  \ell \le r$, we compare them with the subtent
  exact solutions, denoted by  
  $\uh_\ell = \uh(\hat t^\kkk)$.  
  In the pseudotime coordinate of the un-split tent $\Tv$, the value
  $\tau=1/r$ corresponds to the top of the first subtent, where the
  exact solution is $\uh_1 = \uh(1/r)$.
  Thus Lemma~\ref{lem:Rsat-ee-0} with $\tau=1/r$ gives 
  \[
    \| \uh_{1, hs} - P_h \uh_1\|_{M(1/r)}
    \lesssim \| \uh_h^0 - P_h \uh(0) \|_{M(0)} 
    +
    \frac 1 r \left(  \hv^s \| u\|_{s+1, \infty, \vtx}
      + \hv^{p+1}| u |_{\vtx, p+1, s-1}\right).
  \]
  Similarly, on the $\ell$th subtent, for $\ell=1, 2, \ldots, r$, 
  \begin{align*}
    \| \uh_{\ell, hs} - P_h \uh_\ell\|_{M(\hat t^\kkk)}
    & \lesssim \| \uh_{\ell-1, hs} - P_h \uh_{\ell-1}
    \|_{M(\hat t^\kk)} 
    \\
    &+
    \frac 1 r \left(  \hv^s \| u\|_{s+1, \infty, \vtx}
      + \hv^{p+1}| u |_{\vtx, p+1, s-1}\right).
  \end{align*} 
  Applying this estimate for $\ell=r, r-1, \ldots, 1$, successively in
  that order, where at each step the first term on the right hand side
  is bounded using the next estimate,
  \[
    \| \uh_{rhs} - P_h \uh(1)\|_{M(1)}
    \lesssim \| \uh_h^0 - P_h \uh(0) \|_{M(0)}  +
     \left(  \hv^s \| u\|_{s+1, \infty, \vtx}
      + \hv^{p+1}| u |_{\vtx, p+1, s-1}\right) \sum_{\ell=1}^r \frac 1 r
  \]
  which completes the proof.
\end{proof}

Letting $\RRsat_{rhs}$ denote the collection of explicit tent
propagators $\Rsat_{rhs} \circ P_h : \Lv \to \Vhv$ on all
tents, we use Definition~\ref{def:T2G} to set the global propagators
$\TG(i, j, \RRsat_{rhs})$, and consider the discrete solution
$\usat_{rhs} = \TG(m, 0, \RRsat_{rhs}) u^0$ at the final time $T$.

\begin{theorem}[Error estimate for the SAT scheme]
  \label{thm:sat-ee}  
  Assume that
  there is a mesh-independent $\Cstab \ge 0$ such that 
  \begin{equation}
    \label{eq:stability-sat-assume}
    \| \Rsat_{rhs} v \|_{M(1)} \le (1 + \Cstab \hv) \| v \|_{M(0)}
  \end{equation}
  for all $v \in \Vhv$ on all tents $\Tv$. Suppose also that \eqref{eq:32}
  and the conditions of Theorem~\ref{thm:semidiscrete} hold. Then the
  fully discrete explicit $s$-stage SAT solution $\usat_{rhs}$,
  obtained using spatial polynomial degree $p$, satisfies
  \[
    \| u (T)  - \usat_{rhs} \|_\om^2
    \;\lesssim\; 
    \sum_{j=1}^m \sum_{\vtx \in V_j} 
    \hv^{2s-1} \|  u\|_{s+1, \infty, \vtx}^2
    + \hv^{2p+1} \| u \|_{\vtx, p+1, s-1}^2.
  \]
\end{theorem}
\begin{proof}
  The proof proceeds along the lines of the extension of the proof of
  Theorem~\ref{thm:semidiscrete} mentioned in
  Remark~\ref{rem:weaker_stab}, replacing the application of
  Lemma~\ref{lem:stability_semidiscrete}
  by~\eqref{eq:stability-sat-assume}, and replacing the application of
  Lemma~\ref{lem:local_semidiscrete_error} by
  Lemma~\ref{lem:local-err-sat}.
\end{proof}

Theorem~\ref{thm:sat-ee} bounds the error by terms that converge to
zero at the same rate, provided the number of stages in the SAT scheme
is tied to the spatial degree by $s= p+1$. Then, the convergence rate
given by Theorem~\ref{thm:sat-ee} is $O(h^{p+1/2})$, the same rate we
obtained for the semidiscretization (in
Theorem~\ref{thm:semidiscrete}). In \S\ref{sssec:ratenotimprovable} we
show, through a numerical example,
that this rate is generally un-improvable.

One can solve a local eigenproblem on a tent to computationally check
if the stability assumption~\eqref{eq:stability-sat-assume} is
satisfied. Since this eigenvalue computation is described in detail
in~\cite[\S6.1]{GopalSchobWinte20}, we shall not comment further on
this computational avenue for stability verification. 
In \S\ref{sssec:p=0} and \S\ref{ssec:stab-s=2}, we 
describe two cases where stability can be proved staying within the
framework of the general symmetric linear hyperbolic systems we have
been considering.

\subsubsection{Numerical observations on the convergence rate}
\label{sssec:ratenotimprovable}

It is natural to wonder if the convergence rate of $O(h^{p+1/2})$,
given by Theorem~\ref{thm:sat-ee} (and 
Theorem~\ref{thm:semidiscrete}), is
improvable.  Our numerical experience from computations with various
hyperbolic systems suggests that one is likely to observe a higher
convergence rate of $O(h^{p+1})$ on generic examples and meshes.
Yet, as we show
now, there is at least one family of tent meshes in the $N=2$ case
where  $O(h^{p+1/2})$ rate of
convergence is observed.
Such  tent meshes are created by selecting the spatial mesh
$\oh$  from the mesh families described in~\cite{Peterson91},
where it is shown that the standard $O(h^{p+1/2})$ error estimate for
the DG method for {\em stationary} advection equation cannot be
improved. Building causal tents atop such a mesh, we show that our
$O(h^{p+1/2})$ estimate for the time-dependent advection problem also
cannot be improved.

The structured spatial meshes we borrow from~\cite{Peterson91} consist
of horizontal layers of right triangles grouped in vertical bands.  As
the mesh is refined, the number of vertical bands is controlled by a
parameter $\sigma \in [0, 1]$.  We used the MTP discretization with
polynomial orders $p$ varying from $0$ to $3$, together with SAT time
stepping with $r=\max \{1, 2p\}$ and $s=p+1$
to solve the advection problem of
Example~\ref{eg:advection} (modified to take a nonhomogeneous inflow
boundary condition).  The domain $\om $ is set to the unit square, the
advective vector field $b$ is set to the constant vector $b=[0,1]^t$,
so that $\din \om = \{ (x_1, x_2): 0 \le x_2 \le 1\},$ and the inflow
boundary condition is set by $u = x_1^{p+1}$ on $\din \om$. The
initial condition is $u^0(x_1, x_2) \equiv x_1^{p+1}$. At $t = T=1$, the MTP
solution approximated the exact solution $u(x_1, x_2, t) = x_1^{p+1}$
at all spatial points $(x_1,x_2) \in \om.$

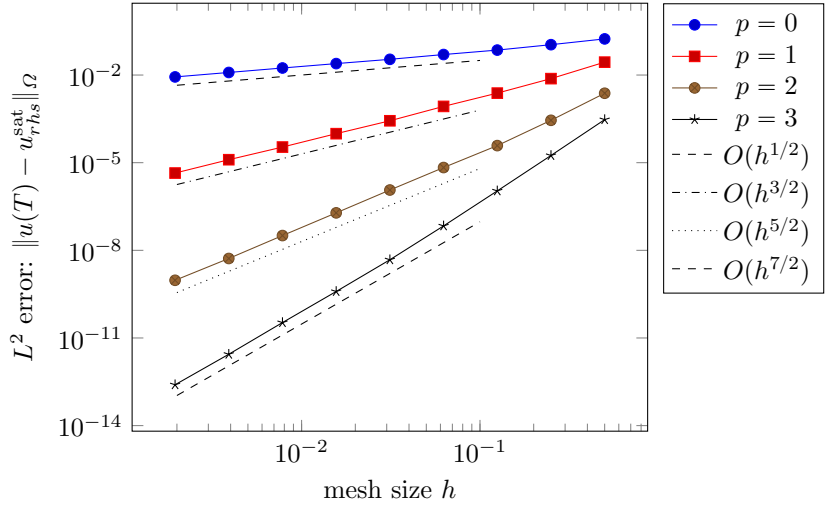
\begin{figure}[htpb]
  \centering
  \begin{tikzpicture}
  \begin{loglogaxis}[
    mark repeat={1},
    ylabel={$L^2$ error: $\| u(T) - \usat_{rhs} \|_\om$},
    xlabel={mesh size $h$},
    legend
    entries={$p=0$, $p=1$, $p=2$, $p=3$, 
             $O(h^{1/2})$, $O(h^{3/2})$, $O(h^{5/2})$,
             $O(h^{7/2})$},  
    legend pos=outer north east,
  ]
  \foreach \column in {0,...,3}{
    \addplot+[] table[x={h},y={l2e_p\column}]
    {peterson_mesh/table_convect2dconv_SAT.txt};
  }
  \addplot[domain=2/1000:1/10, samples=3, dashed]{.1*x^.5};
  \addplot[domain=2/1000:1/10, samples=3, dashdotted]{.02*x^1.5};
  \addplot[domain=2/1000:1/10, samples=3, dotted]{.002*x^2.5};
  \addplot[domain=2/1000:1/10, samples=3, dashed]{.0003*x^3.5};
  \end{loglogaxis}
  \end{tikzpicture}
  \caption{Convergence rates observed when solving the
    advection problem}
  \label{fig:satcv}
\end{figure}

We obtained different convergence rates for different choices of
$\sigma$, but in all cases, the rates are bounded between
$O(h^{p+1/2})$ and $O(h^{p+1})$.  We obtained the minimal convergence
rate (largest errors) when $\sigma = 3/4$ and $\sigma = 1/2$ for
$p \ge 1$ and $p = 0$, respectively.  The errors and rates observed
for these values of $\sigma$ are plotted in Figure~\ref{fig:satcv},
which clearly show $O(h^{p+1/2})$ rate of convergence. We note that
our rate-minimizing $\sigma$-value of $3/4$ is the same value of
$\sigma$ used in~\cite{Peterson91} for the $p=1$ case (the only case
where numerical results are given there).

\subsubsection{Stability verification in the  $p=0, \;s=1$ case}
\label{sssec:p=0}

This case is motivated by the many studies of the $p=0$ case in the DG
literature (see e.g.~\cite{BurmaErnFerna10, Despr04, SunShu17}), often
called the finite volume case, and is illustrative of why special cases
are worth pursuing. We focus on the operator of the SAT
scheme, obtained by setting $s=1$ in Definition~\ref{def:sSATflow-Xh}, 
which can be simplified to
\[
  \Rsat_{h1}(\tau) = I + \tau M(\tau)^{-1}(A_h + M_1),
\]
and the corresponding operator $\Rsat_{rh1}$ obtained using $r$
subtents, per Defintion~\ref{def:s-SAT}.  Note that $\Rsat_{h1}(\tau)$
differs slightly from the $q=1$ case of
Definition~\ref{def:lowest-order-explicit}, namely,
$\Rexp_{h1}(\tau) = I + \tau M_0^{-1}(A_h + M_1)$.  While
$\Rexp_{h1q}$ only requires one local mass matrix inversion for $q$
iterations within a tent, the application of $\Rsat_{rh1}$ requires
one local inversion per subtent. However, $\Rsat_{rh1}$ admits a
stronger stability estimate that we shall prove after making the
following observation.

\begin{lemma}
  \label{lem:p0AX}
  When $p=0$, we have, for all $v, w \in \Vhv$,
  \begin{align}
    \label{eq:45}
    (A_h w, v)_\vtx & \lesssim \| w \|_\vtx |v|_d,
    \\
    \label{eq:46}
    \| (A_h + M_1) v \|_{\vtx} & \lesssim |v |_d. 
  \end{align}
\end{lemma}
\begin{proof}
  When $p=0$, the derivative terms in~\eqref{eq:6} vanish, so
  \begin{align*}
    (A_h w, v)_\vtx
    & = \sum_{K \in \ovh}
      -(\delta  \Fh_w,  v)_{\d K}
    \\
    & =
      - \sum_{F \in \Fcvb} \frac 1 2 (\delta (\Dcn + B) w, v)_F
      \\
      &+
      \sum_{F \in \Fcvi}
      (\delta \DcnF \{ w \}, \jmp{v}_F)_F
      - (\delta S \jmp{w}_F, \jmp{v}_F)_F
  \end{align*}
  where we have rearranged  the sum to run over the
  mesh facets. Now, by Cauchy-Schwarz inequality,
  \eqref{eq:DGdesign}, and Lemma~\ref{lem:intgparts}, the estimate
  of~\eqref{eq:45} follows.

  Of course, \eqref{eq:45} can also be written as
  $(w, A_h^t v)_\vtx  \lesssim \| w \|_\vtx |v|_d$ where $A_h^t$ is
  the $\Lv$-adjoint of $A_h$.  When this is added to the
  obvious inequality
  \[
   (-(A_h + A_h^t + M_1)v, w )_\vtx \le
    (-(2A_h + M_1)v, v)_\vtx^{1/2}(-(2A_h + M_1)w, w)_\vtx^{1/2}
    =
    |v|_d |w|_d,
  \]
  we obtain $(- (A_h + M_1) v, w)_\vtx \lesssim |v|_d \|w\|_\vtx$, so
  \[
    \| (A_h + M_1) v \|_{\vtx} =
    \sup_{0 \ne w \in \Vhv}
    \frac{ ( (A_h + M_1) v, w)_\vtx}{ \| w \|_\vtx}
    \lesssim |v|_d
  \]
  proves~\eqref{eq:46}.
\end{proof}

Let $K_h$ denote the kernel of $A_h + A_h^t + M_1 : \Vhv \to \Vhv$ and
let $K_h^\perp$ denote its $\Lv$-orthogonal complement in $\Vhv$. Set
\begin{equation}
  \label{eq:kappa}
  \kappa =
  \sup_{0 \le \tau \le 1}
  \,\sup_{\;0 \ne v \in K_h^\perp}
  \frac{ \| M(\tau)^{-1}(A_h + M_1) v\|_{M(\tau)}^2 } { |v|_d^2}.
\end{equation}

\begin{proposition}[Conditional strong stability]
  \label{lem:p0stability-r}
  In the case $p=0$ and $s=1$, the constant $\kappa$
  of~\eqref{eq:kappa} satisfies $\kappa \lesssim 1$.
  For all
  \begin{equation}
    \label{eq:49-cfl}
    0 \le \tau \le 1/\kappa
  \end{equation}
  and all $v \in \Vhv$, we have
  \begin{equation}
    \label{eq:48-srk}
    \|\Rsat_{h1}(\tau) v \|_{M(\tau)}
    \le \|v \|_{M_0}.
  \end{equation}
  Furthermore, if $r \ge \kappa$ subtents are used, then 
  \begin{equation}
    \label{eq:47}
    \|\Rsat_{rh1} v \|_{M(\tau)}
    \le \|v \|_{M_0},
  \end{equation}
  i.e., the stability assumption~\eqref{eq:stability-sat-assume} of
  Theorem~\ref{thm:sat-ee} holds with $\Cstab=0$.
\end{proposition}
\begin{proof}
  Since $\|M(\tau)^{-1} \|_{\vtx, h} \lesssim 1$ by
  Lemma~\ref{lem:AXRbounds}, the estimate~\eqref{eq:46} of
  Lemma~\ref{lem:p0AX} shows that $\kappa \lesssim 1$ whenever $p=0$.

  Let $v_\tau = \Rsat_{h1}(\tau) v = v + \tau M^{-1} (A_h + M_1)v.$
  Then expanding $\| v_\tau \|_M^2$, 
  \begin{align*}
    \| v_\tau \|_M^2
    & =  \| v \|_{M}^2 + 2\tau( (A_h + M_1) v, v)_\vtx
      + \tau^2 \| M^{-1} (A_h + M_1) v \|_{M}^2
    \\
    & =  \| v \|_{M_0}^2 + \tau( (2A_h + M_1) v, v)_\vtx
      + \tau^2 \| M^{-1} (A_h + M_1) v \|_{M}^2
    \\
    & \le  \| v \|_{M_0}^2 - \tau |v|_d^2
    + \tau^2 \kappa |v|_d^2
  \end{align*}
  so~\eqref{eq:48-srk} follows when $1- \tau \kappa\ge 0$.

  Next, consider a subtent $\Tvk$. By~\eqref{eq:subtent-ids},
  $\Rsat_{\kk, h1}(1) = I + M^\kk(1)^{-1} (A_h^\kk + M_1^\kk) = I +
  r^{-1} M(\hat t^\kkk) (A_h + M_1)$, so translating~\eqref{eq:48-srk}
  with $\tau = 1/r \le 1/\kappa$ to this subtent, we obtain
  $\|\Rsat_{\kk, h1}(1) v \|_{M(t^\kkk)} \le \| v
  \|_{M(t^\kk)}$. Successively applying these estimates over all
  subtents,~\eqref{eq:47} is proved.
\end{proof}

Note that the inverse of $\kappa$ appearing in~\eqref{eq:49-cfl} will
stay away from zero (since $\kappa \lesssim 1$) allowing for a
nontrivial advance in $\tau$. One can view~\eqref{eq:49-cfl} as the
analogue of a traditional ``CFL condition'' {\em within} a tent.
Indeed, the pseudotime restriction~\eqref{eq:49-cfl} may be interpreted
as a restriction on time advance in the physical spacetime by a small
subtent whose tent pole height is a scalar multiple of $\hv$.
Even in the event~\eqref{eq:49-cfl}
forbids us to reach the tent-top pseudotime
(i.e., when $\tau=1$ does not satisfy~\eqref{eq:49-cfl}), splitting
the tent into smaller subtents does allow
the analogue of~\eqref{eq:49-cfl} to hold throughout every subtent.

\subsubsection{Stability verification in the $s=2$ case}
\label{ssec:stab-s=2}

We will now show how to prove stability under a stronger CFL condition
in the two-stage case.  Definition~\ref{def:sSATflow-Xh} with $s=2$ yields
\[
  \Rsat_{h2}(\tau) = I + \tau \Xh 1 + \frac{\tau^2}{2} M^{-1}M_0
  {\Xh 2}.
\]

\begin{lemma}
  \label{lem:SAT2id}
  For any $v \in \Vhv$,
  \begin{align*}
    \| \Rsat_{h2}(\tau) v \|_{M(\tau)}^2
    & = \| v\|_{M_0}^2
      -\tau \left| v + \frac \tau 2 \Xh 1 v\right|_d^2
      + \tau^3\, Z(\tau, v),
  \end{align*}
  where  $Z(\tau, v) = 
    [
      2(M_1 \Xh 1v,
      \Xh 1v)_\vtx -|\Xh 1 v |_d^2 
      + \tau \| M(\tau)^{-1} M_0 {\Xh 2}v \|_M^2
    ]/4.$
\end{lemma}
\begin{proof}
  Let $w = \tau \Xh 1 v+ (\tau^2/2) M^{-1}M_0 {\Xh 2}v.$ Then
  $v_\tau = \Rsat_{h2}(\tau) v$ can be written as 
  $v_\tau = v + w$. Expanding $\| v + w\|_M^2$,
  \begin{align*}
    \| v_\tau \|_M^2
    & = \| v\|_{M_0}^2 - \tau (M_1 v,v )_\vtx
      + 2 (w, v)_M + \| w \|_M^2 
    \\
    & = \| v\|_{M_0}^2 - \tau (M_1 v,v )_\vtx
      + 2 \tau ((M_0-\tau M_1) \Xh 1 v, v)_\vtx + \tau^2 (M_0{\Xh 2} v, v)_\vtx
      +  \| w \|_M^2 
    \\
    & = \| v\|_{M_0}^2 + \tau ((2A_h + M_1) v,v )_\vtx
      + \tau^2 (A_h  \Xh 1  v, v)_\vtx 
      + \| w \|_M^2.
  \end{align*}
  Note that $(A_h  \Xh 1  v, v)_\vtx = ( \Xh 1 v, A_h^t v)_\vtx
  = ( \Xh 1 v, (A_h^t + A_h + M_1) v)_\vtx - \|  \Xh 1  v
  \|_{M_0}^2$. Since 
  $d(y, z) = -((A_h^t + A_h + M_1)y, z)_\vtx$, we have
  \begin{align*}
    \| v_\tau\|_M^2
    &
      =  \| v \|_{M_0}^2 - \tau | v|_d^2 -\tau^2 d( \Xh 1  v, v) - \tau^2
      \|  \Xh 1  v\|_{M_0}^2 + \|w \|_M^2.
  \end{align*}
  Next, letting $z = (1/2) M^{-1} M_0 {\Xh 2} v$,
  expanding the last term above
  $ \|w \|_M^2 = \tau^2\|  \Xh 1  v \|_M^2 + \tau^3( \Xh 1 v, M_0 {\Xh 2} v)_\vtx
  + \tau^4 \| z \|_M^2,$
  noting that
  $M_0{\Xh 2} = (A_h + 2 M_1) \Xh 1 $, and simplifying,
  \begin{align*}
    \| v_\tau\|_M^2
    &
      =  \| v \|_{M_0}^2 - \tau | v|_d^2 -\tau^2 d( \Xh 1  v, v)
      + \tau^3((A_h + M_1) \Xh 1  v,  \Xh 1 v)_\vtx
      + \tau^4 \| z \|_M^2
    \\
    & =  \| v \|_{M_0}^2
      - \tau \big| v + \frac \tau 2  \Xh 1  v\big|_d^2
      + \frac{\tau^3}{4}((2A_h + 3M_1)  \Xh 1  v,  \Xh 1 v)_\vtx
      + \tau^4 \| z \|_M^2
  \end{align*}
  from which the stated identity follows.
\end{proof}

\begin{proposition}
  \label{prop:sat_s2}
  Let $v \in \Vhv$ and $r$ be chosen as the smallest integer not
  smaller than $\kappa_2^{1/3}/ \hv^{1/2}$ where $\kappa_2$ is defined
  using $Z(\tau, v)$ of Lemma~\ref{lem:SAT2id} by
  \[
    \kappa_2 = \sup_{0 \le \tau \le 1} \; 
    \sup_{0 \ne v \in \Vhv}
    \frac{Z(\tau, v)}{\| v \|_{M_0}^2}.
  \]
  Then 
  \[
    \| \Rsat_{h2}(\tau) v \|_{M(\tau)}
    \le (1 + \hv^{3/2})^{1/2}  \| v\|_{M_0}
    \quad \text{ for all } \tau \le 1/r,
  \]
  and $\Rsat_{rh2}$ satisfies
  the stability assumption~\eqref{eq:stability-sat-assume} of
  Theorem~\ref{thm:sat-ee}.
\end{proposition}
\begin{proof}
  By Lemma~\ref{lem:SAT2id} and the definition of $\kappa_2$,
  \[
    \| \Rsat_{h2}(\tau) v \|_{M(\tau)}^2
    \le  \| v\|_{M_0}^2
    + \tau^3 \kappa_2 \| v\|_{M_0}^2
    \le (1 + \hv^{3/2}) \| v \|_{M_0}^2,
  \]
  since $\tau^3 \kappa_2 = \kappa_2/r^3 \le \hv^{3/2}.$
  Applying this successively on each subtent, we obtain
  \[
    \| \Rsat_{rh2} v \|_{M(1)}
    \le (1+ \hv^{3/2})^{r/2} \| v \|_{M_0}.
  \]
  Next, we use the bound
  $(1+ \hv^{3/2})^{r/2}\le \exp( \hv^{3/2} r / 2)$. Since the 
  argument of the exponential is bounded, 
  $ \exp( h^{3/2} r / 2) -1 \lesssim \hv^{3/2} r / 2 \lesssim \hv$.
  Thus there is an $\hv$-independent constant $C>0$ such that
  $\| \Rsat_{rh2} v \|_{M(1)} \le (1+ C\hv) \| v \|_{M_0}.$
\end{proof}

Note that by Lemma~\ref{lem:AXRbounds}, the constant $\kappa_2$
satisfies $\kappa_2 \lesssim 1$. Hence Proposition~\ref{prop:sat_s2}
gives stability under a so-called ``3/2-CFL'' condition. The latter term
is an adaptation of the terminology
on CFL conditions in~\cite{BurmaErnFerna10} for our tents,
in view of the fact
that our 
$\tau \le 1/r$ condition,
with $r$ as in Proposition~\ref{prop:sat_s2}, implies
that the amount of time advance
along a tent pole
($\tau\delta$)
is limited by $O(\hv^{3/2})$.

\section{Conclusion}

We have developed a convergence theory for MTP schemes for a large
class of linear hyperbolic systems, covering the semidiscrete case
(\S\ref{sec:semidisc}), as well as a few fully discrete schemes
(\S\ref{sec:fully-discr}). The convergence rate for the
semidiscretization was established to be $O(h^{p+1/2})$ in
Theorem~\ref{thm:semidiscrete} under reasonable assumptions.  When the
number of stages $s = p+1$, the fully discrete SAT scheme also gave
the same convergence rate (Theorem~\ref{thm:sat-ee}) under the
stability assumption~\eqref{eq:stability-sat-assume}. Through a
selected numerical example, we showed in
\S\ref{sssec:ratenotimprovable} that this convergence rate cannot be
improved in general.  The stability of SAT scheme was verified in
\S\ref{sssec:p=0} for the $p=0,$ $s=1$ case and in
\S\ref{ssec:stab-s=2} for the (arbitrary $p$) $s=2$ case.  Proving the
stability of SAT schemes (verifying~\eqref{eq:stability-sat-assume})
for other values of $s$ is currently an open problem.  It is however
possible to computationally verify stability within each tent by
solving a small eigenvalue problem as shown
in~\cite{GopalSchobWinte20}.  The numerical results there suggest that an
estimate of the form
$\| \Rsat_{rhs} v\|_M \le (1 + C r^{-s}) \| v \|_{M_0}$ might hold for
general~$r$ and $s$. If this is provable, then for larger $s$, a
slight modification of the argument of Proposition~\ref{prop:sat_s2}
would prove stability under a less stringent $(1 + 1/s)$-CFL
condition, which limits the amount of time advance by a
scalar multiple of $\hv^{1+1/s}$. Also, if our analysis
in~\S\ref{sssec:p=0} is any indication, it might be a worthwhile
future pursuit to seek further special cases where stability holds
under even weaker CFL conditions within a tent. The simplest cases of
the fully discrete analyses we presented are those of the lowest order
tent-implicit scheme in \S\ref{ssec:lowest-order-tent-implicit} and the
lowest order iterated explicit scheme in
\S\ref{ssec:lowest-order-expl}. The latter was obtained from a
nontraditional viewpoint of explicit schemes as iterative solvers for
implicit schemes.

\bibliographystyle{amsplain}

\end{document}